\documentclass[11pt]{amsart}

\usepackage{amsmath,
amssymb,
amscd,
amsfonts,
color,
enumitem,
epsfig,
float,
graphicx,
hyperref,
latexsym,
mathrsfs,
mathtools,
psfrag,
transparent,
url,
wasysym,
wrapfig,
tikz,
tikz-cd,
xr
}
\usepackage{fullpage}
\usepackage[all]{xy}
\usepackage[T1]{fontenc}
\pdfmapfile{-mpfonts.map}
\let\oldtocsection=\tocsection

\let\oldtocsubsection=\tocsubsection

\let\oldtocsubsubsection=\tocsubsubsection

\makeatletter
\newcommand{\specialcell}[1]{\ifmeasuring@#1\else\omit$\displaystyle#1$\ignorespaces\fi}
\makeatother

\renewcommand{\tocsection}[2]{\hspace{0em}\oldtocsection{#1}{#2}}
\renewcommand{\tocsubsection}[2]{\hspace{2em}\oldtocsubsection{#1}{#2}}
\renewcommand{\tocsubsubsection}[2]{\hspace{4em}\oldtocsubsubsection{#1}{#2}}

\newtheorem{theorem}{Theorem}[section]
\newtheorem{corollary}[theorem]{Corollary}
\newtheorem{question}[theorem]{Question}
\newtheorem{conjecture}[theorem]{Conjecture}
\newtheorem{proposition}[theorem]{Proposition}
\newtheorem{lemma}[theorem]{Lemma}
\newtheorem{expectation}[theorem]{Expectation}

\theoremstyle{definition}
\newtheorem{definition}[theorem]{Definition}
\newtheorem{example}[theorem]{Example}
\theoremstyle{remark}
\newtheorem{remark}[theorem]{Remark}

\theoremstyle{plain}
\newcommand{\thistheoremname}{}
\newtheorem{genericthm}[theorem]{\thistheoremname}

\newtheorem*{genericthm*}{\thistheoremname}
\newenvironment{namedthm*}[1]
  {\renewcommand{\thistheoremname}{#1}%
   \begin{genericthm*}}
  {\end{genericthm*}}
  
\newcommand\cA{\mathcal{A}}
\newcommand\cB{\mathcal{B}}
\newcommand\cC{\mathcal{C}}
\newcommand\cD{\mathcal{D}}
\newcommand\cE{\mathcal{E}}

\newcommand\cI{\mathcal{I}}
\newcommand\cJ{\mathcal{J}}

\newcommand\cL{\mathcal{L}}
\newcommand\cM{\mathcal{M}}

\newcommand\cO{\mathcal{O}}

\newcommand\cS{\mathcal{S}}

\newcommand{\bC}{\mathbb{C}}
\newcommand{\bF}{\mathbb{F}}
\newcommand{\bH}{\mathbb{H}}
\newcommand{\bK}{\mathbb{K}}

\newcommand{\bR}{\mathbb{R}}
\newcommand{\bZ}{\mathbb{Z}}
\newcommand{\bRP}{\mathbb{RP}}
\newcommand{\bCP}{\mathbb{CP}}
\newcommand\bk{\mathbf{k}}
\newcommand\bm{\mathbf{m}}
\newcommand\bn{\mathbf{n}}

\newcommand\bzero{\mathbf{0}}

\newcommand{\btB}{{\mathbf{2B}}}
\newcommand{\sB}{\mathscr{B}}
\newcommand{\stB}{2\mathscr{B}}

\newcommand{\sA}{\mathscr{A}}
\newcommand{\sC}{\mathscr{C}}
\newcommand{\sE}{\mathscr{E}}
\newcommand{\sF}{\mathscr{F}}
\newcommand{\sM}{\mathscr{M}}
\newcommand{\sP}{\mathscr{P}}
\newcommand{\sQ}{\mathscr{Q}}
\newcommand{\sR}{\mathscr{R}}
\newcommand{\sY}{\mathscr{Y}}

\newcommand{\fX}{\mathfrak{X}}

\newcommand{\fg}{\mathfrak{g}}

\newcommand{\ro}{\mathrm{o}}
\newcommand{\on}{\operatorname}

\newcommand\pt{\on{pt}}

\newcommand\dbar{\ol\partial}
\newcommand\id{\on{id}}

\newcommand\loc{{\on{loc}}}

\newcommand{\Fuk}{\on{Fuk}}

\newcommand{\Symp}{\mathsf{Symp}}

\newcommand{\Ch}{\textsf{Ch}}
\newcommand{\Ab}{\textsf{Ab}}

\newcommand{\CO}{{CO}}
\newcommand{\OC}{{OC}}
\newcommand{\CC}{{CC}}

\newcommand{\SC}{{SC}}

\newcommand{\Id}{\on{Id}}

\newcommand{\Top}{\textsf{Top}}
\newcommand{\Set}{\textsf{Set}}

\newcommand{\tree}{{\on{tree}}}
\newcommand{\Ob}{\on{Ob}}
\newcommand{\Mor}{\on{Mor}}
\newcommand{\Hom}{\on{Hom}} 
\newcommand{\br}{{\on{br}}}
\newcommand{\cell}{{\on{cell}}}
\newcommand{\sing}{{\on{sing}}}

\newcommand{\Cl}{{\on{Cl}}}
\newcommand{\rk}{{\on{rk}\:}}

\newcommand{\Aut}{{\on{Aut}}}
\newcommand{\eq}{{\on{eq}}}
\newcommand{\perf}{{\on{perf}}}
\renewcommand{\mod}{{\on{mod}}}
\newcommand{\ind}{{\on{ind}}\,}
\newcommand{\poly}{{\on{poly}}}
\newcommand{\cone}{{\on{cone}}}

\newcommand{\ol}{\overline}
\newcommand{\ul}{\underline}

\newcommand{\sr}{\stackrel}
\newcommand{\wh}{\widehat}
\newcommand{\wt}{\widetilde}

\newcommand{\eps}{\epsilon}

\def\rd{{\rm d}}

\def\bi{\mathbf{i}}

\def\hra{\hookrightarrow}
\def\lra{\longrightarrow}

\begin{document}
\sloppy

\title{Functoriality in categorical symplectic geometry}
\author{Mohammed Abouzaid}
\address{Department of Mathematics, Columbia University,
2990 Broadway, New York, NY 10027, USA}
\email{\href{mailto:abouzaid@math.columbia.edu}{abouzaid@math.columbia.edu}}
\author{Nathaniel Bottman}
\address{Max Planck Institute for Mathematics,
Vivatsgasse 7, 53111 Bonn, Germany}
\email{\href{mailto:bottman@mpim-bonn.mpg.de}{bottman@mpim-bonn.mpg.de}}
%\date{\today}

\maketitle
\begin{abstract}
Categorical symplectic geometry is the study of a rich collection of invariants of symplectic manifolds, including the Fukaya $A_\infty$-category, Floer cohomology, and symplectic cohomology.
Beginning with work of Wehrheim and Woodward in the late 2000s, several authors have developed techniques for functorial manipulation of these invariants.
We survey these functorial structures, including Wehrheim--Woodward's quilted Floer cohomology and functors associated to Lagrangian correspondences, Fukaya's alternate approach to defining functors between Fukaya $A_\infty$-categories, and the second author's ongoing construction of the symplectic $(A_\infty,2)$-category.
In the last section, we describe a number of direct and indirect applications of this circle of ideas, and propose a conjectural version of the Barr--Beck Monadicity Criterion in the context of the Fukaya $A_\infty$-category.
\end{abstract}

{\small
\setcounter{tocdepth}{1}
\tableofcontents
}

\section{Introduction}
\label{s:intro}

A \emph{symplectic manifold} $(M,\omega)$ is a smooth even-dimensional manifold $M^{2n}$, together with a 2-form $\omega \in \Omega^2(M;\bR)$ that is closed ($\rd\omega = 0$) and non-degenerate in the sense that its top exterior power is a volume form ($\omega^{\wedge n} \neq 0$ pointwise).
The original motivation for this definition came from celestial mechanics, but much of modern symplectic geometry is independent of these physical origins.

\begin{example}
  The fundamental example of a symplectic manifold is Euclidean space with the \emph{Darboux symplectic form}:
\begin{equation}
\Bigl(
\bR^{2n},
\omega_0
\coloneqq
\sum_{i=1}^n
\,\rd p_i \wedge	 \,\rd q_i
\Bigr),
\end{equation}
where $\bR^{2n}$ is equipped with coordinates $(q_1, \ldots, q_n, p_1, \ldots, p_n)$.
  This choice of notation goes back to classical mechanics, where the coordinates $q_i$ record the position of a particle, and $p_i$ its momentum.
  From the point of view of a mathematician, the $q_i$ might as well represent local coordinates on a smooth manifold, in which case the coordinates $p_i$ can be understood as coordinates on the cotangent fibre.
  In this way, one obtains the \emph{canonical symplectic form} 
  \begin{equation}
    \Bigl(
T^*Q,
\omega_{\mathrm{can}}
\coloneqq
\sum_{i=1}^n \,\rd p_i \wedge \,\rd q_i
\Bigr),
  \end{equation}
  on the total space of the cotangent bundle $T^*Q$ of any smooth manifold.

Starting with the Darboux symplectic form, one constructs a large class of examples as follows: identify $\bR^{2n}$ with complex affine space $\bC^{n}$, by setting $p_i = \mathrm{Re}(z_i)$ and $q_i = \mathrm{Im}(z_i)$, and observe that the symplectic form is given, in terms of the $\partial$ and $\dbar$ operators of complex analysis ($\partial f = \sum \frac{\partial f}{\partial z_i} dz_i$ and $ \dbar f = \sum \frac{\partial f}{\partial \bar{z}_i} d\bar{z}_i$) as
\begin{equation}
\omega_0
\coloneqq
\frac\bi2\partial\dbar\bigl(|z|^2\bigr),
\end{equation}
with $|z|^2 = \sum |z_i|^2$ (this amount to the statement that the Darboux form is the real part of the standard K\"ahler form).
Since the norm of a vector is invariant under rotation, one obtains an induced symplectic form on the quotient $\bCP^{n-1}$ of the unit sphere $S^{2n-1}$ by the circle action.
This symplectic form on projective space is known as the \emph{Fubini-Study form} $\omega_{\text{FS}}$, and may be expressed directly in terms of coordinates on a standard affine chart of projective space as
\begin{equation}
\frac\bi2\partial\dbar\log\bigl(|z|^2\bigr)
=
\frac\bi{2|z|^4}\sum_{j,k=0}^{n-1}
\bigl(
|z_j|^2 \,\rd z_k\wedge\,\rd\ol z_k
-
\ol z_jz_k\,\rd z_j\wedge\,\rd\ol z_k
\bigr).
\end{equation}
Via the complex geometry result that K\"ahlerness is preserved by restriction to complex submanifolds, one then obtains from complex submanifolds of projective space (i.e.\ projective algebraic varieties) a large class of compact symplectic manifolds.
\null\hfill$\triangle$  
\end{example}

One of the fundamental questions in symplectic geometry is to understand the geometry of the \emph{Lagrangian submanifolds} (or simply \emph{Lagrangians}), i.e.\ those embedded submanifolds $L \subset M$ along which the symplectic form vanishes.
(In this paper, we will assume that all Lagrangians are oriented.)

\begin{example}
The fundamental examples of Lagrangians are the $\ul p$- and $\ul q$- planes in $\bR^{2n}$, equipped with the standard symplectic form.
This naturally generalises to the cotangent fibre and the zero section of the cotangent bundle $T^*Q$.
The zero section is an example of a more general class: the graph $\Gamma(\alpha) \subset T^*Q$ of any closed 1-form $\alpha$ on $Q$.
  
In the examples which arise from complex geometry, one may use real geometry to produce examples by: fix a smooth projective variety $Z \subset \bCP^n$ that is defined by a set of equations with real coefficients.
Considering $Z$ as a symplectic manifold equipped with the restriction of the Fubini--Study form, the real locus $Z \cap \bRP^n$ (whenever it is smooth) is a Lagrangian submanifold.
\null\hfill$\triangle$  
\end{example}

\subsection{Symplectic invariants from pseudoholomorphic curves}

Unlike Riemannian geometry, symplectic geometry has no local symplectic invariants as a consequence of Darboux's theorem \cite{darboux}.
Below, we state this theorem in combination with Weinstein's Lagrangian neighborhood theorem \cite{weinstein:lag_submfds}, which is the analogous result for the local geometry of near a Lagrangian submanifold.
Weinstein's theorem involves the notion of a \emph{symplectomorphism}, which is a diffeomorphism $\varphi\colon M \sr{\cong}{\lra} N$ between two symplectic manifolds that satisfies $\varphi^*\omega_N = \omega_M$.

\begin{theorem}
Any point in a symplectic manifold admits a neighbourhood which is symplectomorphic to a neighbourhood of the origin in $(\bC^n,\omega_0)$.
Similarly, any Lagrangian embedding in $(M,\omega)$ of a closed manifold $L$ extends to a symplectomorphism between a neighborhood of the zero section in $(T^*L, \omega_{\mathrm{can}})$ and a neighborhood of $L$ in $M$.
\null\hfill$\square$
\end{theorem}

The reader new to this field may get the sense from these theorems that symplectic geometry is similar in flavor to differential topology, but this is not the case.
Indeed, a motif in symplectic geometry is the interplay between flexibility and rigidity.
In the foundational paper \cite{gromov}, Mikhail Gromov opened the floodgates to a wide variety of rigidity results, by importing holomorphic techniques from complex geometry.
Consider, for instance, the following result.

\begin{theorem}[Theorem 0.4.A$_2$, \cite{gromov}]
For any closed embedded Lagrangian $L \subset \bC^n$, there exists a non-constant map $u : D^2 \to \bC^{n}$, mapping the boundary to $L$, and which is holomorphic with respect to the standard complex structures on $D^2$ and $\bC^n$.
\null\hfill$\square$
\end{theorem}

\noindent
If we write $J$ for the standard complex structure on $\bC^n$, and $j$ for complex structure on the disc, the holomorphicity condition on $u$ amounts to the requirement that the operator
\begin{equation} \label{eq:dbar}
  \dbar u \equiv du - J \circ du \circ j  
\end{equation}
vanish pointwise on the domain.
We can easily deduce the following corollary, which establishes a topological obstruction to Lagrangian embeddings into $\bCP^n$.

\begin{corollary}
Suppose that $L$ is a closed $n$-manifold with $H^1(L;\bR) = 0$.
Then $L$ does not admit a Lagrangian embedding into $\bC^n$.	
\end{corollary}

\begin{proof}
Define $\lambda \in \Omega^1(\bC^n;\bR)$ by $\lambda \coloneqq \tfrac12\sum_{i=1}^n (x_i\,\rd y_i - y_i \,\rd x_i)$, and note that $\lambda$ is a primitive of $\omega_0$.
Denote by $\gamma$ the restriction of $u$ to $S^1$.
By Stokes's theorem, we have:
\begin{align}
\int_{S^1} \gamma^*\lambda
=
\int_{D^2} u^*\omega_0
>
0,
\end{align}
where the inequality follows from the fact that the latter integral is equal to the area of the image of $u$, which is nonnegative by the holomorphicity condition, and strictly so by non-constancy.
\end{proof}

Gromov's key insight in \cite{gromov} is that one can use similar ideas even in the case of symplectic manifolds that are not K\"ahler: if $(M,\omega)$ is any symplectic manifold, there is a contractible (in particular, nonempty) space of \emph{$\omega$-compatible almost complex structures}, i.e.\ endomorphisms $J\colon TM \to TM$ of the tangent bundle satisfying the following properties:
\begin{itemize}
\item[]
{\bf (AC structure)}
$J^2 = -\Id$.

\smallskip

\item[]
{\bf ($\omega$-compatible)}
The contraction $\omega(-,J-)$ defines a Riemannian metric on $M$.
\end{itemize}
While one can study holomorphic maps from a Riemann surface to an arbitrary almost complex manifold, the fundamental result proved by Gromov in \cite{gromov} is that the moduli space of such maps admits a natural compactification whenever the target is symplectic.
This is the foundation of all later developments extracting symplectic invariants from moduli spaces of holomorphic maps.

In this article, we will be primarily concerned with two symplectic invariants.
The first is the \emph{Floer cohomology group}\footnote{This group is not defined for arbitrary pairs $(L,K)$, as its construction depends on a choice of additional data which may not always exist.
We suppress this point until \S\ref{ss:anom-lagr-floer} below.} $HF^*(L,K)$ associated to a pair $L$ and $K$ of appropriate Lagrangians in a symplectic manifold $M$, which categorifies their intersection number.
Andreas Floer introduced this invariant in the 1980s, and as an immediate consequence obtained a proof of one version of the \emph{Arnold--Givental conjecture}.
Briefly, $HF^*(L,K)$ is the homology of a chain complex $CF^*(L,K)$ freely generated by the elements of $L \cap K$.
The differential is defined by counting holomorphic maps from $\bR \times [0,1]$ to $M$, with boundary conditions defined by $L$ and $K$.
We will discuss Floer cohomology in more detail in \S\ref{ss:HF}.

\medskip

Lagrangian Floer cohomology groups form the morphism spaces of the second invariant which we will consider, the \emph{Fukaya $A_\infty$-category} $\Fuk M$, whose objects are Lagrangians (appropriately decorated).
Its definition originated in work of Simon Donaldson and Kenji Fukaya in the early 1990s, and over the intervening three decades its structure and properties have been steadily developed.
It plays a central role in Maxim Kontsevich's \emph{Homological Mirror Symmetry (HMS) conjecture} \cite{kontsevich_hms}, which posits that in certain situations there are pairs $(M,X)$ of a symplectic manifold $M$ and a complex algebraic variety $X$ for which a ``derived'' version of $\Fuk M$ is equivalent to an invariant of $X$ called the \emph{derived category of coherent sheaves on $X$}.
A great deal of work has gone into proving and refining the HMS conjecture in various settings.
In \S\ref{ss:fuk} we will give an overview of the definition and of some of the properties of $\Fuk M$.

\subsection{Functorial properties of pseudoholomorphic curve invariants, the Operadic Principle, and the plan for this paper}
\label{ss:functorial_properties}

If one wants to develop a toolbox for computing pseudoholomorphic curve invariants, the following is an obvious question:

\medskip

\begin{center}
\fbox{\parbox{0.9\columnwidth}{
If one understands the Fukaya category of a symplectic manifold $M$, which is geometrically related to a possibly-different manifold $N$, is it possible to then compute $\Fuk N$?
}}
\end{center}

\medskip

\noindent
Until the late 2000s, the only answer with any of degree of generality was given by Seidel in \cite{seidel_picard-lefschetz}, where he demonstrated an inductive method for computing the Fukaya $A_\infty$-category of the total space of a Lefschetz fibration $M \hra E \twoheadrightarrow D^2$ in terms of the Fukaya category of $M$.
While Seidel's work provides a powerful toolbox, it is limited to the setting of Lefschetz fibrations, and we might want a more flexible framework.
To that end, consider this variant on the above question:

\medskip

\begin{center}
\fbox{\parbox{0.9\columnwidth}{
What functorial properties are enjoyed by Floer cohomology, the Fukaya $A_\infty$-category, and other symplectic invariants defined by counting pseudoholomorphic curves?
}}
\end{center}

\medskip

\noindent
One approach to this question is given by Weinstein's \emph{symplectic creed} \cite{weinstein_symplectic_category}, which states that ``Everything is a Lagrangian submanifold.''
In particular, this suggests that when it comes to the Fukaya category, we should attempt to associate functors to Lagrangian correspondences, i.e.\ Lagrangians $L_{12} \subset M_1^- \times M_2 \coloneqq (M_1\times M_2, (-\omega_{M_1}) \oplus \omega_{M_2})$.
In the late 2000s, Wehrheim and Woodward pursued this approach, which led them to develop their theory of \emph{pseudoholomorphic quilts}.
In \S\ref{s:quilted_floer}, we will describe this work in detail.
In \S\ref{s:symp}, we will describe the second author's development of the symplectic $(A_\infty,2)$-category.
Finally, in \S\ref{s:applications}, we will survey a variety of applications of the theory of pseudoholomorphic quilts.

Throughout this paper, we will emphasize the following principle:

\medskip

\begin{center}
\fbox{\parbox{0.9\columnwidth}{
{\bf The operadic principle in symplectic geometry:}
The algebraic nature of a symplectic invariant defined by counting rigid pseudoholomorphic maps is inherited from the operadic structure of the underlying collection of domain moduli spaces.
}}
\end{center}

\medskip

Finally, we note that Kenji Fukaya made a major contribution to this field in his 2017 preprint \cite{fukaya2017unobstructed}.
Specifically, he associates functors to Lagrangians correspondences under very general hypotheses.
Fukaya used quilts to accomplish this, but he took a quite different approach from that taken by Wehrheim and Woodward.
See \S\ref{ss:fukaya} for an account of Fukaya's work.

\medskip

\subsection{Acknowledgments}
N.B.\ was supported by an NSF Standard Grant (DMS-1906220) during the preparation of this article.
He is grateful to the Max Planck Institute for Mathematics in Bonn for its hospitality and financial support.
M.A.\ would like to thank Kobi Kremnizer for asking him, many years ago, about whether there is a place for the Barr--Beck theorem in Floer theory. He was supported by an NSF Standard Grant (DMS-2103805), the Simons Collaboration on Homological Mirror Symmetry, a Simons Fellowship award, and the Poincaré visiting professorship at Stanford University.
The authors thank Kenji Fukaya, Yank{\i} Lekili, and Paul Seidel for useful conversations.
%%% Local Variables:
%%% mode: latex
%%% TeX-master: "functoriality_in_categorical_symplectic_geometry"
%%% End:

\section{Floer cohomology, the Fukaya $A_\infty$-category, and the Operadic Principle}
\label{s:fuk}

\begin{center}
\fbox{\parbox{0.9\columnwidth}{
{\bf Default geometric hypotheses:}
In \S\S\ref{ss:HF}--\ref{ss:fuk}, we assume our symplectic manifolds and Lagrangians are closed and aspherical, i.e.\ satisfy $\pi_2(M,L) = 0$, unless otherwise stated.
In \S\ref{ss:anom-lagr-floer}, we relax this hypothesis and work with general closed symplectic manifolds.
}}
\end{center}

\smallskip

In this section, we will introduce some fundamental objects in categorical symplectic geometry.
After introducing Floer cohomology in \S\ref{ss:HF} and the Fukaya $A_\infty$-category in \S\ref{ss:fuk}, we will explain in \S\ref{ss:associahedra_and_OP} that $\Fuk M$ is the first instance of the Operadic Principle mentioned in \S\ref{ss:functorial_properties}.

\subsection{Floer cohomology}
\label{ss:HF}

Given two Lagrangians $L$ and $K$ in a symplectic manifold, their Lagrangian Floer cohomology group $HF^*(L,K)$, when defined, categorifies their intersection number.
Shortly, we will mention a major result that motivated Floer to define this invariant.
Before this, we need to introduce the notion of a Hamiltonian diffeomorphism.

Given symplectic manifolds $M$ and $N$, a \emph{symplectomorphism} $\varphi\colon M \to N$ is a diffeomorphism satisfying $\varphi^*\omega_N = \omega_M$.
The infinitesimal version of self-symplectomorphisms of $M$ is given by the \emph{symplectic vector fields}, i.e.\ those $X \in \fX(M)$ with the property that $\omega(X,-)$ is closed.
Indeed, this follows from Cartan's magic formula:
\begin{align}
\cL_X\omega = \rd(\iota_X\omega) + \iota_X\rd\omega = \rd(\iota_X\omega).
\end{align}
An important class of symplectic vector fields is formed by the \emph{Hamiltonian vector fields}, i.e.\ those $X$ for which $\omega(X,-)$ is not only closed, but exact.
Note that since $\omega$ is nondegenerate, we can associate to any smooth function $H\colon M \to \bR$ a Hamiltonian vector field $X_H$ defined by solving the equation $\omega(X_H,-) = \rd H$.
Given a path of functions $H_t\colon M \to \bR$, we can integrate the associated vector fields $X_{H_t}$ to obtain a symplectomorphism $\varphi\colon M \to M$.
Such a map is called a \emph{Hamiltonian diffeomorphism}.

In 1988, Floer introduced Lagrangian Floer cohomology $HF^*(L,K)$ in order to prove the following case of a conjecture due to Arnold (\cite[Appendix 9]{arnold_methods}, \cite{arnold_sur_une_propriete}) and typically referred to as the Arnold--Givental conjecture.

\begin{theorem}[$t=1$ case of Theorem 1, \cite{floer_lag_int}]
\label{thm:arnold-givental}
Suppose $M$ is a closed symplectic manifold and that $L \subset M$ is a Lagrangian with $\pi_2(M,L) = 0$.
Fix a Hamiltonian diffeomorphism $\phi$ of $M$ such that $L$ and $\phi(L)$ intersect transversely.
Then the following estimate holds:
\begin{align}
\#|L \cap \phi(L)|
\geq
\sum_{i=0}^{\dim L}
\rk H^i(L; \bZ/2\bZ).
\end{align}
\null\hfill$\square$
\end{theorem}

\noindent
At the beginning of this subsection we called the Floer cohomology $HF^*(L,K)$ an ``invariant'', but we did not specify what it is invariant with respect to.
In fact, $HF^*(L,K)$ is built so that there is a canonical isomorphism $HF^*(L,K) \sr{\simeq}{\lra} HF^*(L,\phi(K))$ for $\phi$ a Hamiltonian diffeomorphism, and this isomorphism was the key ingredient in the proof of this result.

\begin{remark}
We should think of this result as saying that deforming Lagrangians by Hamiltonian vector fields is a less flexible operation than we might expect from purely differential-topological considerations.
Indeed, the normal and tangent bundles $\nu L, TL$ are isomorphic (an $\omega$-compatible almost complex structure, as introduced later in this subsection, defines such an isomorphism).
Choose a vector field $X \in \fX(L)$ whose zeroes are isolated and have index $\pm 1$.
On one hand, the Poincar\'{e}--Hopf index theorem implies that the sum of the indices of the zeroes of $X$ is equal to the Euler characteristic $\chi(L)$.
On the other hand, our identification $\nu L \simeq TL$ allows us to interpret this same sum as the signed intersection number of $L$ with a transverse pushoff of itself.
We can summarize this reasoning in the following inequality:
\begin{align}
\#|L \cap \phi(L)|
\geq
\sum_{i=0}^{\dim L}
(-1)^i\rk H^i(L; \bZ).
\end{align}
\null\hfill$\triangle$
\end{remark}

\subsubsection{The definition of $HF^*(L,K)$.}

Fix transversely-intersecting Lagrangians $L$ and $K$.
In this subsubsection, we will sketch the definition of $HF^*(L,K)$.
$HF^*(L,K)$ is the homology of a chain complex whose chain group is generated freely by intersection points:
\begin{align}
\label{eq:HF_definition}
HF^*(L,K)
\coloneqq
H\bigl(
CF^*(L,K)
\coloneqq
\Lambda\langle p\rangle_{p \in L \cap K},
d
\bigr).
\end{align}
Given a base field $\bk$, the chain group $CF^*(L,K) $ is generated freely over the \emph{Novikov field},
\begin{align}
\Lambda
\coloneqq
\left\{
\sum_{k=0}^\infty
\left.
a_kT^{\lambda_k}
\:\right|\:
a_k \in \bk,
\:
\lambda_k \in \bR,
\:
\lim_{k\to\infty} \lambda_k = +\infty
\right\}.
\end{align}
For simplicity, we will take $\bk = \bZ/2\bZ$; one can arrange to work with other base fields, assuming additional hypotheses on $M$, $L$, and $K$.
We work with a power series ring because if we did not, the differential would not necessarily converge.

\begin{remark}
The complex $CF^*(L,K)$ is typically graded by $\bZ$ or a finite cyclic group.
The grading will not be important to us in this paper.
\null\hfill$\triangle$
\end{remark}

The matrix coefficients of $d$ are defined in terms of the cardinalities of certain moduli spaces of pseudoholomorphic curves in $M$.
To set this up, we define the notion of an $\omega$-compatible almost complex structure.

\begin{definition}
An \emph{$\omega$-compatible almost complex structure on $M$} is an endomorphism $J\colon TM \to TM$ of the tangent bundle with $J^2 = -\Id$ and such that $\omega(-,J-)$ is a Riemannian metric on $M$.
The space of such $J$'s is denoted $\cJ(M,\omega)$.
We also allow for time-dependence, by taking paths $[0,1] \to \cJ(M,\omega)$.
\null\hfill$\triangle$
\end{definition}

\noindent
There always exists an $\omega$-compatible $J$; in fact $\cJ(M,\omega)$ is homotopy-equivalent to a point (cf.\ \cite[Proposition 2.50(iii)]{mcduff_salamon:small}).

Given a time-dependent $\omega$-compatible $J_t$, generators $p_-, p_+ \in L \cap K$, and a homotopy class $[v] \in \pi_2(M,L\cap K)$ of Floer strips, define the following moduli space of $J_t$-holomorphic strips:
\begin{gather}
\cM(p_-,p_+;[v])
\coloneqq
\bigl\{
u
\colon
\bR\times[0,1]
\to
M
\:\big|\:
(*)
\bigr\}/\bR,
\\
(*)
\quad\coloneqq\quad
u(s,0) \in L,
\quad
u(s,1) \in K,
\quad
\lim_{s\to\pm\infty} u(s,t) = p_\mp,
\quad
\partial_s u + J_t(u)\partial_t u = 0,
\quad
[u] = [v],
\nonumber
\end{gather}
where the $\bR$-action used on the first line is defined by an overall translation in $s$.
We can depict these \emph{Floer strips} either by drawing the domain of $u$ with the different parts labeled by where they map to, or by drawing a representation of the image of $u$, as in the following figure:
\begin{figure}[H]
\centering
\def\svgwidth{0.65\columnwidth}
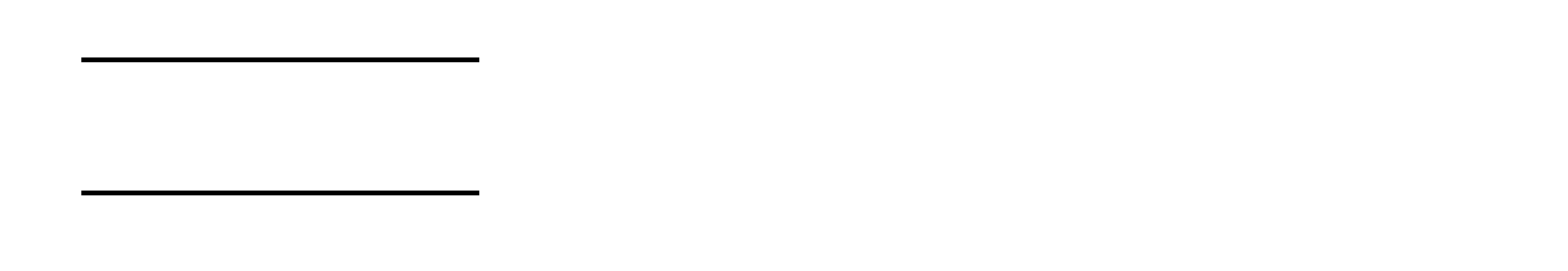
\caption{
\label{fig:floer_strip}
}
\end{figure}
\noindent
As long as $J_t$ is chosen generically, Floer--Hofer--Salamon proved in \cite{floer_hofer_salamon} that $\cM(p_-,p_+,[v])$ is a smooth manifold of dimension $\ind v - 1$, where $\ind v$ is the \emph{Maslov index} of $u$.

We can now define $d\colon CF^*(L,K) \to CF^*(L,K)$:
\begin{align}
\label{def:HF_d_definition}
dp_-
\coloneqq
\sum_{p_+ \in L\cap K}
\sum_{[v]: \ind v = 1}
\#\cM(p_-,p_+;[v])\cdot T^{\omega(v)}\cdot p_+,
\end{align}
where $\omega(v) \coloneqq \int v^*\omega$ is the symplectic area of $v$.
Convergence follows from \emph{Gromov compactness}, which guarantees that for any $p_-$, $p_+$, $[v]$, and $C > 0$, the moduli space $\cM(p_-,p_+;[v]) \cap \{u \:|\: \omega(u) \leq C\}$ is compact up to disks and spheres bubbling off, and strips breaking into chains of strips.

It remains to show $d^2=0$.
This is the step which fails for general pairs of Lagrangian submanifolds, and which will require us to use the topological assumption in Theorem \ref{thm:arnold-givental}.
According to \eqref{def:HF_d_definition}, $d^2=0$ is equivalent to the following equality holding for every $p_-$ and $p_+$:
\begin{align}
\label{eq:d^2=0_identity}
\sum_{q \in L \cap K}
\sum_{\ind v_1 = \ind v_2 = 1}
\#\cM(p_-,q;[v_1])\cdot\#\cM(q,p_+;[v_2])\cdot T^{\omega(v_1)+\omega(v_2)}
=
0.
\end{align}
Floer proved \eqref{eq:d^2=0_identity} by showing that for every $[v]$ of index 2, the discrete space
\begin{align}
\bigsqcup_{q\in L\cap K}
\bigsqcup_{{[v_1]+[v_2]=[v],}
\atop{\ind[v_i] = 1}}
\cM(p_-,q;[v_1]) \times \cM(q,p_+;[v_2])
\end{align}
is nullcobordant.
In fact, $\ol\cM(p_-,p_+;[v])$ provides such a nullcobordism, where this space is the result of compactifying $\cM(p_-,p_+;[v])$ by allowing bubbling and breaking phenomena.
Indeed, $\ol\cM(p_-,p_+;[v])$ is a moduli space of dimension 1, and the codimension-1 boundary strata are formed by breaking and disk bubbling, as illustrated in the following figure:

\begin{figure}[H]
\centering
\def\svgwidth{0.9\columnwidth}
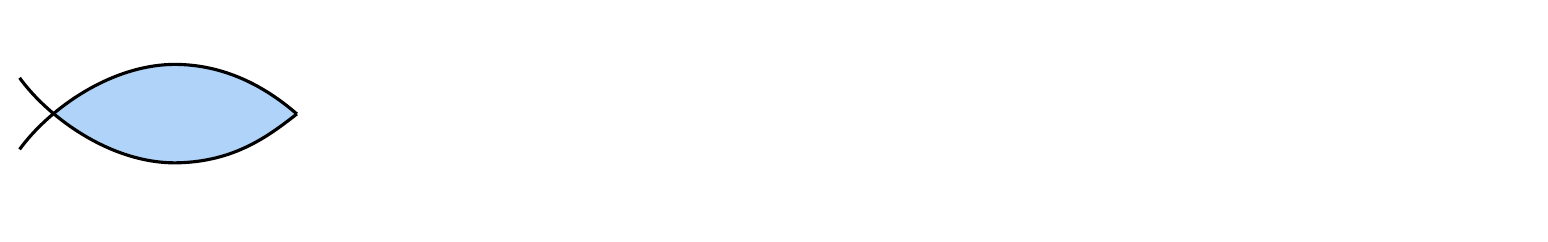
\caption{
\label{fig:first_curved_equation}
}
\end{figure}

\noindent
By the hypothesis of Theorem \ref{thm:arnold-givental}, there are no nontrivial disks with boundary on $L$ or $K$, hence \eqref{eq:d^2=0_identity} holds.

\begin{remark}[allowing non-transverse intersections]
\label{rmk:self-HF_models}
We have only described the definition of $HF^*(L,K)$ when $L, K$ intersect transversely.
There are various ways to extend this definition to general $L, K$.
For instance, one important task is to define the self-Floer cohomology $HF^*(L,L)$.
\begin{enumerate}
\item
One approach is to choose a Hamiltonian diffeomorphism $\phi$ such that $L$ and $\phi(L)$ intersect transversely, and to then set $HF^*(L,L) \coloneqq HF^*(L,\phi(L))$.
This is the tack taken in \cite{seidel_picard-lefschetz}, for instance; it is straightforward, but comes at the price of keeping careful track of all the perturbations that one has introduced.

\smallskip

\item
If one is willing to work with higher-dimensional moduli spaces of Floer strips, one can set $CF^*(L,L)$ to be the chain group for some model of the cohomology of $L$ --- e.g.\ $C_*^\sing(L)$ as in \cite{auroux_T-duality} or $C^*_{\text{dR}}(L)$.
A prominent example is \cite{fooo_1}.

\smallskip

\item
By extending the moduli space of Floer strips, one can avoid both choosing a pushoff of $L$ and working with an infinitely-generated chain group.
This is the ``cluster model'' for self-Floer cohomology, in which we choose a Morse function on $L$ and set $CF^*(L,L) \coloneqq C^*_{\text{Morse}}(L)$.
See \cite{fukaya:quantization} and \cite{oh:quantum} for foundational work on this topic.
Cornea--Lalonde later used this as the basis for their definition of ``cluster homology'' in \cite{cornea_lalonde_cluster}.
See \cite{li_thesis} for a very detailed construction of this model for $HF^*(L,L)$.
\null\hfill$\triangle$
\end{enumerate}
\end{remark}

\begin{remark}[other geometric settings where one can define $HF^*(L,K)$]
\label{rmk:geometric_settings}
We have sketched the definition of $HF^*(L,K)$ when the only disks on $L$ resp.\ $K$ are nullhomotopic, which was the setting originally considered by Floer.
(An immediate extension is to only disallow those disks of positive symplectic area.)
Here are some other typical settings for defining $HF^*(L,K)$:
\begin{enumerate}
\item
{\it The exact setting:}
$\omega = d\alpha$ is an exact 2-form, and $\alpha|_L = df$ and $\alpha|_K = df$ are exact 1-forms.
(Note that Stokes' theorem implies that $M$ is necessarily noncompact.)
In this setting, Stokes' theorem excludes all nonconstant pseudoholomorphic spheres and disks with boundary on $L$ or $K$.
This greatly simplifies the analysis, and $HF^*(L,K)$ is defined for a generic choice of $J_t$.
See \cite{seidel_picard-lefschetz} for a construction of $\Fuk$ in the exact setting.

\smallskip

\item
{\it The monotone setting:}
The homomorphisms
\begin{align}
\langle c_1(TM), -\rangle\colon \pi_2(M) \to \bZ,
\qquad
\int_{S^2} -^*\omega\colon \pi_2(M) \to \bR
\end{align}
are positively proportional, and a similar condition holds for $L$ and $K$.
In this setting, sphere and disk bubbling are not a priori excluded.
It is still possible to define the Floer differential $d$, but it does not necessarily square to zero: in fact, $d^2 = (\pm m_0(L) \mp m_0(K)) \id$, where $m_0(L)$ denotes the number of pseudoholomorphic disks that pass through a generic point on $L$.
As long as $m_0(L) = m_0(K)$, $HF^*(L,K)$ is well-defined.
See \cite{auroux_T-duality} for a survey of the definition of the Fukaya category in this setting.
(Originally, this approach is due to Oh, \cite{oh_monotone}.)

\smallskip

\item
{\it The general compact setting:}
When $M, L, K$ are assumed only to be closed, it is no longer possible to produce well-behaved moduli spaces by choosing $J_t$ appropriately.
One must rely on an \emph{abstract perturbation scheme}, see e.g.\ \cite{fooo_2}.
\null\hfill$\triangle$
\end{enumerate}
\end{remark}

\subsubsection{A first example: two great circles on $S^2$}
\label{sss:S2_HF}

Equip $S^2$ with an area form, and regard this as a symplectic manifold.
Define $L_\eq$ to be a great circle, by which we mean an embedded circle which divides $S^2$ into two halves of equal area.
Note that the hypotheses of Theorem \ref{thm:arnold-givental} fail, because the relative homology group $\pi_2(M,L)$ does not vanish.
Nonetheless, due to work of Oh \cite{oh_monotone}, this example falls within a larger class of examples for which the self-Floer cohomology is well-defined.
($M$ is monotone, as in Remark \ref{rmk:geometric_settings}.)
We now explain how to compute it.

We will use the first model for self-Floer cohomology described in Remark \ref{rmk:self-HF_models}, and compute $HF^*(L_\eq, L_\eq')$, where $L_\eq'$ is another great circle which intersects $L_\eq$ transversely, in two points.
Calling these points $p$ and $q$, our setup looks like this:

\begin{figure}[H]
\centering
\def\svgwidth{0.2125\columnwidth}
%% Creator: Inkscape 1.1 (c4e8f9e, 2021-05-24), www.inkscape.org
%% Accompanies image file 'HF_eq_setup.pdf' (pdf, eps, ps)
\begingroup%
  \makeatletter%
  \providecommand\color[2][]{%
    \errmessage{(Inkscape) Color is used for the text in Inkscape, but the package 'color.sty' is not loaded}%
    \renewcommand\color[2][]{}%
  }%
  \providecommand\transparent[1]{%
    \errmessage{(Inkscape) Transparency is used (non-zero) for the text in Inkscape, but the package 'transparent.sty' is not loaded}%
    \renewcommand\transparent[1]{}%
  }%
  \providecommand\rotatebox[2]{#2}%
  \newcommand*\fsize{\dimexpr\f@size pt\relax}%
  \newcommand*\lineheight[1]{\fontsize{\fsize}{#1\fsize}\selectfont}%
  \ifx\svgwidth\undefined%
    \setlength{\unitlength}{248.40347896bp}%
    \ifx\svgscale\undefined%
      \relax%
    \else%
      \setlength{\unitlength}{\unitlength * \real{\svgscale}}%
    \fi%
  \else%
    \setlength{\unitlength}{\svgwidth}%
  \fi%
  \global\let\svgwidth\undefined%
  \global\let\svgscale\undefined%
  \makeatother%
  \begin{picture}(1,0.85998419)%
    \lineheight{1}%
    \setlength\tabcolsep{0pt}%
    \put(0,0){\includegraphics[width=\unitlength,page=1]{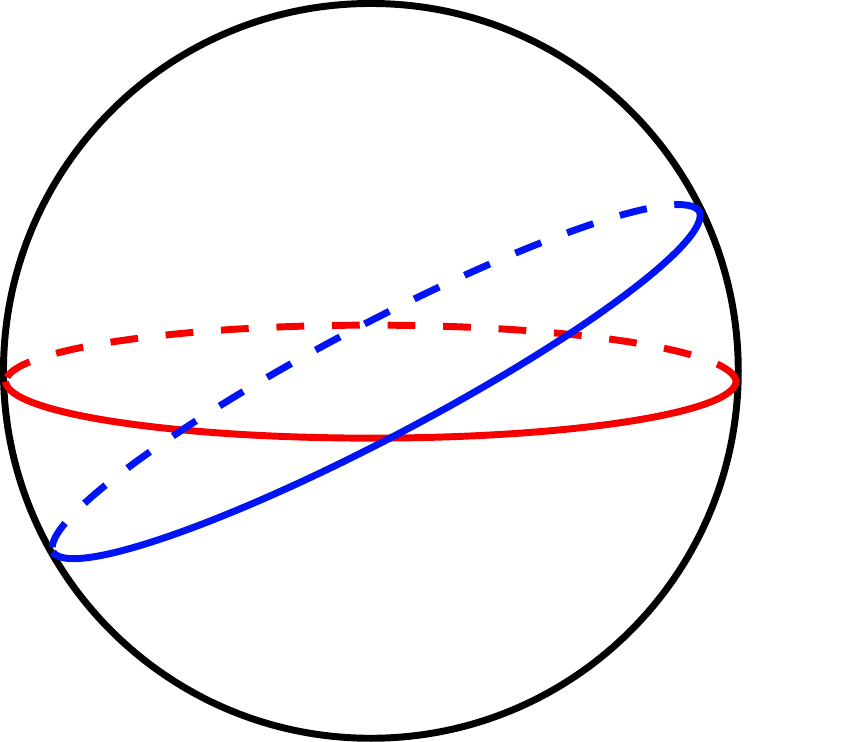}}%
    \put(0.3960452,0.53888074){\makebox(0,0)[lt]{\lineheight{1.25}\smash{\begin{tabular}[t]{l}$q$\end{tabular}}}}%
    \put(0.43599166,0.26173718){\makebox(0,0)[lt]{\lineheight{1.25}\smash{\begin{tabular}[t]{l}$p$\end{tabular}}}}%
    \put(0.83367605,0.6628689){\makebox(0,0)[lt]{\lineheight{1.25}\smash{\begin{tabular}[t]{l}$L'$\end{tabular}}}}%
    \put(0.90081953,0.39020392){\makebox(0,0)[lt]{\lineheight{1.25}\smash{\begin{tabular}[t]{l}$L$\end{tabular}}}}%
  \end{picture}%
\endgroup%

\caption{
\label{fig:HF_eq_setup}
}
\end{figure}

\noindent
The chain complex $CF^*(L_\eq, L_\eq')$ is freely generated by $p$ and $q$.
The Floer differential involves the following four strips, with the green strips on the left resp.\ the orange strips on the right contributing to $dp$ resp.\ $dq$:

\begin{figure}[H]
\centering
\def\svgwidth{1.0\columnwidth}
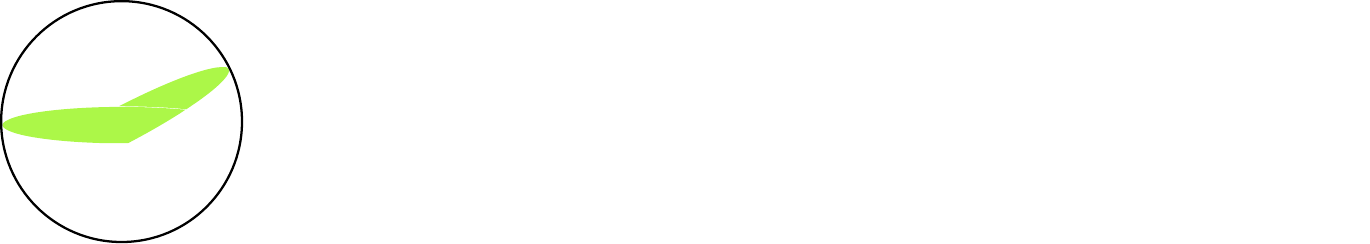
\caption{
\label{fig:HF_eq_strips}
}
\end{figure}

\noindent
This yields the following computation of $d$:
\begin{align}
dp
=
T^bq + T^bq
=
0,
\qquad
dq
=
T^aq + T^aq
=
0.
\end{align}
Here we have denoted by $a$ resp.\ $b$ the area of each of the two smaller resp.\ each of the two larger of the four portions that $S^2$ is divided into.
It follows that $HF^*(L_\eq, L_\eq') = \Lambda^2$.

Note that by the invariance of Floer cohomology, $L_\eq$ cannot be displaced from itself by a Hamiltonian diffeomorphism.
Note also that if we had made the same computation but for a non-great circle, then the differential would not have been zero.

\begin{remark}
There are several introductory texts that contain more detailed introductions to $HF^*(L,K)$.
We will therefore refrain from discussing the basic examples that illustrate the phenomena we have described.
\cite{auroux_beginners_guide} is a good starting point for the interested reader.
In particular, see Example 1.11 in that text to see an illustration of the fact that $d^2=0$ can fail if we remove Floer's hypothesis that $\pi_2(M,L) = \pi_2(M,K) = 0$.
\null\hfill$\triangle$
\end{remark}

\subsection{The Fukaya $A_\infty$-category}
\label{ss:fuk}

As we discussed in \S\ref{ss:HF}, Floer defined $HF^*(L,K)$ assuming that there are no nontrivial topological disks on $L$ and $K$, because this allowed him to exclude disk bubbling in his proof that $d^2 = 0$.
In the early 1990s, Fukaya (influenced by Donaldson) categorified Floer cohomology in such a way that disk bubbling can be interpreted as defining an algebraic operation (the curvature) which, together with Floer's differential $d$, form the first two operations in a coherent hierarchy.
From this data, one extracts (by a purely algebraic procedure), a categorification called the \emph{Fukaya $A_\infty$-category}.
It is an invariant of certain symplectic manifolds $M$, and it is denoted $\Fuk M$.

The Fukaya category has been defined in many settings.
(For instance, \cite[Part II]{seidel_picard-lefschetz} defines $\Fuk M$ in the exact setting, and serves as the classic introduction to $\Fuk$.
Fukaya--Oh--Ohta--Ono relied on the theory of Kuranishi structures to define $\Fuk M$ in the general compact setting in the formidable \cite{fooo_1,fooo_2}.)
Each incarnation is some variant of the following.

\begin{definition} \label{def:Fukaya_category}
The Fukaya $A_\infty$-category $\Fuk M$ consists of the following data:
\begin{itemize}
\item
An object set $\Ob$ consisting of Lagrangians $L \subset M$, equipped with some auxiliary data.

\smallskip

\item
For $L, K \in \Ob$, a graded $\bK$-vector space $\hom(L,K)\coloneqq CF^*(L,K)$.

\smallskip

\item
For $r \geq 1$ and $L^0,\ldots, L^r \in \Ob$, a $\bK$-linear map
\begin{align}
\mu_r
\colon
CF^*(L^0,L^1)
\otimes
\cdots
\otimes
CF^*(L^{r-1},L^r)
\to
CF^{*-2+r}(L^0,L^r)
\end{align}
called the $r$-ary composition map, and whose matrix coefficients are defined by counting rigid pseudoholomorphic $(r+1)$-gons with boundary on $L^0,\ldots,L^r$.
\end{itemize}
The operations satisfy the \emph{$A_\infty$-equations}:
\begin{align}
\label{eq:A-infinity_equations}
\sum_{{1 \leq a \leq r}
\atop
{1 \leq i \leq r-a+1}}
\pm\mu_{r-a+1}\bigl(x_1,\ldots,x_{i-1},\mu_a(x_i,\ldots,x_{i+a-1}),x_{i+a},\ldots,x_r\bigr)
=
0
\quad\forall\quad
r \geq 1.
\end{align}
As a result, $\Fuk M$, is a \emph{$\bK$-linear $A_\infty$-category}.
\null\hfill$\triangle$
\end{definition}

\begin{remark}
In our definition of Floer cohomology in \S\ref{ss:HF}, we stratified the compactified space of Floer strips using the Maslov index, then defined the Floer differential in terms of the cardinalities of the dimension-0 strata.
In an analogous way, in the definition of the Fukaya category, we stratify the spaces of pseudoholomorphic polygons, and count only the elements of the 0-dimensional strata.
This is what we meant when we referred to \emph{rigid} pseudoholomorphic polygons in the preceding definition.
\null\hfill$\triangle$
\end{remark}

\begin{remark}
In its most fundamental form, the Fukaya category is defined with Novikov coefficients.
(This actually leads to a definition of $\Fuk M$ as a \emph{filtered gapped $A_\infty$-category}, as we discuss in \S\ref{ss:anom-lagr-floer}.)
However, under certain assumptions, such as monotonicity or exactness, one can take coefficients in the base field $\bk$.
We have therefore left the choice of coefficients vague, and one should take $\bK$ to denote either $\Lambda$ or $\bk$.

Let us elaborate a bit about the monotone or exact case.
Under either of these hypotheses, $\Fuk(M; \Lambda)$ (where we are making the coefficient field explicit) can actually be defined over a polynomial version of the Novikov field, $\Lambda_\poly$, in which the power series in $T$ are finite.
(In the monotone case, one needs an additional ``balancedness'' hypothesis, which we will not explain here.)
This leads to the following diagram, in which the left pointing arrow is  the result of specializing the Novikov parameter $T$ to a chosen finite value in $\bk$, and the right pointing arrow is induced by the inclusion of $\Lambda_\poly$ in $\Lambda$.
\begin{align}
\xymatrix{
 & \Fuk(M; \Lambda_\poly) \ar@{-->}[ld] \ar@{-->}[rd] &
\\
\Fuk(M; \bk) & & \Fuk(M; \Lambda).
}
\end{align}
\null\hfill$\triangle$
\end{remark}

Note that we can interpret the first few $A_\infty$-equations like so:
\begin{itemize}
\item
$(\mu_1)^2 = 0$, i.e.\ $\mu_1$ is a differential on $CF^*(L^0,L^1)$ (which we denoted $d$ earlier).

\smallskip

\item
$\mu_2$ is a chain map with respect to $\mu_1$.

\smallskip

\item
$\mu_2$ is associative up to a chain homotopy given by $\mu_3$.
\end{itemize}
We can interpret an $A_\infty$-category as a $\bK$-linear dg-category in which composition is only homotopy-associative, but where we have a hierarchy of homotopies to control this lack of strict associativity.

To give the idea of how one verifies the $A_\infty$-equations for the Fukaya $A_\infty$-category, we will now sketch the proof of the fact that $\mu_3$ is a chain homotopy between the two different ways of composing $\mu_2$ with itself.
The following figure illustrates the codimension-1 degenerations of the 1-dimensional moduli space of pseudoholomorphic 4-gons:
\begin{figure}[H]
\centering
\def\svgwidth{0.7\columnwidth}
%% Creator: Inkscape 1.2 (dc2aeda, 2022-05-15), www.inkscape.org
%% Accompanies image file 'mu_2_is_associative_up_to_mu_3.pdf' (pdf, eps, ps)
\begingroup%
  \makeatletter%
  \providecommand\color[2][]{%
    \errmessage{(Inkscape) Color is used for the text in Inkscape, but the package 'color.sty' is not loaded}%
    \renewcommand\color[2][]{}%
  }%
  \providecommand\transparent[1]{%
    \errmessage{(Inkscape) Transparency is used (non-zero) for the text in Inkscape, but the package 'transparent.sty' is not loaded}%
    \renewcommand\transparent[1]{}%
  }%
  \providecommand\rotatebox[2]{#2}%
  \newcommand*\fsize{\dimexpr\f@size pt\relax}%
  \newcommand*\lineheight[1]{\fontsize{\fsize}{#1\fsize}\selectfont}%
  \ifx\svgwidth\undefined%
    \setlength{\unitlength}{542.36965472bp}%
    \ifx\svgscale\undefined%
      \relax%
    \else%
      \setlength{\unitlength}{\unitlength * \real{\svgscale}}%
    \fi%
  \else%
    \setlength{\unitlength}{\svgwidth}%
  \fi%
  \global\let\svgwidth\undefined%
  \global\let\svgscale\undefined%
  \makeatother%
  \begin{picture}(1,0.54484574)%
    \lineheight{1}%
    \setlength\tabcolsep{0pt}%
    \put(0,0){\includegraphics[width=\unitlength,page=1]{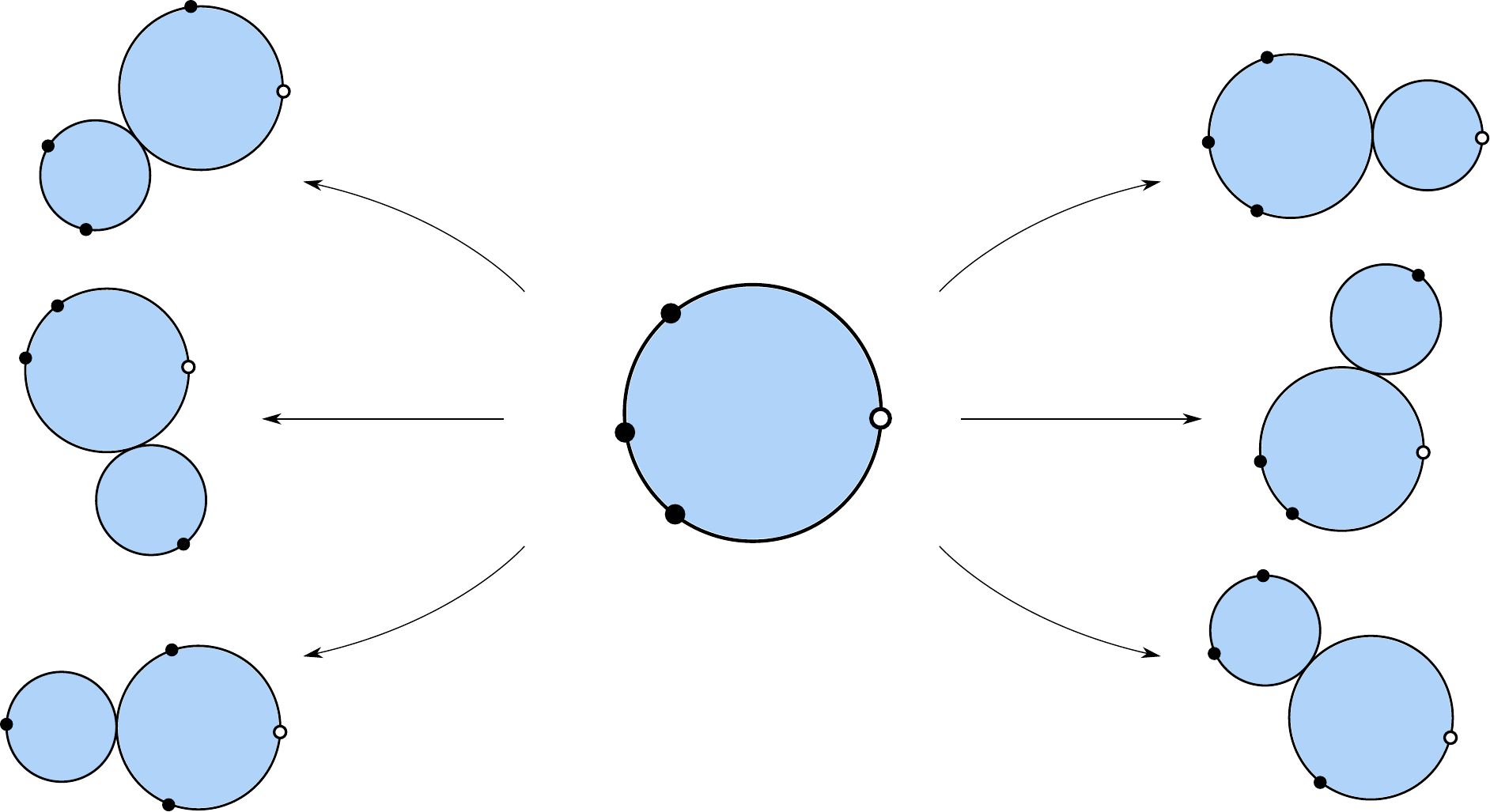}}%
    \put(0.51878065,0.36266407){\makebox(0,0)[lt]{\lineheight{1.25}\smash{\begin{tabular}[t]{l}$L_0$\end{tabular}}}}%
    \put(0.38961854,0.30102668){\makebox(0,0)[lt]{\lineheight{1.25}\smash{\begin{tabular}[t]{l}$L_1$\end{tabular}}}}%
    \put(0.39652873,0.21037863){\makebox(0,0)[lt]{\lineheight{1.25}\smash{\begin{tabular}[t]{l}$L_2$\end{tabular}}}}%
    \put(0.55076094,0.173072){\makebox(0,0)[lt]{\lineheight{1.25}\smash{\begin{tabular}[t]{l}$L_3$\end{tabular}}}}%
    \put(0.42569258,0.35051396){\makebox(0,0)[lt]{\lineheight{1.25}\smash{\begin{tabular}[t]{l}$p_1$\end{tabular}}}}%
    \put(0.38074307,0.25036709){\makebox(0,0)[lt]{\lineheight{1.25}\smash{\begin{tabular}[t]{l}$p_2$\end{tabular}}}}%
    \put(0.42758545,0.16708499){\makebox(0,0)[lt]{\lineheight{1.25}\smash{\begin{tabular}[t]{l}$p_3$\end{tabular}}}}%
    \put(0.60655403,0.25867255){\makebox(0,0)[lt]{\lineheight{1.25}\smash{\begin{tabular}[t]{l}$q$\end{tabular}}}}%
  \end{picture}%
\endgroup%

\caption{
\label{fig:third_a-infty-equation}
}
\end{figure}
In the same way that considering 1-dimensional moduli space of pseudoholomorphic strips led us to our proof of $d^2=0$ in \S\ref{ss:HF}, this analysis implies the ternary $A_\infty$-equation:
\begin{align}
\label{eq:ternary_A-infinity_eqn}
&\mu_2(\mu_2(x,y),z)
-
\mu_2(x,\mu_2(y,z))
\\
&\hspace{0.25in}=
\pm\mu_3(\mu_1(x),y,z)
\pm\mu_3(x,\mu_1(y),z)
\pm\mu_3(x,y,\mu_1(z))
\pm\mu_1(\mu_3(x,y,z)).
\nonumber
\end{align}

\begin{remark}
The punctilious reader will have noted the following difference between Figures \ref{fig:first_curved_equation} and \ref{fig:third_a-infty-equation}: boundary strata in which a component disc carries no marked point other than the node appear in the former but not the latter.
These terms in fact also appear in a general analysis of the boundary moduli space shown in Figure \ref{fig:third_a-infty-equation}, but the point of the auxilliary data alluded to in Definition \ref{def:Fukaya_category} is precisely to algebraically cancel these terms so that their contribution vanishes.
\null\hfill$\triangle$
\end{remark}

\begin{example}
\label{ex:S2_Fuk}
We now return to $S^2$.
In \S\ref{sss:S2_HF}, we deduced that if $L_\eq \subset S^2$ is a great circle, then $HF^*(L_\eq,L_\eq)$ is well-defined and isomorphic to $\Lambda^2$.
We will now turn to the task of determining $H\Fuk S^2$.
(Our treatment will be brief, and we invite the interested reader to consult \S3.4 of the excellent article \cite{ballard_meet_hms}.)

The only Lagrangians in $S^2$ are embedded circles.
If $S^1 \subset S^2$ does not divide $S^2$ into two halves of equal area, then a computation similar to the one made in \S\ref{sss:S2_HF} shows that its self-Floer cohomology is zero, hence that it is zero in $\Fuk S^2$.
On the other hand, Oh shows in \cite{oh_second_variation} that any two great circles are Hamiltonian-isotopic.
It follows that when we use Novikov coefficients over $\bZ/2\bZ$, $\Fuk S^2$ has one object, whose endomorphism algebra is $\Lambda^2$.
A computation of the composition operation would show that
\begin{align}
HF^*(L_\eq, L_\eq)
\simeq
\Lambda\langle x\rangle
\big/
\bigl\langle
x^2
=
T^{\omega(S^2)/2}
\bigr\rangle.
\end{align}
This is a \emph{Clifford algebra}.
(C.f.\ \cite[\S5.5]{smith_prolegomenon} for a discussion of the self-Floer cohomology of the moment fibers of a toric Fano variety.
As Cho proved in \cite{cho:toric_fiber_clifford_algebra}, this self-Floer cohomology is always a Clifford algebra.)
\null\hfill$\triangle$
\end{example}

\begin{remark}
By using $\bZ/2\bZ$ as our base field in this example, we avoided discussing the fact that for more general base fields (such as $\bC$), there are actually two nonzero objects of $\Fuk S^2$: $L_\eq$, equipped with a flat $U(1)$-bundle with monodromy either $1$ or $-1$.

This is an illustration of a more general fact (\cite[Corollary 1.3.1]{evans2019generating}) about compact monotone toric symplectic manifolds: working over a field of arbitrary characteristic, the Fukaya category is split-generated by copies of the unique monotone torus fiber, equipped with the finitely-many choices of flat line bundles for which this torus fiber is a nonzero object.

We also note that when we work over $\bC$, the $A_\infty$-category $\Fuk S^2$ contains no more information than $H\Fuk S^2$.
This follows from the \emph{intrinsic formality} of complex Clifford algebras (c.f.\ \cite[Corollary 6.4]{sheridan_fano}).
On the other hand, over $\bZ/2\bZ$ the Fukaya category of $S^2$ is not formal, c.f.\ \cite[\S\S7.2--3]{evans2019generating}.
\null\hfill$\triangle$
\end{remark}

\subsection{Associahedra and the operadic principle}
\label{ss:associahedra_and_OP}

In \S\ref{ss:fuk}, we sketched the proof of the third $A_\infty$-equation, which was based on analyzing the boundary strata of a 1-dimensional space of pseudoholomorphic 4-gons.
The reason that this works, and the reason that analogous arguments allow one to establish all the $A_\infty$-equations for $\Fuk M$, is that the moduli space of domains have a certain recursive structure.
In order to articulate this, we formally define these domain moduli spaces.

\begin{definition}
For $r \geq 2$, define $K_r$ to be the compactified configuration space of $r$ distinct unlabeled points on the real line, modulo translations and dilations.
Equivalently, we can view such a configuration as a configuration of $r+1$ marked points on the boundary of the unit disk, with one point distinguished, modulo the M\"{o}bius transformations that preserve the unit disk.
\null\hfill$\triangle$
\end{definition}

\noindent
The compactification amounts to a recipe for how to define the limit of a path in which some points collide (in the disk model).
It turns out that the right compactification to use here is essentially the one proposed by Fulton--MacPherson in \cite{fulton_macpherson}: when points collide, we ``zoom in'' in order to remember the relative rates of the collision.
This zoomed-in view produces a ``bubbled-off'' configuration of points on a disk, which is attached to the original disk at the point where the points collided.
We illustrate this in the following depictions of $K_3$ and $K_4$:

\begin{figure}[H]
\centering
\def\svgwidth{0.8\columnwidth}
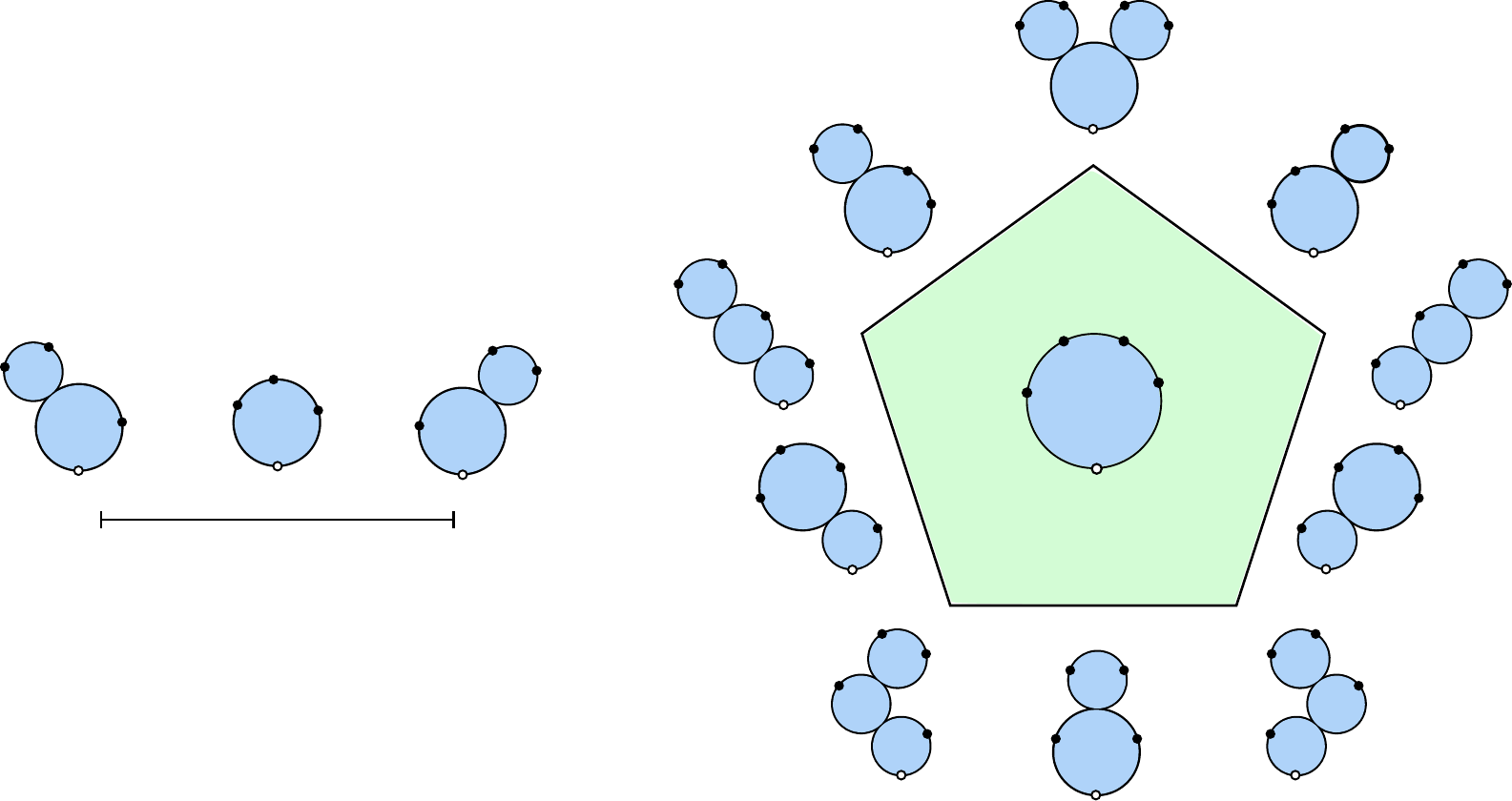
\caption{
\label{fig:K3_K4}
}
\end{figure}

\noindent
Here we have labeled each stratum by a typical example.
It is apparent in these examples that the posets of strata have straightforward combinatorial interpretations.
For instance, we can identify the poset of strata of $K_4$ with posets of stable (i.e., no vertex of valence 2) planted planar trees with four leaves, or with parenthesizations of four letters:

\begin{figure}[H]
\centering
\def\svgwidth{1.0\columnwidth}
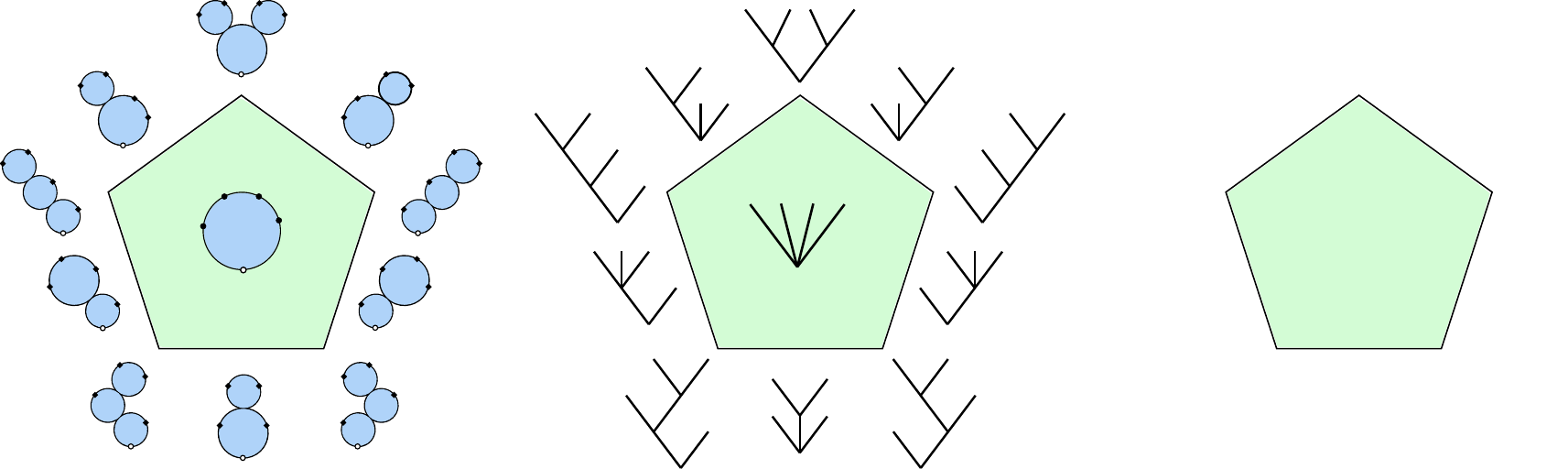
\caption{
\label{fig:K4_models}
}
\end{figure}

\noindent
These spaces are called \emph{associahedra}.
They were originally defined by Stasheff in \cite{stasheff_h-spaces} in the context of recognizing loop spaces, and they arise in many situations that involve homotopy associativity.
They deserve the suffix ``-hedra'' because they can be realized as convex polytopes.
There are now a variety of different polytopal realizations.
In symplectic geometry, their use goes back at least to \cite{fukaya_oh}.

Recall our proof in \S\ref{ss:fuk} of the ternary $A_\infty$ equation \eqref{eq:ternary_A-infinity_eqn}: we produced a 1-dimensional nullcobordism of a disjoint union of finite sets, such that the count over each finite set corresponded to one of the terms in \eqref{eq:ternary_A-infinity_eqn}.
This argument comes into clearer focus when seen through the lens of the associahedra.
For instance, consider the task of proving the quaternary $A_\infty$ equation.
To do so, we fix inputs $a, b, c, d$ and regard the 1-dimensional moduli space of pseudoholomorphic pentagons as a nullcobordism.
The projection of this nullcobordism to the domain moduli space $K_4$ looks something like this:

\begin{figure}[H]
\centering
\def\svgwidth{0.7\columnwidth}
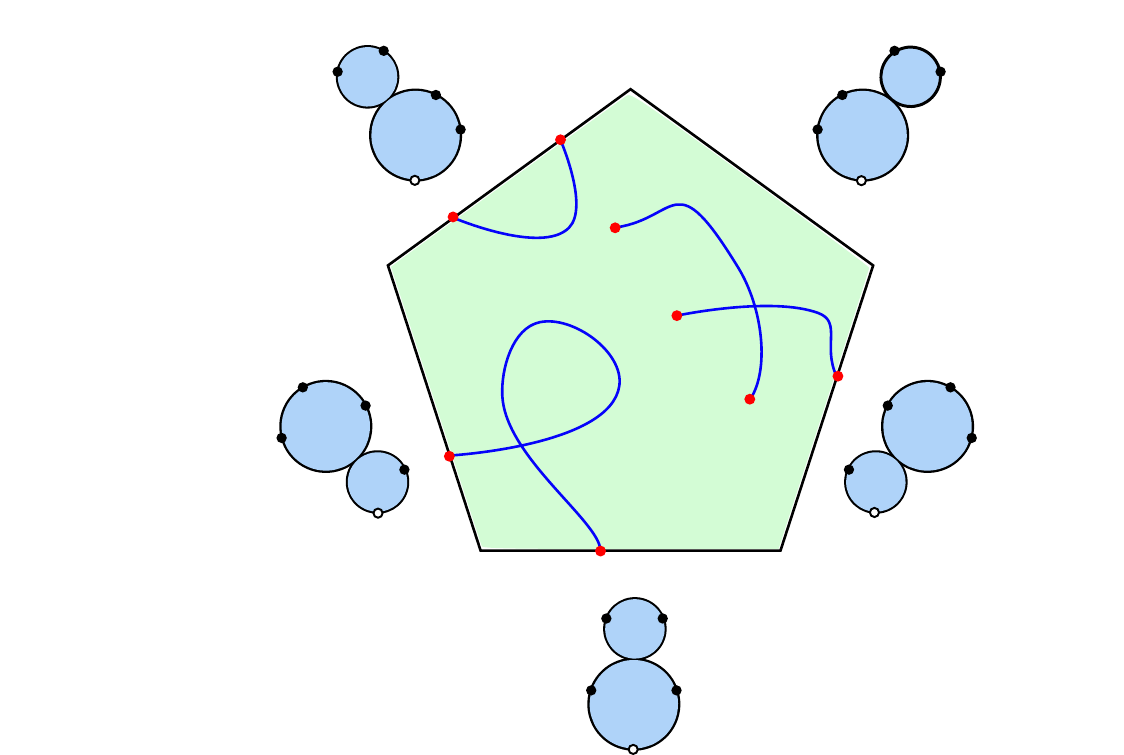
\caption{
\label{fig:K4_with_projected_moduli_spaces}
}
\end{figure}

\noindent
Here we have labeled the ends of this nullcobordism by the terms in the quaternary $A_\infty$ equation that they correspond to.
Note that the ends that lie in the boundary correspond to terms of the form $\mu_i(\cdots,\mu_j(\cdots),\cdots)$ for $i, j \geq 2$, whereas the ends in the interior correspond to terms with either $i=1$ or $j=1$.

This is the first instance we have encountered so far of the Operadic Principle, which we remind the reader of:

\medskip

\begin{center}
\fbox{\parbox{0.9\columnwidth}{
{\bf The Operadic Principle in symplectic geometry:}
The algebraic nature of a symplectic invariant defined by counting rigid pseudoholomorphic maps is inherited from the operadic structure of the underlying collection of domain moduli spaces.
}}
\end{center}

\medskip

\noindent
This is implemented by following the following procedure:

\begin{itemize}
\item
Define stratified and compactified moduli spaces of pseudoholomorphic maps, as well as the with associated domain moduli spaces.

\smallskip

\item
Prove a ``gluing theorem'', which identifies the boundary of the 1-dimensional moduli spaces with spaces constructed from the 0-dimensional moduli spaces.

\smallskip

\item
Interpret this recursive structure ``operadically'', i.e.\ that the curve-counting invariant resulting from counts over the 0-dimensional moduli spaces is a category over the operad of chains on the domain moduli spaces (or some related notion).
\end{itemize}

\begin{remark}
This implementation procedure potentially hides an enormous amount of technical complexity.
For instance, in the first bullet, one needs to define moduli spaces that are well-defined topological objects.
In simple settings, such as the one Floer worked in, one can do this with only a moderate amount of work.
In other settings, such as when $M$ is a general compact symplectic manifold, one needs to choose a framework for virtual counts, such as the virtual approach in \cite{fooo_2}, the polyfolds package defined in a series of papers by Hofer, Fish, Wysocki, and Zehnder, or Pardon's approach \cite{pardon}.
\null\hfill$\triangle$
\end{remark}

\begin{remark}
The notion of exploiting the operadic structure of the domain moduli spaces, especially when considering the $A_\infty$-operad or with variants of the moduli of nodal genus-0 stable curves, is not new in symplectic geometry; c.f.\ for instance \cite{fukaya:operads}.
\null\hfill$\triangle$
\end{remark}

\begin{remark}
Our definition of $K_r$ only makes sense for $r \geq 2$.
In fact, in the Floer-theoretic setting one can set $K_1 \coloneqq \pt \sqcup \pt/\bR$, where the second term should be interpreted as an Artin stack.
In terms of operations, $\pt$ corresponds to a unit in $CF^*(L,L)$, and $\pt/\bR$ corresponds to the differential.
\null\hfill$\triangle$
\end{remark}

We conclude this subsection by introducing the notions of operads and categories over them, which we need in order to make the Operadic Principle precise in the case of the Fukaya $A_\infty$-category.

\begin{definition}
Fix a symmetric monoidal category $\sC$ (such as $\Top$, $\Set$, or $\Ch$).
A \emph{nonsymmetric operad $\cO$ in $\sC$} is the following data:

\begin{itemize}
\item
For every $r \geq 1$, an object $\cO(r)$ in $\sC$, which we think of as a collection of ``$r$-ary operations''.

\smallskip

\item
Morphisms
\begin{align}
\circ\colon \cO(r) \times \cO(s_1)\times\cdots\times\cO(s_r) \to \cO(s_1+\cdots+s_r)
\end{align}
for any choice of $r, s_1, \ldots, s_r \geq 1$.
We think of these as ``composition maps'' and we require their compositions to satisfy an appropriate associativity condition.

\smallskip

\item
An element $1 \in \cO(1)$ that acts as the identity.
\null\hfill$\triangle$
\end{itemize}
\end{definition}

\noindent
For instance, the associahedra $(K_r)_{r\geq 1}$ form an operad in $\Top$, where we define $K_1 \coloneqq \pt$ and where the composition maps $K_r \times K_{s_1} \times \cdots \times K_{s_r} \to K_{s_1+\cdots+s_r}$ are defined by concatenating trees of disks.
An operad that we will need shortly is $\bigl(C_*^\cell(K_r)\bigr)$, which is an operad in $\Ch$.

Given an operad $\cO$, we need a way of turning $\cO(r)$, which consists morally of ``$r$-ary operations'', into genuine operations.
The right receptacle for this sort of procedure is given by the following notion.

\begin{definition}
Fix a (nonsymmetric) operad $\cO$ in $\sC$.
Then a \emph{category over $\cO$} (or an \emph{$\cO$-category}) $\sA$ consists of the following data:
\begin{itemize}
\item
A set of objects $\Ob \sA$.

\smallskip

\item
For every $X, Y \in \Ob \sA$, an object $\hom(X,Y) \in \sC$ called the \emph{morphism space from $X$ to $Y$}.

\smallskip

\item
For every sequence $X_0, \ldots, X_r \in \Ob \sA$, a \emph{composition operation}
\begin{align}
\cO(r)
\times
\hom(X_0,X_1) \times \cdots \times \hom(X_{r-1},X_r)
\to
\hom(X_0,X_r).
\end{align}
\end{itemize}
We require these operations to be compatible with the composition maps in $\cO$, and with the unit $1 \in \cO(1)$.
\null\hfill$\triangle$
\end{definition}

$\Fuk M$ is a category over $\bigl(C_*^\cell(K_r)\bigr)$, because we can associate to a cell $C \subset K_r$ an $r$-ary operation defined by counting pseudoholomorphic maps whose domain is in $C$.
Using the following proposition, this implies that $\Fuk M$ is an $A_\infty$-category.

\begin{proposition}
Linear $A_\infty$-categories can be identified with categories over the operad $C_*^\cell(K) \coloneqq \bigl(C_*^\text{cell}(K_r)\bigr)_{r\geq 1}$.
\end{proposition}

\begin{proof}[Proof of the backward direction]
Let us explain a recipe to construct from a category $C$ over $\bigl(C_*^\text{cell}(K_r)\bigr)_{r\geq 1}$ a linear $A_\infty$-category $\sA$.

We define the data of $\sA$ as follows.
\begin{itemize}
\item
The objects of $\sA$ are the same as those of $C$.

\smallskip

\item
We define the unary operation $\mu_1\colon \hom(X_0,X_1)\to\hom(X_0,X_1)$ to be the differential on the chain complex $\hom(X_0,X_1)$.

\smallskip

\item
For $r \geq 2$, we define the $r$-ary operation
\begin{align}
\mu_r
\colon
\hom(X_0,X_1) \otimes \cdots \otimes \hom(X_{r-1},X_r)
\to
\hom(X_0,X_r)
\end{align}
to be the result of feeding the fundamental class $[K_r]$ into the first slot of the operation
\begin{align}
\label{eq:action_of_K_r}
\varphi_r
\colon
C_*^\cell(K_r)
\otimes
\hom(X_0,X_1) \otimes \cdots \otimes \hom(X_{r-1},X_r)
\to
\hom(X_0,X_r).
\end{align}
\end{itemize}

The only thing that we need to check is that these operations satisfy the $A_\infty$-equations.
To see this, note that since \eqref{eq:action_of_K_r} is a chain map, the following equation holds:
\begin{align}
\label{eq:K_r-action_chain_map}
&\pm\varphi_r\bigl(\partial[K_r],x_1,\ldots,x_r\bigr)
+
\sum_{1 \leq i \leq r} \pm\varphi_r\bigl([K_r],x_1,\ldots,x_{i-1},\partial x_i,x_{i+1},\ldots,x_r\bigr)
\nonumber
\\
&\hspace{3.5in}
=
\pm\partial\varphi_r\bigl([K_r],x_1,\ldots,x_r\bigr).
\end{align}
Using the recursive structure of the associahedra, together with the coherences satisfied by the operad structure of $\bigl(C_*^\cell(K_r)\bigr)_r$ and the action of this operad on $C$, we see that \eqref{eq:K_r-action_chain_map} yields the following equation:
\begin{align}
&\sum_{{2 \leq a \leq r-1}
\atop
{1 \leq i \leq r-a+1}}
\pm\mu_{r-a+1}\bigl(x_1,\ldots,x_{i-1},\mu_a(x_i,\ldots,x_{i+a-1}),x_{i+a},\ldots,x_r\bigr)
\nonumber
\\
&\hspace{0.75in}
+
\sum_{1 \leq i \leq r}
\pm\mu_r\bigl(x_1,\ldots,x_{i-1},\mu_1(x_i),x_{i+1},\ldots,x_r\bigr)
\nonumber
\\
&\hspace{3.5in}
=
\pm\mu_1(\mu_r(x_1,\ldots,x_r)).
\end{align}
Rearranging this, we obtain the $r$-th $A_\infty$ equation:
\begin{align}
\sum_{{1 \leq a \leq r}
\atop
{1 \leq i \leq r-a+1}}
\mu_{r-a+1}\bigl(x_1,\ldots,x_{i-1},\mu_a(x_i,\ldots,x_{i+a-1}),x_{i+a},\ldots,x_r\bigr)
=
0.
\end{align}
\end{proof}

\begin{remark}
\label{rem:Floer_theory_break_symmetry_mu_1}
In the above formulation of $A_\infty$-algebras as algebras over the $A_\infty$ operad, we separated the differential $\mu_1$ from the remainder of the operations, and chose to consider it as part of the structure of the target category (of chain complexes), rather than as one of the operations indexed by the operad.
This perspective is unnatural from the point of view of Floer theory because the differential is constructed, like all other operations, by a count of holomorphic curves, but it succeeds precisely because the differential is the first operation (in terms of parity) which we consider.
\null\hfill$\triangle$
\end{remark}

\subsection{Anomaly in Lagrangian Floer theory}
\label{ss:anom-lagr-floer}

On general (closed) symplectic manifolds, the construction of Fukaya $A_\infty$-categories requires significantly more work than indicated above, because the possible presence of holomorphic discs with boundary on a single Lagrangian \emph{obstructs} the equation $\mu_1^2=0$, in the sense that there is an element $\mu_0$ so that the vanishing of the differential is replaced by the equation
\begin{equation}
\label{eq:first_curved_equation}
\mu_1^2(x)
=
\mu_2(x,\mu_0) - \mu_2(\mu_0,x).
\end{equation}
Geometrically, this equation arises from an analysis of degenerations of moduli spaces of holomorphic strips with boundary on a pair $(L_1,L_2)$ of Lagrangians.
As indicated in Figure \ref{fig:first_curved_equation}, there are three boundary components to these moduli spaces; the first, corresponding to energy concentration taking place at a sequence of points which escape the strip along the end (breaking of strips), gives rise to the term $\mu_1^2$, while the two terms on the right hand side of \eqref{eq:first_curved_equation} arise when energy concentration takes place along one of the two boundary components of the strip.

In \cite{fooo_1,fooo_2}, Fukaya, Oh, Ohta, and Ono encoded the full structure of operations arising from the count of holomorphic discs with arbitrary numbers of marked point on a general Lagrangian submanifold into the structure of \emph{filtered gapped curved $A_\infty$-algebra,} and explained how to extract from this data an (ordinary) $A_\infty$-category, which informally consists of ways of (algebraically) correcting the differential so that the curvature vanishes.
We shall presently define this notion from an operadic perspective, by introducing the \emph{curved $A_\infty$ operad} after noting two essential difficulties:
\begin{enumerate}
\item
Unlike the case of ordinary $A_\infty$-categories, the curved $A_\infty$ operad which we shall introduce does not arise from an operad in the category of spaces.

\smallskip

\item
The notion of filtration, which is essential for the theory of curved algebras to produce meaningful answers, requires a technical extension of the notions of the theory of operads to a context where there is an additional label which, from the point of view of Floer theory, records the energy.
\end{enumerate}
\begin{remark} \label{rem:stack-operad}
To elaborate further, the lack  of a space-level model for the curved $A_\infty$ operad can be seen as follows: in \eqref{eq:A-infinity_equations}, the differential $\mu_1$ has degree $1$, while the product $\mu_2$ has degree $0$.
This forces the curvature term $\mu_0$ to have degree $2$.
However, because of our cohomological conventions, operations indexed by a manifold of dimension $k$ have degree $-k$, so this operation is associated to a space of dimension $-2$.
There is a ready explanation for this negative dimension: in the generic situation, the interior of the moduli spaces that define $\mu_0$ are \emph{quotients} of the space of maps with domain a disc with a boundary marked point by the automorphism of the disc preserving this point.
Since any two discs with a boundary marked point are biholomorphic, the abstract space of such discs (i.e.\ in the absence of a target symplectic manifold) is a point, but it should more properly be considered as the \emph{stack} quotient of this point by the automorphism group.
Since the group of biholomorphic automorphisms of the upper-half plane is easily seen to be a contractible $2$-dimensional Lie group, we see that the natural dimensional associated to it is indeed $-2$.

The above discussion suggests that we should construct the curved $A_\infty$ operad in the category of topological stacks.
The first step in this construction was performed by Lurie and Tanaka \cite{lurie_tanaka}, who revisited the formulation of differentials in Floer and Morse theory from this perspective by constructing a topological stack which encodes the differential at the equation $d^2=0$.
\null\hfill$\triangle$
\end{remark}

\begin{remark}
The necessity of imposing the additional structure of a filtration to obtain a good theory of curved categories is more difficult to justify geometrically, but can be algebraically justified as follows: one can associate to each curved category a \emph{category of modules}, which in the special case of trivial curvature recovers the usual notion.
However, it turns out that whenever the curvature is a non-trivial element, this category is completely trivial.
This can be proved by filtering the relevant morphism complexes so that the differential on the associated graded group depends only on the curvature, and using the fact that, for any vector space $V$, equipped with an \emph{non-zero} element $v$, the complex
\begin{equation}
V
\to
V \otimes V
\to
V \otimes V \otimes V
\to
\cdots, 
\end{equation}
with differential $V^{\otimes k} \to V^{\otimes k+1}$ given by the alternate sum of inserting $v$ at the $i$th position, is acyclic.
Heuristically, and as explained in \cite{positselski_weakly_curved_a-infinity}, this is a consequence of the fact that the formalism of $A_\infty$-categories is defined in such a way that operations with smaller number of inputs dominate, but the structure associated to the element $\mu_0$ is too trivial for any information to survive if we allow it to dominate.

A filtration resolves this issue as follows: by introducing a norm on the underlying graded vector spaces, one may require that each operation $\mu_k$ be expressed as a sum of contribution of decreasing norm.
By requiring that all contributions to $\mu_0$ have norm strictly smaller than unity, we ensure that the lowest order contributions of the other operations (in particular, of $\mu_1$) dominate.
In this way, we obtain an analogous theory to the one for ordinary algebras, and are able to formulate all structures in terms of the homotopy of theory of \emph{filtered chain complexes}, and we must take these norms into account when introducing, for example, the category of modules.
\null\hfill$\triangle$
\end{remark}

We now proceed to give the definition of the curved operad in the category of chain complexes:
for each strictly positive integer $r$, define
\begin{equation}
\cC_0(r)
\coloneqq
C_*^\cell(K_r),
\end{equation}
and set $\cC_0(0) \coloneqq 0$ and $\cC_0(1) \coloneqq 0$.
We recall that the underlying vector space of this chain complex is a direct sum indexed by the topological type $T$ of stable discs with $r+1$ marked points, 
\begin{equation}
C_*^\cell(K_r)
\simeq
\bigoplus_{T} \ro_{T}
\end{equation}
with the graded $\ro_{T}$ given by the orientation line of the associated product of Stasheff associahedra indexed by the components of this topological type, which we label $(v_1, \ldots, v_{d_T})$:
\begin{equation}
\ro_{T}
\coloneqq
\ro_{K_{v_1}} \otimes \cdots \otimes  \ro_{K_{v_{d_T}}}.
\end{equation}
Moreover, the differential on this complex can be naturally written in terms of choosing an edge to collapse.

We extend this construction to that of a chain complex $\cC_\lambda(r)$, whose underlying graded vector space is an (infinite) direct sum indexed by the topological type of pre-stable discs $\Sigma$, with $r+1$ boundary marked points, which are labelled by non-negative real numbers, so that the following condition holds:
\begin{center}
\it The label of any unstable component is strictly positive.
\end{center}
We associate to each such (labelled) topological type $T$ the graded vector space
\begin{equation}
\cC_T(r)
\coloneqq
\bigotimes_{v} \ro_{K_{v}},
\end{equation}
where the tensor product, as before, is indexed by the components $v$ of the underlying topological type, and the Stasheff associahedron $K_{v}$ is the one associated to having number of inputs equal to valence of $v$.
Here, we have formally set
\begin{equation}
\ro_{K_0}
\coloneqq
\ro^{-1}_{\Aut(D^2, 1)},
\qquad
\ro_{K_1}
\coloneqq
\ro^{-1}_{\Aut(D^2, \pm 1)},
\end{equation}
as would be expected from Remark \ref{rem:stack-operad}.
The key point is that the direct sum
\begin{equation}
\cC_{\lambda}(r)
\coloneqq
\bigoplus_{T}  \cC_T(r)
\end{equation}
can naturally be equipped with a differential which can be expressed in terms of collapsing an edge.

Finally, we take the direct sum of all these chain complexes, and define
\begin{equation}
\cC(r)
\coloneqq
\bigoplus_{\lambda \in [0,\infty)} \cC_{\lambda}(r).
\end{equation}
Displaying each topological type as a tree, the concatenation of trees defines the structure map
\begin{equation}
  \circ\colon \cC_{\lambda}(r) \otimes \cC_{\lambda_1}(s_1) \otimes\cdots\otimes\cC_{\lambda_r}(s_r) \to \cC_{\lambda + \lambda_1+\cdots+\lambda_r}(s_1 + \cdots + s_r),
\end{equation}
and the direct sum of these operations over all weights yields the desired operation.

This leads to the following operadic notion:

\begin{definition}
A \emph{gapped, filtered, curved $A_\infty$-algebra} is an algebra in over $\cC$ with the property that there is a discrete subset $\Gamma$ of $[0,\infty)$ so that the action of $ \cC_{\lambda}(r) $ vanishes unless $\lambda$ lies in $\Gamma$.
\null\hfill$\triangle$
\end{definition}

To recover Fukaya, Oh, Ohta, and Ono's notion from this construction, one tensors the algebra (over the ground ring), with the Novikov ring, denotes by $\mu^r_\lambda$ the operation associated to the unique generator of $ \cC_{\lambda}(r) $ associated to a tree with no internal edges (after trivialising the corresponding line), and defines the operation
\begin{equation}
  \mu^k = \sum T^{\lambda} \mu^k_\lambda .
\end{equation}

The main result of Fukaya, Oh, Ohta, and Ono \cite{fooo_1} is that a closed embedded Lagrangian $L$ in a closed symplectic manifold determines a curved filtered $A_\infty$-algebra, denoted $CF^*(L,L)$ in the above sense.
Before we indicate how one can extract an ordinary $A_\infty$-category from this data, we discuss a special situation (related to Remark \ref{rmk:geometric_settings}) in which the theory simplifies:  when the Maslov class and the symplectic class in $H^2(M,L)$  are positively proportional
\begin{equation}
  c_1(M,L) = m [\omega]  
\end{equation}
with $m$ strictly larger than $1$, then one can construct $CF^*(L,L)$ so that the curvature element for each energy (which in our formulation is the element of $\cC_{\lambda}(0)$ associated to the unique tree with no interior edge), is a multiple $\mu_0^L \eqqcolon m_0(L)\cdot 1_L$ of the identity.
In this case, one can essentially ignore the curvature term, and obtain an $A_\infty$-algebra in the usual sense.

In general, the procedure to obtain an $A_\infty$-category from such a curved algebra is much more complicated.
Writing $A$ for the underlying module over the Novikov ring, the first step is to consider elements $b \in A$ which satisfy the Maurer-Cartan equation
\begin{equation}
  \mathfrak{P} \cdot \id_{A}  = \mu^0 + \mu^1(b) + \mu^2(b,b) + \mu^3(b,b,b) + \cdots,
\end{equation}
for some scalar $\mathfrak{P}$ (in the Novikov ring), which is called the \emph{potential value}.
Such solutions are called \emph{weak bounding cochains}, and the algebra is called \emph{weakly unobstructed} if a solution exists (the adjective \emph{weak} is dropped if $\mathfrak{P}$ vanishes).
\begin{lemma}
The collection of weak bounding cochains with a given potential value are objects of an $A_\infty$-category in which the morphisms from $b$ to $b'$ are given by a chain complex with underlying module $A$ and differential
\begin{equation}
x
\mapsto
\mu^2(b,x) + \mu^2(x,b') + \mu^3(b,b,x) + \mu^3(b,x,b') + \mu^3(x,b',b') + \cdots.
\end{equation}
\end{lemma}

%%% Local Variables:
%%% mode: latex
%%% TeX-master: "functoriality_in_categorical_symplectic_geometry"
%%% End:

\section{Quilted Floer theory and functors from Lagrangian correspondences}
\label{s:quilted_floer}

\begin{center}
\fbox{\parbox{0.9\columnwidth}{
{\bf Default geometric hypotheses:}
In \S\S\ref{sec:corr-sympl-topol}--\ref{ss:reduction}, we assume our symplectic manifolds and Lagrangians are closed and monotone, and that the Lagrangians have minimal Maslov index 3.
By default, composition $L_1 \circ L_{12}$ of Lagrangians with Lagrangian correspondences is assumed to be cut out transversely and to result in an embedded Lagrangian.
In \S\ref{ss:fukaya} we relax these hypotheses and consider general closed symplectic manifolds.
}}
\end{center}

\smallskip

As we alluded to in the introduction, much of the original development of the theory of Fukaya categories arose in an attempt to formulate a precise statement of Kontsevich's Homological Mirror Symmetry Conjecture \cite{kontsevich_hms}.
In its initial formulation, this conjecture asserted that the mirror phenomena discovered by string theorists starting with \cite{candelas_delaossa_green_parkes}, which relate the enumerature geometry of a complex Calabi--Yau $3$-fold $X$ with period integrals of a conjecturally-existing mirror Calabi--Yau $3$-fold $Y$, are consequences of an equivalence between the (then conjecturally-existing) Fukaya category $\Fuk X$ on of one side of the mirror correspondence, and the \emph{bounded derived category of coherent sheaves} $D^b(Y)$  of the other.

\subsection{A brief overview of derived categories of coherent sheaves}
\label{sec:brief-overview-db}

The classical construction, following Verdier \cite{verdier_asterisque}, describes $D^b(Y)$ as a triangulated category.
This version of $D^b(Y)$ is a category whose objects are complexes of coherent sheaves, and whose morphisms $\Hom_*(-,-)$ are graded complex vector spaces obtained by taking the cohomology of maps between resolutions of these complexes.
In the smooth setting, every complex is equivalent to a complex of holomorphic vector bundles, and one can compute morphisms as the cohomology groups of the associated maps of complexes of smooth vector bundles, equipped with the differential induced by the holomorphic structure on the source and the target.

In addition to the data of linear composition of morphisms, a triangulated category is equipped with the additional datum of a choice of \emph{exact triangles}
\begin{align}
\xymatrix{
F \ar[rr] && \ar[ld] G
\\
& \ar[lu]^{[1]} H, &
}
\end{align}
each of which induces a long exact sequence on morphisms
\begin{equation} \label{eq:LES-exact-triangle}
  \cdots  \to \Hom_*(-,F) \to \Hom_*(-,G) \to \Hom_*(-,H) \to  \Hom_{*+1}(-,F) \to \cdots   
\end{equation}
for any choice of input.
The set of exact triangles satisfies various axioms which we will not discuss.

As the use of derived categories grew in algebraic geometry, the formulation of $D^b(Y)$ as a triangulated category came to be seen as an increasingly cumbersome technicality.
The issue lies in \eqref{eq:LES-exact-triangle}, which shows that $H$ is specified up to isomorphism by the arrow $F \to G$, but does not specify a particular choice of $H$.
This causes particular difficulties when proving gluing results for derived categories.

The solution for this problem is usually to \emph{enhance} the structure of the derived category to that of a \emph{differential graded category}.
In this context, being an exact triangle can be formulated as a property (of a triple of morphisms and null-homotopies for their compositions), which enables gluing.
(This was first proposed by Bondal--Kapranov in \cite{bondal_kapranov_enhanced}.
See also \cite{drinfeld_dg} for a more modern example of this approach.)
Since a differential graded category is a particular example of an $A_\infty$-category, the formulation of Kontsevich's mirror conjecture is usually made using these enhanced categories.

\subsection{Correspondences in algebraic geometry}
\label{sec:corr-sympl-algebr}

A notable deficiency of the categorical formulation of mirror symmetry is that, while a map of schemes induces pullback and pushforward functors on their derived categories, there is no analogous functoriality of Fukaya categories with respect to maps of symplectic manifolds (e.g.\ maps $f : M \to N$ with $f^* \omega_N = \omega_M$).
This can be seen for examples as straightforward as an open inclusion $D^2 \hra \bCP^1$: the equator $\bRP^1$ is a nontrivial Lagrangian in the target, but not in the source, which contradicts an inclusion of unital $A_\infty$-categories $\Fuk D^2 \hra \Fuk S^2$.
\begin{remark}
 The only classes of symplectic maps which the authors know to induce functors on Fukaya categories are (i) unbranched coverings \cite{seidel_HMS_genus-two_curve}, and (ii) codimension $0$ inclusions with contact boundary, whose complement is exact \cite{abouzaid_seidel_open_string}.
 Both of these classes of maps include the class of symplectomorphisms, and fall within the framework which we shall presently describe (c.f.\ \cite[Theorem 1.4]{gao2018functors}).
\null\hfill$\triangle$
\end{remark}

Somewhat surprisingly, one way to arrive at a good notion of functoriality for Fukaya categories is to investigate the coherent sheaf side more carefully: the symplectic automorphism group of a symplectic $2$-torus includes as a subgroup the group $SL_2(\bZ)$ of modular transformations, which therefore acts on its Fukaya category.
One of the first test cases \cite{polishchuk_zaslow} of mirror symmetry identifies the mirror as an elliptic curve --- but there is no elliptic curve with such a large automorphism group, so derived categories must admit many more automorphisms than those arising from automorphisms of the underlying schemes.

In this example, the missing automorphisms can be recovered from the theory of \emph{Fourier--Mukai transforms} which goes back at least to \cite{mukai_duality}, which associates to a coherent sheaf $\sP$ on $X \times Y$  the functor
\begin{align}
\Phi_\sP
\colon
D^b(X) \to D^b(Y),
\qquad
\sE^\bullet
\to
p_*(q^*\sE^\bullet \otimes \sP)
\end{align}
obtained by pulling a sheaf on $X$ back to the product $X \times Y$ along the projection to the first factor, then tensoring the result with $\sP$, and finally pushing forward to $Y$ along the second projection.
The fundamental theorem of Fourier--Mukai theory is:

\begin{theorem}[Theorem 2.2, \cite{orlov_K3}]
Fix smooth projective varieties $X$ and $Y$, and suppose that $F\colon D^b(X) \to D^b(Y)$ is a fully faithful exact functor.
If $F$ has both a left and a right adjoint, then there exists an object $\sP \in D^b(X\times Y)$ such that $F$ is naturally isomorphic to $\Phi_\sP$.
Moreover, $\sP$ is uniquely determined, up to isomorphism.
\null\hfill$\square$
\end{theorem}

It would be difficult to overstate the importance of the above result, which forms the basis of all substantial results in the subject.
In fact, the result is quite a bit stronger, and consists of a natural equivalence between the derived category of coherent sheaves on $X \times Y$ and the category of functors from $D^b(X)$ to $D^b(Y)$.
(See \cite[Theorem 1.1]{lunts_schnuerrer} for a variant of this statement.)
We give two examples:

\begin{theorem}[Proposition 10.10, \cite{huybrechts_FM}; see also Theorem 3.3, \cite{orlov_K3}]
Let $X$ and $Y$ be K3 surfaces.
There exists a linear, exact equivalence between their derived categories $D^b(X)$ and $D^b(Y)$ if and only if there exists a Hodge isometry $\wt H(X; \bZ) \sr{\simeq}{\to} \wt H(Y; \bZ)$.
\null\hfill$\square$
\end{theorem}

\begin{theorem}[Theorem 4.3, \cite{orlov_blowup}]
Fix a smooth projective variety $X$ and a codimension-$c$ smooth subvariety $Y$ with $c \geq 2$, and denote by $\wt X$ the blowup of $X$ along $Y$.
Then $D^b\bigl(\wt X\bigr)$ admits a semiorthogonal decomposition of the following form:
\begin{align}
D^b\bigl(\wt X\bigr)
=
\Bigl\langle
D^b(X),
\underbrace{D^b(Y), \ldots, D^b(Y)}_{c-1}
\Bigr\rangle.
\end{align}
\null\hfill$\square$
\end{theorem}

We conclude this subsection with one more basic property of Fourier--Mukai transforms:

\begin{proposition}[Proposition 5.10, \cite{huybrechts_FM}; see also its original statement as Proposition 1.3 in \cite{mukai_duality}]
\label{prop:FM_commutativity}
Fix Fourier--Mukai kernels $\sP \in D^b(X\times Y)$ and $\sQ \in D^b(Y\times Z)$, and define $\sR \in D^b(X\times Z)$ by
\begin{align}
\sR
\coloneqq
{\pi_{XZ}}_*\bigl(\pi_{XY}^*\sP \otimes \pi_{YZ}^*\sQ\bigr).
\end{align}
Then the composition $\Phi_\sQ \circ \Phi_\sP$ is naturally isomorphic to $\Phi_\sR$.
\null\hfill$\square$
\end{proposition}

\noindent
We invite the reader to consult \cite{huybrechts_FM} for a survey of this beautiful subject.

\subsection{Correspondences in symplectic topology}
\label{sec:corr-sympl-topol}

By comparing the Fukaya category with the derived category of coherent sheaves, we might hope for a machine that associates to a Lagrangian correspondence, i.e.\ an object $\Lambda$ of $\Fuk(M^- \times N) \coloneqq \Fuk(M\times N, (-\omega_M)\oplus \omega_N)$, some sort of functor $\Phi_\Lambda\colon \Fuk M \to \Fuk N$.
(See Remark \ref{rmk:why_minus} for an explanation of the minus sign on $M^-$.)
In a remarkable series of five papers published between 2010 and 2018, Wehrheim and Woodward (and in the case of one of these papers, Ma'u) developed just such a construction.
Our main goal in the current section is to explain the two theorems that form the culmination of Wehrheim--Woodward's work on pseudoholomorphic quilts.

The first theorem asserts that one can extend the Fukaya category in such a way that a Lagrangian correspondences from $M_0$ to $M_1$ defines an $A_\infty$-functor from the category associated to $M_0$ to that associated to $M_1$.
(We will explain below why Ma'u--Wehrheim--Woodward needed this extension.)

\begin{theorem}[Theorem 1.1, \cite{mww}]
\label{thm:MWW_functors}
Suppose that $M_0, M_1$ are symplectic manifolds satisfying standard monotonicity hypotheses.
Given an admissible Lagrangian correspondence $L_{01} \subset M_0^- \times M_1$ equipped with a brane structure, there exists an $A_\infty$-functor
\begin{align}
\Phi^\#_{L_{01}}
\colon
\Fuk^\#(M_0)
\to
\Fuk^\#(M_1)
\end{align}
which acts on objects by appending $L_{01}$ to a generalized Lagrangian $\ul L_0 \in \Fuk^\#(M_0)$, and on morphisms by counting quilted disks with two patches and boundary marked points.
\null\hfill$\square$
\end{theorem}

The second result is summarized in the phrase \emph{``composition commutes with categorification''}:

\begin{theorem}[Theorem 1.2, \cite{mww}]
Suppose that $M_0, M_1, M_2$ are symplectic manifolds satisfying standard monotonicity hypotheses.
Let $L_{01} \subset M_0^- \times M_1$, $L_{12} \subset M_1^- \times M_2$ be admissible Lagrangian correspondences with spin structures and gradings such that $L_{01} \circ L_{12}$ is smooth, embedded by $\pi_{02}$ in $M_0^- \times M_2$, and admissible.
Then there exists a homotopy of $A_\infty$-functors
\begin{align}
\Phi^\#_{L_{12}} \circ \Phi^\#_{L_{01}}
\simeq
\Phi^\#_{L_{01} \circ L_{12}}.
\end{align}
\null\hfill$\square$
\end{theorem}

\noindent
(Compare this with Proposition \ref{prop:FM_commutativity}, which is the analogous result for Fourier--Mukai transforms.)

The crucial tool that Wehrheim and Woodward used to establish these results was their theory of \emph{pseudoholomorphic quilts}, which we will briefly survey later in this section.

\begin{remark}
\label{rmk:why_minus}
The essential reason that we have to introduce the symplectic manifold $M^-$, which has no analogue in algebraic geometry, is that the derived category of coherent sheaves is equipped with a natural equivalence to its opposite category (in which the direction of morphisms are reversed), given by assigning to a sheaf its dual (i.e.\ the sheaf of morphisms to the structure sheaf).
In contrast, the Fukaya category has no such duality isomorphism, in part because there is no canonical object that plays the role of the structure sheaf.
One way to produce such a duality isomorphism is to exploit an anti-symplectic involution on $M$, as discussed by Casta\~{n}o-Bernard--Matessi--Solomon in  \cite{castano-bernard_matessi_solomon}, who explored this structure in the context of mirror symmetry.
\null\hfill$\triangle$
\end{remark}

\subsection{Approaching Ma'u--Wehrheim--Woodward's $A_\infty$-functor}
\label{ss:approaching_Phi}

We begin this subsection by approaching Ma'u--Wehrheim--Woodward's construction of the functors $\Phi_{L_{12}}^\#$ in Theorem \ref{thm:MWW_functors} via the Operadic Principle of \S\ref{ss:HF}.

There is a very natural way to try and define a functor $\Phi_{L_{12}}\colon \Fuk M_1 \to \Fuk M_2$ on the level of objects.
Indeed, we can view a Lagrangian $L_1$ in $M_1$ as a correspondence from $\pt$ to $M_1$, and then compose $\pt \sr{L_1}{\lra} M_1$ and $M_1 \sr{L_{12}}{\lra} M_2$ as correspondences --- which is to say, to form
\begin{align}
\Phi_{L_{12}}(L_1)
\coloneqq
L_1 \circ L_{12}
&\coloneqq
\pi_{M_2}(L_1 \times_{M_1} L_{12})
\\
&\coloneqq
\pi_{M_2}\bigl((L_1 \times L_{12}) \cap (\Delta_{M_1} \times M_2)\bigr).
\nonumber
\end{align}

\noindent
This idea is supported by an observation due to Guillemin--Sternberg:

\begin{proposition}[Theorem 4, \cite{guillemin_sternberg}]
\label{prop:guillemin_sternberg}
If the submanifolds $L_1 \times L_{12}$ and $\Delta_{M_1} \times M_2$ of $M_1 \times M_1^- \times M_2$ intersect transversely, then $\pi_{M_2}\colon M_1 \times M_1^- \times M_2 \to M_2$ restricts to a Lagrangian immersion of $L_1 \times_{M_1} L_{12}$ into $M_2$.
\null\hfill$\square$
\end{proposition}

\noindent
When the hypotheses of Proposition \ref{prop:guillemin_sternberg} are satisfied, we say that $L_1\circ L_{12}$ have \emph{immersed composition}.
If $L_1$ and $L_{12}$ satisfy the additional hypothesis that $L_1\circ L_{12}$ is an embedded submanifold of $M_2$, then we say that they have \emph{embedded composition}.

There are two apparent issues with our attempted definition of $\Phi_{L_{12}}$ on the level of objects: for one thing, the composition $L_1 \circ L_{12}$ may not even be immersed, hence may not produce a Lagrangian submanifold, indeed a submanifold at all, of $M_2$.
For another, the Fukaya category typically consists of \emph{embedded} Lagrangians.
Regarding the first issue, Wehrheim--Woodward showed in \cite[Proposition 2.2.1]{wehrheim_woodward_geometric_composition} that composition of Lagrangian correspondences can be made immersed after applying a Hamiltonian isotopy in each factor.
The second issue is surmountable via the theory of ``immersed Fukaya categories'', which were initially studied by Akaho and Joyce \cite{akaho_joyce} in an early extension of Fukaya, Oh, Ohta, and Ono's work.

We will return later to the issue of the definition of $\Phi_{L_{12}}$ on objects.
Before that, we consider its action on morphisms.
We expect this part of the definition of $\Phi_{L_{12}}$ to be based on counts of some sort of pseudoholomorphic objects, and therefore we turn to the Operadic Principle.
This principle suggests that we should follow these steps:
\begin{enumerate}	
\item
Find an operadic construction that encodes the structure of an $A_\infty$-functor.

\smallskip

\item
Find moduli spaces of Riemann surfaces that realize this operadic construction.

\smallskip

\item
Complete our definition-in-progress of $\Phi_{L_{12}}$ by counting pseudoholomorphic curves whose domains are the Riemann surfaces from the previous step.
\end{enumerate}

That is nearly what Wehrheim and Woodward did.
We begin by describing an answer to (1) and (2) simultaneously.
(We could treat them separately, but this would require introducing combinatorial objects that would not be enlightening.
Instead, we will conflate the combinatorial and topological instantiations of the multiplihedra.)

\begin{definition}
For every $r \geq 1$, the \emph{$(r-1)$-dimensional multiplihedron} $J_r$ (defined in \cite{saneblidze_umble:permutahedra_multiplihedra_associahedra}) is the compactified moduli space of configurations of a vertical line in the right half plane $\bH^0$ and $r$ points on the imaginary axis, up to real translations and dilations.
Alternately, we can visualize such a configuration as a disk decorated with $r$ boundary marked points and one interior circle tangent to the boundary at a single point.
We depict both presentations of a point in $J_3$ in the following figure:
\begin{figure}[H]
\centering
\def\svgwidth{0.3\columnwidth}
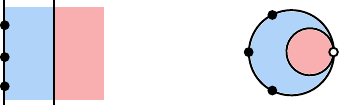
\caption{
\label{fig:two-patch_disks}
}
\end{figure}
\null\hfill$\triangle$
\end{definition}
\noindent
Topologically, $J_r$ is a compact $(r-1)$-dimensional manifold with boundary.
For instance, $J_1$ is a singleton because any two circles tangent to the disc at $1$ are related by a unique M\"obius transformation that fixes $\pm 1$.
$J_r$ can be realized as a convex polytope, but this polytopal structure is not directly relevant for our current purposes.
Below, we illustrate the 1- and 2-dimensional instances $J_2$ and $J_3$:

\begin{figure}[H]
\centering
\def\svgwidth{1.0\columnwidth}
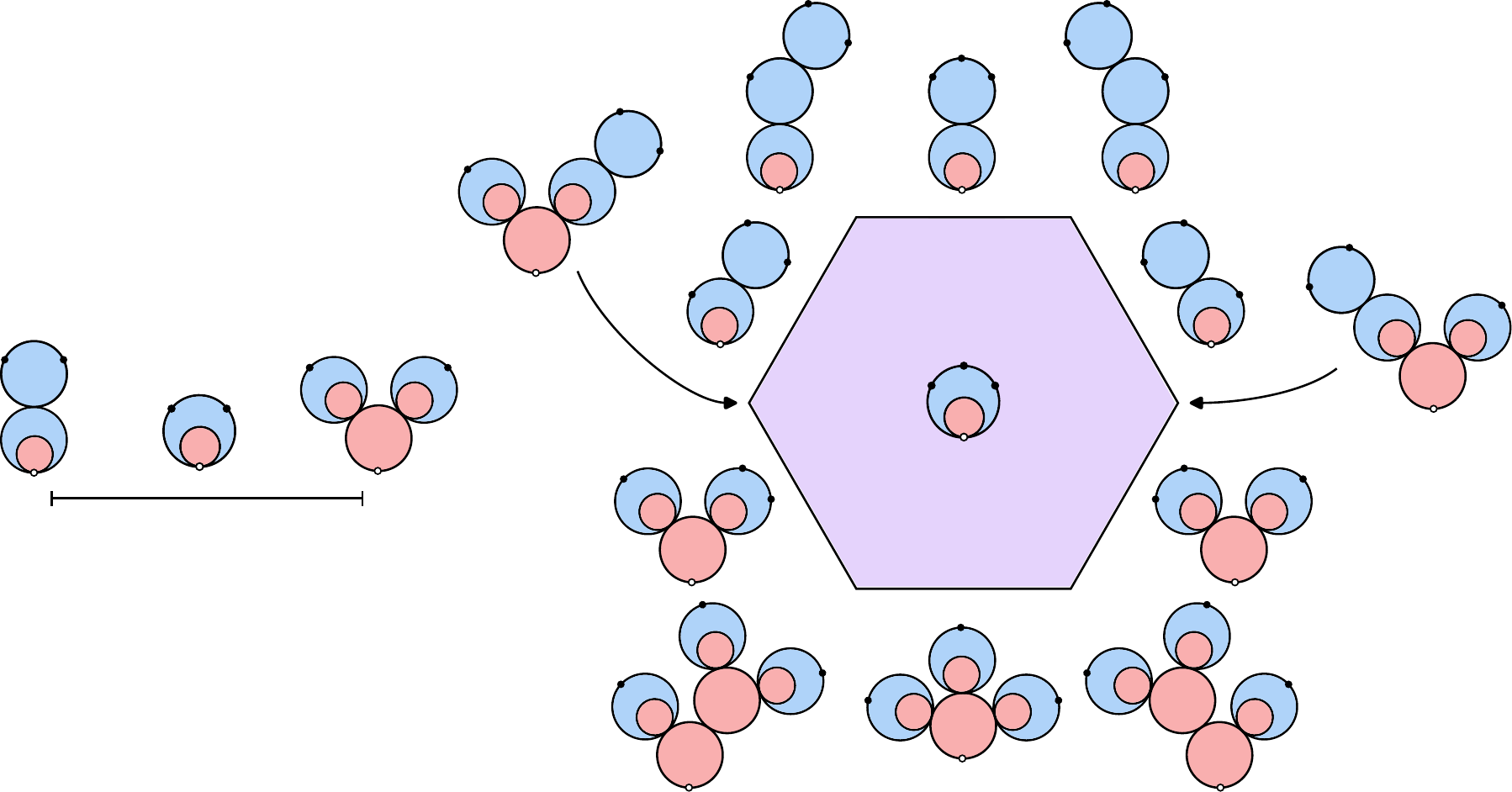
\caption{
The multiplihedra $J_2$ and $J_3$.
\label{fig:J2_J3}
}
\end{figure}

To see why the multiplihedra enable us to apply the Operadic Principle, we must understand what sort of operadic structure $(J_r)$ supports.
It turns out that the answer to this question is that the multiplihedra $(J_r)$ form a \emph{bimodule} over the operad $(K_r)$ of associahedra \cite{laplante-anfossi_mazuir:multiplihedra}.
This bimodule structure consists of the following:
\begin{itemize}
\item[]
{\bf (left module)}
A map $K_r \times J_	{a_1} \times \cdots \times J_{a_r} \to J_{a_1+\cdots+a_r}$, which is defined like so:

\begin{figure}[H]
\centering
\def\svgwidth{0.6\columnwidth}
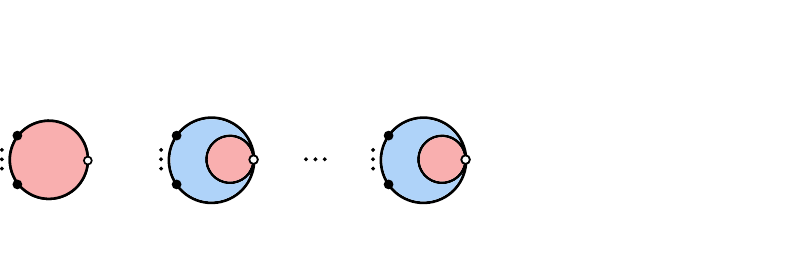
\caption{}
\end{figure}

\item[]
{\bf (right module)}
A map $J_a \times K_{r_1} \times \cdots \times K_{r_a} \to J_{r_1+\cdots+r_k}$, which is defined like so:

\begin{figure}[H]
\centering
\def\svgwidth{0.6\columnwidth}
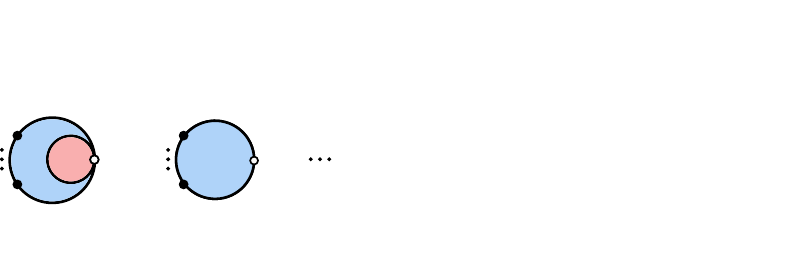
\caption{}
\end{figure}

\end{itemize}
\begin{remark}
There is a fundamental asymmetry between the left and right module structures.
Namely, we can naturally extend the polyhedra $K_r$ to the case $r=1$ by declaring that $K_1$ is a point; this element then corresponds to the identity map of each $A_\infty$-algebra.
In this way, the right module maps can be extracted from operations $\circ_i : J_a \times K_r \to J_{a+r-1}$, for $r \geq 2$, associated to gluing a disc as a single input labelled by an integer $1 \leq i \leq a$.
In Figure \ref{fig:J2_J3}, these operations label the top boundary strata of $J_3$ corresponding to the upper half of the hexagon.
On the other hand, while the moduli space $J_1$ is also a singleton, it does not correspond to a tautological operation at the level of algebra, and this is reflected in the fact that the codimension $1$ boundary strata of $J_a$ which correspond to these operations involve attaching quilted discs at every input of an element of $K_r$, as can be seen by inspecting the bottom boundary strata of $J_3$.

Applying $C_*^\cell$, we see that $C_*^\cell(J) \coloneqq \bigl(C_*^\cell(J_r)\bigr)$ is a bimodule over $C_*^\cell(K)$.
In fact, just as an $A_\infty$-category is the same thing as an $C_*^\cell(K)$-category, an $A_\infty$-functor is the same thing as a functor over $C_*^\cell(J)$.
(The notion of a functor over a bimodule over an operad is just what one might expect: in the present case, given categories $C, D$ over $C_*^\cell(K)$, a functor $F\colon C \to D$ over $C_*^\cell(J)$ consists of a map $\Ob C \to \Ob D$ and a collection of chain maps
\begin{align}
C_*^\cell(J_r)
\otimes
\hom(X_0,X_1)
\otimes
\cdots
\otimes
\hom(X_{r-1},X_r)
\to
\hom(F(X_0), F(X_r))
\end{align}
that satisfy a collection of coherences.)

The Operadic Principle now tells us that if we want to define a functor $\Phi_{L_{12}}$ from $\Fuk M_1$ to $\Fuk M_2$, we should construct moduli spaces of pseudoholomorphic curves whose domains are the decorated disks in Figure \ref{fig:two-patch_disks}.
This requires deviating from the conventional notion of a pseudoholomorphic curve, because the disks we would like to use as our domains are divided into two ``patches'', and because there are two possible targets, $M_1$ and $M_2$.
One of Wehrheim and Woodward's central insights was to ask for a map from one patch to $M_1$ and from the other to $M_2$, with these maps coupled along the interior ``seam'' by $L_{12}$.
Indeed, this is an instance of Wehrheim--Woodward's notion of a \emph{pseudoholomorphic quilt}, or simply a \emph{quilt}.
\null\hfill$\triangle$
\end{remark}

\subsection{Quilted Floer cohomology}
\label{ss:qHF}

As we described in the previous subsection, a key innovation of Wehrheim and Woodward was their notion of pseudoholomorphic quilts.
We will approach this general construction via the notion of \emph{quilted Floer cohomology}, which is an extension of Lagrangian Floer cohomology that is defined by counting rigid pseudoholomorphic quilts whose domain is a cylinder divided into strips by interior lines.
Just as ordinary Floer cohomology groups are the natural input and output for operations defined by counting pseudoholomorphic curves with boundary punctures, quilted Floer cohomology forms the natural input and output for operations defined by counting pseudoholomorphic quilts.

We begin by defining generalizations of Lagrangian submanifolds and correspondences.
For symplectic manifolds $M_0, M_1$, a \emph{generalized (Lagrangian) correspondence from $M_0$ to $M_1$} is a sequence
\begin{align}
\label{eq:generalized_correspondence}
\ul L
\coloneqq
\bigl(
M_0 = N_0
\:\sr{L_{01}}{\lra}\:
N_1
\:\sr{L_{12}}{\lra}\:
\cdots
\:\sr{L_{(r-1)r}}{\lra}\:
N_r = M_1
\bigr).
\end{align}
A \emph{cyclic generalized correspondence} is a generalized correspondence with $M_0 = M_1$.
Just as a Lagrangian is a correspondence from $\pt$ to a symplectic manifold, a \emph{generalized Lagrangian} is a generalized correspondence with $M_0 = \pt$.
When $r=1$, of course, these generalized notions reduce to their ordinary counterparts.

Floer cohomology associates to a pair of Lagrangians $L, K \subset M$ an Abelian group $HF^*(L,K)$.
From $L$ and $K$, we can produce the following cyclic generalized correspondence:
\begin{align}
\xymatrix{
\pt \ar@/^/[r]^L & \ar@/^/[l]^K M.
}
\end{align}
Quilted Floer cohomology extends $HF^*(L,K)$ to arbitrary cyclic generalized correspondences.

Indeed, fix a length-$r$ cyclic correspondence $\ul L$, as in \eqref{eq:generalized_correspondence}.
To define its quilted Floer cohomology, we first need to specify a tuple $\ul\delta \in (0,\infty)^r$ of ``widths''.
Having done so, the quilted Floer cohomology $HF^*_{\ul\delta}(\ul L)$ is defined to be the homology of a complex
\begin{align}
\bigl(
CF_{\ul\delta}(\ul L)
\coloneqq
\bK\langle \ul p\rangle_{\ul p \in\cI(\ul L)}, d_{\ul\delta}
\bigr),
\end{align}
where $\cI(\ul L)$ is the set of \emph{generalized intersection points}.
$\cI(\ul L)$ is defined by
\begin{align}
\label{eq:generalized_intersection_points}
\cI(\ul L)
\coloneqq
\bigl\{
(p_{01},\ldots,p_{(r-1)r})
\in
L_{01}\times\cdots\times L_{(r-1)r}
\:|\:
\pi_j(p_{(j-1)j})=\pi_j(p_{j(j+1)})
\bigr\},
\end{align}
where we require the condition to hold for all $1 \leq j \leq r$, and where we make the convention $L_{r(r+1)} \coloneqq L_{01}$.
(Indeed, from now on, we treat $j$ as a cyclic index living in $\bZ/r\bZ$.)
We require that $\cI(\ul L)$ is cut out transversally.
This is to say that the intersection
\begin{align}
(L_{01} \times \cdots \times L_{(r-1)r})
\cap
\sigma(\Delta_{M_0}\times\cdots\times\Delta_{M_r})
\end{align}
is transverse, where $\sigma$ is the obvious permutation of factors.

\begin{remark}
Wehrheim--Woodward do not assume this transversality --- rather, they show that they can always arrange for it to hold, by applying Hamiltonian perturbations.
In \cite[\S5]{wehrheim2010quilted}, they show that split Hamiltonian perturbations (that is, the result of perturbing each factor separately) are enough to ensure that $\cI(\ul L)$ is cut out transversally.
\null\hfill$\triangle$
\end{remark}

We will now encounter our first pseudoholomorphic quilt: to define the differential $d_{\ul\delta}$ on $CF^*(\ul L)$, we need to count certain \emph{quilted} Floer trajectories.
Specifically, for elements $\ul p^\pm$, we define the following moduli space:
\begin{gather}
\label{eq:qHF_moduli_space}
\cM_{\ul\delta}(\ul p^-,\ul p^+)
\coloneqq
\Bigl\{
\ul u \coloneqq \bigl(u_j\colon \bR\times[0,\delta_j] \to N_j\bigr)_{0\leq j\leq r}
\:\big|\:
(*1), (*2), (*3), (*4)
\Bigr\}/\sim,
\\
(*1):
\partial_s u_j + J_j(u_j)\partial_t u_j = 0
\:\forall\: j,
\nonumber
\qquad
(*2):
\bigl(u_j(s,\delta_j),u_{j+1}(s,0)\bigr)
\in
L_{j(j+1)}
\:\forall\: j,
\:\forall\: s \in \bR,
\nonumber
\\
(*3):
E(\ul u)
\coloneqq
\sum_{j=0}^r
\int_{\bR\times[0,\delta_j]}
u_j^*\omega_j
<
\infty,
\qquad
(*4):
\lim_{s\to\pm\infty}
u_j(s,-)
=
p_j^\mp
\:\forall\: j,
\nonumber
\end{gather}
where the quotient in the definition of $\cM(\ul p^-,\ul p^+)$ indicates that we identify two quilted Floer trajectories that differ by a simultaneous translation in the $s$-coordinate.
Conditions $(*1)$, $(*3)$, and $(*4)$ are familiar: they say that each map $u_j$ is pseudoholomorphic, has finite energy, and limits to $p_j^\mp$ as $s \to \pm\infty$.

Condition $(*2)$, which can be thought of a nonlocal boundary condition, is less familiar.
It is known as a \emph{seam condition}, because we can visualize the domain of a quilted Floer trajectory as the result of attaching together the domains of the $u_j$'s by identifying $(s,\delta_j)$ in the $j$-th rectangle with $(s,0)$ in the $(j+1)$-th.
Indeed, we can view $\ul u$ as a \emph{quilted map} whose domain is the cylinder $\bR \times (\bR/|\ul\delta|\bZ)$, where $|\ul\delta|$ denotes the sum $\delta_0+\cdots+\delta_r$.
This cylinder is divided into regions called \emph{patches}, the $j$-th of which being equipped with a map to $N_j$.
Specifically, for $0 \leq j \leq r$, the map from the $j$-th patch is
\begin{align}
\wt u_j
\colon
\bR \times [\delta_0+\cdots+\delta_{j-1}, \delta_0+\cdots+\delta_j]
\to
N_j,
\quad
\wt u_j(s,t)
\coloneqq
u_j(s,t-\delta_0-\cdots-\delta_{j-1}).
\end{align}
We depict a quilted cylinder like so:
\begin{figure}[H]
\centering
\def\svgwidth{0.375\columnwidth}
%% Creator: Inkscape 1.2 (dc2aeda, 2022-05-15), www.inkscape.org
%% Accompanies image file 'quilted_cylinder.pdf' (pdf, eps, ps)
\begingroup%
  \makeatletter%
  \providecommand\color[2][]{%
    \errmessage{(Inkscape) Color is used for the text in Inkscape, but the package 'color.sty' is not loaded}%
    \renewcommand\color[2][]{}%
  }%
  \providecommand\transparent[1]{%
    \errmessage{(Inkscape) Transparency is used (non-zero) for the text in Inkscape, but the package 'transparent.sty' is not loaded}%
    \renewcommand\transparent[1]{}%
  }%
  \providecommand\rotatebox[2]{#2}%
  \newcommand*\fsize{\dimexpr\f@size pt\relax}%
  \newcommand*\lineheight[1]{\fontsize{\fsize}{#1\fsize}\selectfont}%
  \ifx\svgwidth\undefined%
    \setlength{\unitlength}{155.17061056bp}%
    \ifx\svgscale\undefined%
      \relax%
    \else%
      \setlength{\unitlength}{\unitlength * \real{\svgscale}}%
    \fi%
  \else%
    \setlength{\unitlength}{\svgwidth}%
  \fi%
  \global\let\svgwidth\undefined%
  \global\let\svgscale\undefined%
  \makeatother%
  \begin{picture}(1,0.48522463)%
    \lineheight{1}%
    \setlength\tabcolsep{0pt}%
    \put(0,0){\includegraphics[width=\unitlength,page=1]{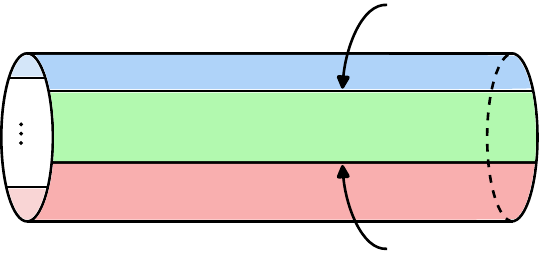}}%
    \put(0.72756329,0.46538761){\makebox(0,0)[lt]{\lineheight{1.25}\smash{\begin{tabular}[t]{l}$L_{01}$\end{tabular}}}}%
    \put(0.72756329,0.00518582){\makebox(0,0)[lt]{\lineheight{1.25}\smash{\begin{tabular}[t]{l}$L_{(r-1)r}$\end{tabular}}}}%
    \put(0.29211858,0.22900391){\makebox(0,0)[lt]{\lineheight{1.25}\smash{\begin{tabular}[t]{l}$N_0$\end{tabular}}}}%
    \put(0.29211858,0.33663322){\makebox(0,0)[lt]{\lineheight{1.25}\smash{\begin{tabular}[t]{l}$N_1$\end{tabular}}}}%
    \put(0.29211858,0.11120119){\makebox(0,0)[lt]{\lineheight{1.25}\smash{\begin{tabular}[t]{l}$N_{r-1}$\end{tabular}}}}%
  \end{picture}%
\endgroup%

\caption{}
\end{figure}

\noindent
Note that when $M_0 = \pt$ (in the notation of \eqref{eq:generalized_correspondence}), the 0-th map contains no information, so we can remove it from our depiction of the domain of $\ul u$ and obtain a quilted strip.

\noindent
When in addition $r=2$, quilted Floer trajectories specialize to ordinary Floer trajectories.

$\cM_{\ul\delta}(\ul p^-,\ul p^+)$ enjoys the regularization and compactification properties as the moduli space of unquilted Floer trajectories.
Underlying this is the fact that the linearization of the $\dbar$-operator that defines $\cM_{\ul\delta}(\ul p^-,\ul p^+)$ satisfies domain-local elliptic estimates that can be patched together into a global estimate, which implies that these linearized operators are Fredholm.
These local elliptic estimates result from the observation that locally in the domain, a quilted Floer trajectory can be recast as an unquilted pseudoholomorphic curve.
For instance, consider the maps $\wt u_0\colon \bR\times[0,\delta_0] \to M_j$ and $\wt u_1\colon \bR\times[\delta_0,\delta_0+\delta_1] \to M_1$ and a point $(s_0,\delta_0)$ in the intersection of their domains.
If we fix $r \in (0,\min\{\delta_0,\delta_1\})$, then the data of the restrictions $\wt u_0|_{B^-_r(s_0,\delta_0)}, \wt u_1|_{B^+_r(s_0,\delta_0)}$ is equivalent to the data of a map
\begin{align}
v
\colon
B_r^+(s_0,\delta_0)
\to
M_0\times M_1,
\quad
\partial_s v
+
((-J_0) \oplus J_1)\partial_t v
= 0,
\quad
v((-r,r))
\subset
L_{01},
\end{align}
where $B_r^+$ resp.\ $B_r^-$ denote the halves of the open disk $B_r$ defined by $\pm t \geq 0$.
(To see this, define $v$ by $v(s,t) \coloneqq \bigl(\wt u_0(s,-t),\wt u_1(s,t)\bigr)$.)

Regarding compactness, the standard approach to compactifying moduli spaces of pseudoholomorphic curves carries over to quilted Floer trajectories: given a sequence $(\ul u^\nu) \subset \cM_{\ul\delta}(\ul p^-,\ul p^+)$, we may pass to a subsequence that converges to a possibly-broken quilted Floer trajectory, to which trees of quilted bubbles may be attached.
Under the present hypotheses, the only codimension-1 degeneration in a family of quilted Floer trajectories is a single breaking.
As a result, the same argument that shows that the ordinary Floer differential squares to zero applies here.
Finally, an argument similar to the one that shows that $HF^*(L,K)$ is independent of the almost complex structure implies that $HF_{\ul\delta}(\ul L)$ is independent of the strip-widths.
We summarize our conclusions in the following definition-proposition.

\medskip

\noindent
{\bf Definition-Proposition.}
For any cyclic generalized correspondence $\ul L$ and tuple of widths $\ul\delta$, the quilted Floer cohomology $HF^*(\ul L)$ is the homology of the complex $(\cI(\ul L), d_{\ul\delta})$, where $d_{\ul\delta}$ is the differential that results from counting quilted Floer trajectories.
\null\hfill$\triangle$

\medskip

\subsubsection{Quilted invariants}
\label{sss:quilts}

The quilted cylinders we counted to define the quilted Floer differential in \S\ref{ss:qHF} are one instance of a general construction of \emph{pseudoholomorphic quilts}, or simply \emph{quilts}.
Quilts form the basis for all the functorial structures we will consider in this paper.

The domain for a quilt is a \emph{quilted surface}, which is a collection of \emph{patches}, i.e.\ Riemann surfaces with boundary marked points, glued together by diffeomorphisms of boundary components.
When two boundary components are identified, the result is a \emph{seam}.
Quilted surfaces may have boundary.
For a complete definition, see \cite[Definition 3.1]{wehrheim_woodward_jhol_quilts}.

Consider a quilted surface $\ul S$; suppose that the patches are labeled by symplectic manifolds, the seams are labeled by Lagrangian correspondences between the adjacent labels, and the boundary components are labeled by Lagrangians.
We denote the set of input marked points by $\cE_i$, and the set of output marked points by $\cE_o$.
For a single marked point $\ul e \in \cE_i \sqcup \cE_o$, we denote by $\ul L_e$ the cyclic generalized Lagrangian correspondence formed by the labels at this end.
Now, if we fix an element $\ul x_{\ul e}$ of the set of generalized intersection points $\cI(\ul L_{\ul e})$ (see \eqref{eq:generalized_intersection_points}) for every $\ul e \in \cE_i \sqcup \cE_o$, we can consider the associated moduli space $\cM_{\ul S}\bigl((\ul x_{\ul e})_{\ul e \in \cE_i \sqcup \cE_o}\bigr)$.
An element of this moduli space consists of a $J_{M_k}$-holomorphic map $u_k\colon S_k \to M_k$ for every $k$, the collection of which is required to satisfy analogues of the conditions in \eqref{eq:qHF_moduli_space}: the maps satisfy \emph{seam conditions} determined by the Lagrangian correspondences; all maps have finite energy; and the limits at the marked points are equal to the specified generalized intersection points.

By counting rigid instances of these quilts, we can produce a \emph{relative invariant} $\Phi_{\ul S}$:
\begin{align}
\Phi_{\ul S}
\colon
\bigotimes_{\ul e \in \cE_i}
HF^*(\ul L_{\ul e})
\to
\bigotimes_{\ul e \in \cE_o}
HF^*(\ul L_{\ul e}).
\end{align}
The fact that this invariant descends to the homology level follows from considering the codimension-1 degenerations in the associated moduli space $\cM_{\ul S}$.
These are exactly Floer breakings at the cylindrical ends, which corresponding to pre- or post-composing with the quilted Floer differential.

\subsubsection{Categorification commutes with composition}
\label{sss:strip-shrinking}

An important feature of quilted Floer cohomology is its invariance under a move called \emph{strip-shrinking}.
This was established in \cite[Theorem 1.0.1]{wehrheim2010functoriality}, in which Wehrheim--Woodward proved that under standard monotonicity hypotheses, and assuming also that the composition $L_{01}\circ L_{12}$ is embedded, the canonical bijection $(L_0\times L_{12})\cap(L_{01}\times L_2) \sr{\simeq}{\lra} (L_0 \times L_2) \cap (L_{01}\circ L_{12})$ induces an isomorphism
\begin{align}
\label{eq:strip_shrinking}
HF^*(L_0, L_{01}, L_{12}, L_2)
\sr{\simeq}{\lra}
HF^*(L_0, L_{01}\circ L_{12}, L_2).
\end{align}

\noindent
They proved this by constructing an isomorphism between the moduli spaces used to define the differentials.
That is, they identified triple and double strips of the following form:

\begin{figure}[H]
\centering
\def\svgwidth{0.9\columnwidth}
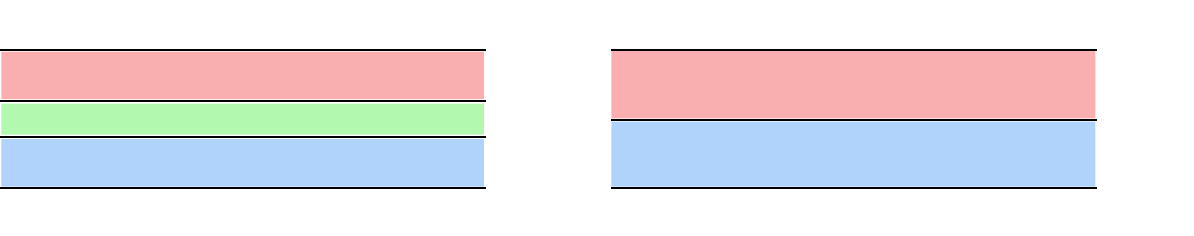
\caption{
\label{fig:double_triple_strips}
}
\end{figure}

\noindent
They accomplished this via a set of domain-independent elliptic inequalities.

\eqref{eq:strip_shrinking} immediately extends to a more general isomorphism of quilted invariants.
Indeed, Wehrheim--Woodward showed in \cite[Theorem 5.1]{wehrheim_woodward_jhol_quilts} that under the same hypotheses of \eqref{eq:strip_shrinking}, and assuming that two quilts differ by shrinking a strip or annulus in the first to a curve, their associated quilted invariants coincide.

The monotonicity and embedded-composition hypotheses that Wehrheim--Woodward imposed in order to prove \eqref{eq:strip_shrinking} are not merely technical.
Indeed, when they are relaxed, there is an expected obstruction to \eqref{eq:strip_shrinking} coming from counts of ``figure eight bubbles''.
The reason for this is that in a moduli space of triple strips (as on the left of Figure \ref{fig:double_triple_strips}) with middle strip-width varying in $[0,\eps_0)$, in the limit as strip-width approaches zero, we can see bubbling on the fused seam, like so:

\begin{figure}[H]
\centering
\def\svgwidth{0.4\columnwidth}
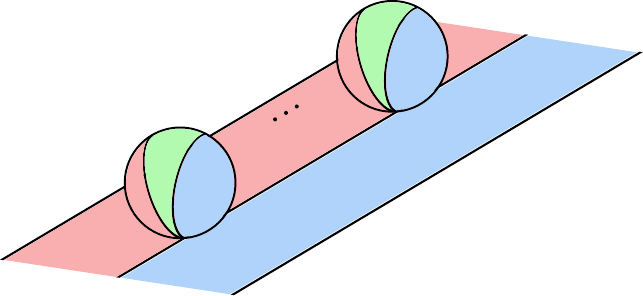
\caption{}
\end{figure}

\noindent
These quilted spheres, called \emph{figure eight bubbles} and first predicted by Wehrheim--Woodward in \cite[\S1]{wehrheim_woodward_geometric_composition}, are the result of gradient blowup on the middle strip at a rate commensurate with that at which the strip-width converges to 0.
One rescales at a rate inverse to the gradient blowup, which yields a quilted plane with three patches.
The second author showed in his thesis \cite{b:thesis} that the singularity at $\infty$ can be removed, hence this plane can be completed to a quilted sphere.

Bottman--Wehrheim \cite[\S4]{bottman_wehrheim} interpreted this bubble formation as saying that \eqref{eq:strip_shrinking} should be replaced by the formula
\begin{align}
\label{eq:strip_shrinking_obstructed}
HF^*(L_0, L_{01}, L_{12}, L_2)
\sr{\simeq}{\lra}
HF^*\bigl(L_0, (L_{01}\circ L_{12}, b_{L_{01},L_{12}}), L_2\bigr),
\end{align}
where $b_{L_{01},L_{12}} \in CF^*(L_{01}\circ L_{12},L_{01}\circ L_{12})$ is the result of counting rigid figure-eight bubbles with an output marked point at the south poles of the sphere (where the two seams come together).
The notation $(L_{01}\circ L_{12}, b_{L_{01},L_{12}})$ indicates that we have equipped the composed correspondence with a \emph{bounding cochain}.
Bounding cochains were introduced by Fukaya--Oh--Ohta--Ono in \cite{fooo_1}, in order to define the Fukaya category in situations where the Lagrangians $L$ may not satisfy $\mu_0 = 0$, i.e.\ there may be nonzero counts of disks with boundary on $L$.
A bounding cochain for $L$ is an element $b \in CF^*(L,L)$ satisfying the Maurer--Cartan equation
\begin{align}
\sum_{k=0}^\infty
\mu_k(b,\ldots,b)
=
0,
\end{align}
and equipping $L$ with a bounding cochain $b$ deforms all $A_\infty$-operations involving this correspondence.
(C.f. \S\ref{ss:anom-lagr-floer} for more about bounding cochains, and \S\ref{s:symp} for the much larger subject of the symplectic $(A_\infty,2)$-category that figure eight bubbling leads into.)

\subsection{The definition of $\Phi_{L_{12}}^\#\colon \Fuk^\#(M_1) \to \Fuk^\#(M_2)$}
\label{ss:def_of_Phi_L12}

In \S\ref{ss:approaching_Phi}, we contemplated from first principles the task of defining an $A_\infty$-functor $\Phi_{L_{12}}\colon \Fuk M_1 \to \Fuk M_2$.
The Operadic Principle helped us to formulate a strategy, which involved counting some sort of pseudoholomorphic maps whose domains are disks with boundary marked points and with one interior circle.
After our introduction to quilts in \S\ref{ss:qHF}, it is clear what we should do --- we should count pseudoholomorphic quilts whose domains are quilts of the following form:
\begin{figure}[H]
\centering
\def\svgwidth{0.2\columnwidth}
%% Creator: Inkscape 1.2 (dc2aeda, 2022-05-15), www.inkscape.org
%% Accompanies image file 'functor_def.pdf' (pdf, eps, ps)
\begingroup%
  \makeatletter%
  \providecommand\color[2][]{%
    \errmessage{(Inkscape) Color is used for the text in Inkscape, but the package 'color.sty' is not loaded}%
    \renewcommand\color[2][]{}%
  }%
  \providecommand\transparent[1]{%
    \errmessage{(Inkscape) Transparency is used (non-zero) for the text in Inkscape, but the package 'transparent.sty' is not loaded}%
    \renewcommand\transparent[1]{}%
  }%
  \providecommand\rotatebox[2]{#2}%
  \newcommand*\fsize{\dimexpr\f@size pt\relax}%
  \newcommand*\lineheight[1]{\fontsize{\fsize}{#1\fsize}\selectfont}%
  \ifx\svgwidth\undefined%
    \setlength{\unitlength}{88.55671105bp}%
    \ifx\svgscale\undefined%
      \relax%
    \else%
      \setlength{\unitlength}{\unitlength * \real{\svgscale}}%
    \fi%
  \else%
    \setlength{\unitlength}{\svgwidth}%
  \fi%
  \global\let\svgwidth\undefined%
  \global\let\svgscale\undefined%
  \makeatother%
  \begin{picture}(1,0.89019759)%
    \lineheight{1}%
    \setlength\tabcolsep{0pt}%
    \put(0,0){\includegraphics[width=\unitlength,page=1]{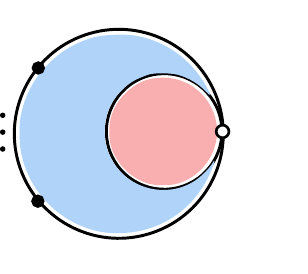}}%
    \put(0.13795923,0.42907018){\makebox(0,0)[lt]{\lineheight{1.25}\smash{\begin{tabular}[t]{l}$M_1$\end{tabular}}}}%
    \put(0.45222737,0.42907018){\makebox(0,0)[lt]{\lineheight{1.25}\smash{\begin{tabular}[t]{l}$M_2$\end{tabular}}}}%
    \put(0,0){\includegraphics[width=\unitlength,page=2]{functor_def.pdf}}%
    \put(0.77706747,0.85543881){\makebox(0,0)[lt]{\lineheight{1.25}\smash{\begin{tabular}[t]{l}$L_{12}$\end{tabular}}}}%
    \put(0.31582435,0.83499825){\makebox(0,0)[lt]{\lineheight{1.25}\smash{\begin{tabular}[t]{l}$L_1^0$\end{tabular}}}}%
    \put(0.31582435,0.00796188){\makebox(0,0)[lt]{\lineheight{1.25}\smash{\begin{tabular}[t]{l}$L_1^d$\end{tabular}}}}%
  \end{picture}%
\endgroup%

\caption{
\label{fig:functor_def}
}
\end{figure}

As we noted in \S\ref{ss:approaching_Phi}, the obvious definition of $\Phi_{L_{12}}$ on objects is
\begin{align}
\Phi_{L_{12}}(L_1) \coloneqq L_1 \circ L_{12} \coloneqq \pi_{M_2}(L_1\times_{M_1}L_{12}).
\end{align}
Besides the two issues mentioned in that subsection, another problem is that the tangential intersection of the seam with the boundary in the quilt just pictured leads to major analytical difficulties.
The way Ma'u--Wehrheim--Woodward dealt with this was to change their goal, and construct instead a functor $\Phi_{L_{12}}^\#\colon \Fuk^\#(M_1) \to \Fuk^\#(M_2)$ between \emph{extended Fukaya categories} (see Theorem \ref{thm:MWW_functors} in \S\ref{sec:corr-sympl-topol}).

The extended Fukaya category is tailor-made to eliminate these difficulties.
Here is a sketch of its definition, which Ma'u--Wehrheim--Woodward formulated under standard monotonicity hypotheses:

\begin{definition}
The \emph{extended Fukaya category} $\Fuk M$ is the $A_\infty$-category defined in the following way.
\begin{itemize}
\item
The objects of $\Fuk M$ are the generalized Lagrangians $\pt \sr{\ul L}{\lra} M$.

\smallskip

\item
The hom-set from $\ul L$ to $\ul K$ is defined to be the quilted Floer cochain complex $CF^*\bigl(\ul L,\ul K\bigr)$.

\smallskip

\item
The $A_\infty$-operations
\begin{align}
\mu_d\bigl(\ul L^0,\ldots,\ul L^d\bigr)
\colon
CF^*\bigl(\ul L^0,\ul L^1\bigr)
\otimes
\cdots
\otimes
CF^*\bigl(\ul L^{d-1},\ul L^d\bigr)
\to
CF^*\bigl(\ul L^0,\ul L^d\bigr)
\end{align}
are defined by counting rigid quilts of the following form:
\begin{figure}[H]
\centering
\def\svgwidth{0.175\columnwidth}
%% Creator: Inkscape 1.2 (dc2aeda, 2022-05-15), www.inkscape.org
%% Accompanies image file 'extended_composition.pdf' (pdf, eps, ps)
\begingroup%
  \makeatletter%
  \providecommand\color[2][]{%
    \errmessage{(Inkscape) Color is used for the text in Inkscape, but the package 'color.sty' is not loaded}%
    \renewcommand\color[2][]{}%
  }%
  \providecommand\transparent[1]{%
    \errmessage{(Inkscape) Transparency is used (non-zero) for the text in Inkscape, but the package 'transparent.sty' is not loaded}%
    \renewcommand\transparent[1]{}%
  }%
  \providecommand\rotatebox[2]{#2}%
  \newcommand*\fsize{\dimexpr\f@size pt\relax}%
  \newcommand*\lineheight[1]{\fontsize{\fsize}{#1\fsize}\selectfont}%
  \ifx\svgwidth\undefined%
    \setlength{\unitlength}{65.47085858bp}%
    \ifx\svgscale\undefined%
      \relax%
    \else%
      \setlength{\unitlength}{\unitlength * \real{\svgscale}}%
    \fi%
  \else%
    \setlength{\unitlength}{\svgwidth}%
  \fi%
  \global\let\svgwidth\undefined%
  \global\let\svgscale\undefined%
  \makeatother%
  \begin{picture}(1,1.17644435)%
    \lineheight{1}%
    \setlength\tabcolsep{0pt}%
    \put(0,0){\includegraphics[width=\unitlength,page=1]{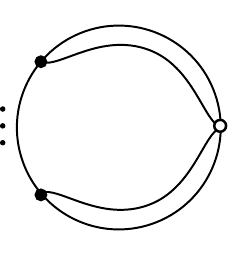}}%
    \put(0.48488811,1.12942919){\makebox(0,0)[lt]{\lineheight{1.25}\smash{\begin{tabular}[t]{l}$\ul L^0$\end{tabular}}}}%
    \put(0.48488811,0.01076933){\makebox(0,0)[lt]{\lineheight{1.25}\smash{\begin{tabular}[t]{l}$\ul L^d$\end{tabular}}}}%
    \put(0,0){\includegraphics[width=\unitlength,page=2]{extended_composition.pdf}}%
    \put(0.45002764,0.59570206){\makebox(0,0)[lt]{\lineheight{1.25}\smash{\begin{tabular}[t]{l}$M$\end{tabular}}}}%
    \put(0,0){\includegraphics[width=\unitlength,page=3]{extended_composition.pdf}}%
  \end{picture}%
\endgroup%

\caption{
\label{fig:extended_composition}
}
\end{figure}
\end{itemize}
\null\hfill$\triangle$
\end{definition}

\noindent
An important aspect of this definition is that the seams at the output marked point approach the boundary transversely, rather than tangentially.
That is, in striplike coordinates near the output marked point, the seams are parallel to the boundary, rather than asymptotic to the boundary.
This can be seen when comparing Figures \ref{fig:functor_def} and \ref{fig:extended_composition}.

While this version of the Fukaya category may appear so large as to be unmanageable, a folk expectation is that in many cases, the embedding $\Fuk M \hra \Fuk^\#(M)$ is a quasi-equivalence.
Using the extended Fukaya category in place of the ordinary one allows us to easily sidestep the first issue in the definition of $\Phi_{L_{12}}^\#$: instead of sending $L_1$ to the geometric composition $L_1 \circ L_{12}$, we can simply act on $\ul L_1 \in \Fuk^\#(M_1)$ by appending $L_{12}$.

We expand on this in the following definition.

\begin{definition}
Given a correspondence $L_{12} \subset M_1^-\times M_2$, we define $\Phi_{L_{12}}^\#$ in the following way.
\begin{itemize}
\item
On objects, $\Phi_{L_{12}}^\#$ acts like so:
\begin{align}
&\Phi_{L_{12}}^\#\bigl(\pt = N_0 \sr{K_{01}}{\lra} \cdots \sr{K_{(k-1)k}}{\lra} N_r = M_1\bigr)
\\
&\hspace{1.25in}
\coloneqq
\bigl(\pt = N_0 \sr{K_{01}}{\lra} \cdots \sr{K_{(k-1)k}}{\lra} N_r = M_1 \sr{L_{12}}{\lra} M_2\bigr).
\nonumber
\end{align}
If we denote the input generalized Lagrangian by $\ul L_1$, then we often denote the output by $\bigl(\ul L_1, L_{12}\bigr)$.

\smallskip

\item
On morphisms,
\begin{align}
\Phi_{L_{12}}^\#
\colon
\hom\bigl(\ul L_1^0,\ul L_1^1\bigr)\otimes\cdots\otimes\hom\bigl(\ul L_1^{r-1},\ul L_1^r\bigr)
\to
\hom\bigl(
\bigl(\ul L_1^0, L_{12}\bigr),
\bigl(\ul L_1^r, L_{12}\bigr)
\bigr)
\end{align}
is defined by counting rigid quilts of the following form:
\begin{figure}[H]
\centering
\def\svgwidth{0.2\columnwidth}
%% Creator: Inkscape 1.2 (dc2aeda, 2022-05-15), www.inkscape.org
%% Accompanies image file 'extended_functor.pdf' (pdf, eps, ps)
\begingroup%
  \makeatletter%
  \providecommand\color[2][]{%
    \errmessage{(Inkscape) Color is used for the text in Inkscape, but the package 'color.sty' is not loaded}%
    \renewcommand\color[2][]{}%
  }%
  \providecommand\transparent[1]{%
    \errmessage{(Inkscape) Transparency is used (non-zero) for the text in Inkscape, but the package 'transparent.sty' is not loaded}%
    \renewcommand\transparent[1]{}%
  }%
  \providecommand\rotatebox[2]{#2}%
  \newcommand*\fsize{\dimexpr\f@size pt\relax}%
  \newcommand*\lineheight[1]{\fontsize{\fsize}{#1\fsize}\selectfont}%
  \ifx\svgwidth\undefined%
    \setlength{\unitlength}{84.73984872bp}%
    \ifx\svgscale\undefined%
      \relax%
    \else%
      \setlength{\unitlength}{\unitlength * \real{\svgscale}}%
    \fi%
  \else%
    \setlength{\unitlength}{\svgwidth}%
  \fi%
  \global\let\svgwidth\undefined%
  \global\let\svgscale\undefined%
  \makeatother%
  \begin{picture}(1,0.90893272)%
    \lineheight{1}%
    \setlength\tabcolsep{0pt}%
    \put(0,0){\includegraphics[width=\unitlength,page=1]{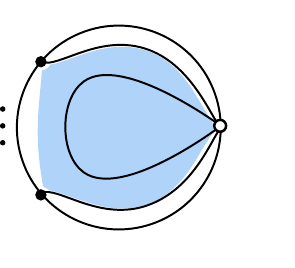}}%
    \put(0.37462942,0.87260834){\makebox(0,0)[lt]{\lineheight{1.25}\smash{\begin{tabular}[t]{l}$\ul L_1^0$\end{tabular}}}}%
    \put(0.37462942,0.0083205){\makebox(0,0)[lt]{\lineheight{1.25}\smash{\begin{tabular}[t]{l}$\ul L_1^d$\end{tabular}}}}%
    \put(0,0){\includegraphics[width=\unitlength,page=2]{extended_functor.pdf}}%
    \put(0.37936463,0.46024539){\makebox(0,0)[lt]{\lineheight{1.25}\smash{\begin{tabular}[t]{l}$M_2$\end{tabular}}}}%
    \put(0,0){\includegraphics[width=\unitlength,page=3]{extended_functor.pdf}}%
    \put(0.76702611,0.82301554){\makebox(0,0)[lt]{\lineheight{1.25}\smash{\begin{tabular}[t]{l}$L_{12}$\end{tabular}}}}%
    \put(0.06987485,0.46024539){\makebox(0,0)[lt]{\lineheight{1.25}\smash{\begin{tabular}[t]{l}$M_1$\end{tabular}}}}%
  \end{picture}%
\endgroup%

\caption{}
\end{figure}
\end{itemize}
\null\hfill$\triangle$
\end{definition}

\medskip

\noindent
{\it Picture proof that $\Phi_{L_{12}}^\#$ satisfies the $A_\infty$-relations, following \cite[Theorem 1.1]{mww}.}
In order for $\Phi_{L_{12}}^\#\colon \Fuk^\#(M_1) \to \Fuk^\#(M_2)$ to define an $A_\infty$-functor, it must satisfy the following relations for any inputs $x_1, \ldots, x_d$:
\begin{align}
\label{eq:functor_relations}
&\sum_{k,\ell}
\pm\Phi_{L_{12}}^\#\bigl(x_1,\ldots,x_k,\mu_\ell(x_{k+1},\ldots,x_{k+\ell}),x_{k+\ell+1},\ldots,x_d\bigr)
\\
&\hspace{0.25in}=
\sum_{\ell,k_1,\ldots,k_\ell}
\pm\mu_\ell\bigl(
\Phi_{L_{12}}^\#(x_1,\ldots,x_{k_1}),
\ldots,
\Phi_{L_{12}}^\#(x_{k_1+\cdots+k_{\ell-1}+1},\ldots,x_{k_1+\cdots+k_\ell})
\bigr).
\nonumber
\end{align}
This follows from considering the codimension-1 degenerations of the quilts that we counted to define $\Phi_{L_{12}}^\#$:
\begin{figure}[H]
\centering
\def\svgwidth{0.55\columnwidth}
%% Creator: Inkscape 1.2 (dc2aeda, 2022-05-15), www.inkscape.org
%% Accompanies image file 'quilted_functor_relations.pdf' (pdf, eps, ps)
\begingroup%
  \makeatletter%
  \providecommand\color[2][]{%
    \errmessage{(Inkscape) Color is used for the text in Inkscape, but the package 'color.sty' is not loaded}%
    \renewcommand\color[2][]{}%
  }%
  \providecommand\transparent[1]{%
    \errmessage{(Inkscape) Transparency is used (non-zero) for the text in Inkscape, but the package 'transparent.sty' is not loaded}%
    \renewcommand\transparent[1]{}%
  }%
  \providecommand\rotatebox[2]{#2}%
  \newcommand*\fsize{\dimexpr\f@size pt\relax}%
  \newcommand*\lineheight[1]{\fontsize{\fsize}{#1\fsize}\selectfont}%
  \ifx\svgwidth\undefined%
    \setlength{\unitlength}{303.99428358bp}%
    \ifx\svgscale\undefined%
      \relax%
    \else%
      \setlength{\unitlength}{\unitlength * \real{\svgscale}}%
    \fi%
  \else%
    \setlength{\unitlength}{\svgwidth}%
  \fi%
  \global\let\svgwidth\undefined%
  \global\let\svgscale\undefined%
  \makeatother%
  \begin{picture}(1,0.45222866)%
    \lineheight{1}%
    \setlength\tabcolsep{0pt}%
    \put(0,0){\includegraphics[width=\unitlength,page=1]{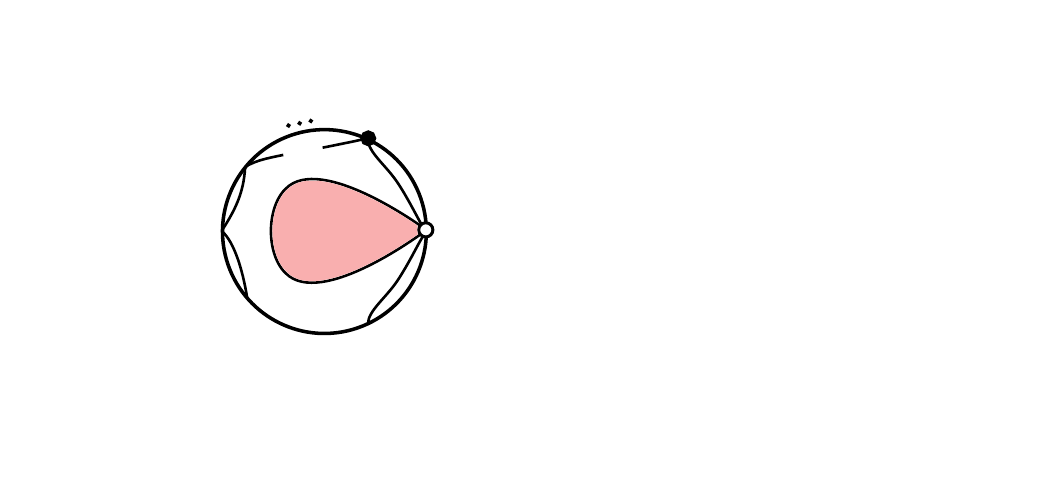}}%
    \put(0.2886486,0.22445446){\makebox(0,0)[lt]{\lineheight{1.25}\smash{\begin{tabular}[t]{l}$M_2$\end{tabular}}}}%
    \put(0,0){\includegraphics[width=\unitlength,page=2]{quilted_functor_relations.pdf}}%
    \put(0.08565044,0.225241){\makebox(0,0)[lt]{\lineheight{1.25}\smash{\begin{tabular}[t]{l}$M_1$\end{tabular}}}}%
    \put(0,0){\includegraphics[width=\unitlength,page=3]{quilted_functor_relations.pdf}}%
  \end{picture}%
\endgroup%

\caption{}
\end{figure}

\noindent
Indeed, counting the quilts on the left resp.\ on the right produces the left resp.\ right of \eqref{eq:functor_relations}.
\null\hfill$\square$

\medskip

The strip-shrinking isomorphism in \S\ref{sss:strip-shrinking} translates into a homotopy of functors.
Recall that $A_\infty$-functors $F, G\colon \cA \to \cB$ are \emph{homotopic} if they agree on objects, and if there exists a collection of morphisms
\begin{align}
T
\colon
\hom(X_{d-1},X_d)\otimes\cdots\hom(X_0,X_1) \to \hom(X_0,X_d)
\end{align}
so that the equation
\begin{align}
\label{eq:homotopy}
&(F - G)(\cdots)
=
\sum \pm T\bigl(\cdots,\mu_\cC(\cdots),\cdots\bigr)
+
\\
&\hspace{2.25in}
+ \sum \pm\mu_\cD\bigl(F(\cdots),\ldots,F(\cdots),T(\cdots),G(\cdots),\ldots,G(\cdots)\bigr)
\nonumber
\end{align}
holds.
(C.f.\ \cite[\S1h]{seidel_picard-lefschetz} for more details.)

Ma'u--Wehrheim--Woodward proved in \cite[Theorem 1.2]{mww} that if $L_{12}$ and $L_{23}$ have embedded composition, and if the usual monotonicity hypotheses hold, then the $A_\infty$-functors $\Phi_{L_{23}} \circ \Phi_{L_{12}}$ and $\Phi_{L_{12}}\circ\Phi_{L_{23}}$ are homotopic.
The homotopy $\Psi_{L_{12},L_{23}}$ is defined by counting quilted disks of the following form:

\begin{figure}[H]
\centering
\def\svgwidth{0.2\columnwidth}
%% Creator: Inkscape 1.2 (dc2aeda, 2022-05-15), www.inkscape.org
%% Accompanies image file 'homotopy.pdf' (pdf, eps, ps)
\begingroup%
  \makeatletter%
  \providecommand\color[2][]{%
    \errmessage{(Inkscape) Color is used for the text in Inkscape, but the package 'color.sty' is not loaded}%
    \renewcommand\color[2][]{}%
  }%
  \providecommand\transparent[1]{%
    \errmessage{(Inkscape) Transparency is used (non-zero) for the text in Inkscape, but the package 'transparent.sty' is not loaded}%
    \renewcommand\transparent[1]{}%
  }%
  \providecommand\rotatebox[2]{#2}%
  \newcommand*\fsize{\dimexpr\f@size pt\relax}%
  \newcommand*\lineheight[1]{\fontsize{\fsize}{#1\fsize}\selectfont}%
  \ifx\svgwidth\undefined%
    \setlength{\unitlength}{79.69374914bp}%
    \ifx\svgscale\undefined%
      \relax%
    \else%
      \setlength{\unitlength}{\unitlength * \real{\svgscale}}%
    \fi%
  \else%
    \setlength{\unitlength}{\svgwidth}%
  \fi%
  \global\let\svgwidth\undefined%
  \global\let\svgscale\undefined%
  \makeatother%
  \begin{picture}(1,0.8782249)%
    \lineheight{1}%
    \setlength\tabcolsep{0pt}%
    \put(0,0){\includegraphics[width=\unitlength,page=1]{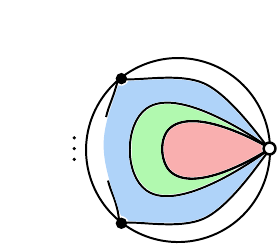}}%
    \put(0.25367907,0.76673639){\makebox(0,0)[lt]{\lineheight{1.25}\smash{\begin{tabular}[t]{l}$M_1$\end{tabular}}}}%
    \put(0.48970571,0.78349576){\makebox(0,0)[lt]{\lineheight{1.25}\smash{\begin{tabular}[t]{l}$M_2$\end{tabular}}}}%
    \put(0.66288538,0.31563191){\makebox(0,0)[lt]{\lineheight{1.25}\smash{\begin{tabular}[t]{l}$M_3$\end{tabular}}}}%
    \put(0,0){\includegraphics[width=\unitlength,page=2]{homotopy.pdf}}%
    \put(0.67405884,0.85891268){\makebox(0,0)[lt]{\lineheight{1.25}\smash{\begin{tabular}[t]{l}$L_{23}$\end{tabular}}}}%
    \put(0,0){\includegraphics[width=\unitlength,page=3]{homotopy.pdf}}%
    \put(-0.00189936,0.07262233){\makebox(0,0)[lt]{\lineheight{1.25}\smash{\begin{tabular}[t]{l}$L_{12}$\end{tabular}}}}%
  \end{picture}%
\endgroup%

\caption{}
\end{figure}

\noindent
Considering the codimension-1 degenerations of these disks leads to \eqref{eq:homotopy}.
Indeed, these degenerations occur when either (i) the $12$- and $23$-seams collide, (ii) the $23$-seam collides with the boundary, (iii) a consecutive sequence of boundary marked points collide, or (iv) the $12$- and $23$-seams collide at commensurate speeds with the boundary.
We depict these degenerations below:

\begin{figure}[H]
\centering
\def\svgwidth{0.8\columnwidth}
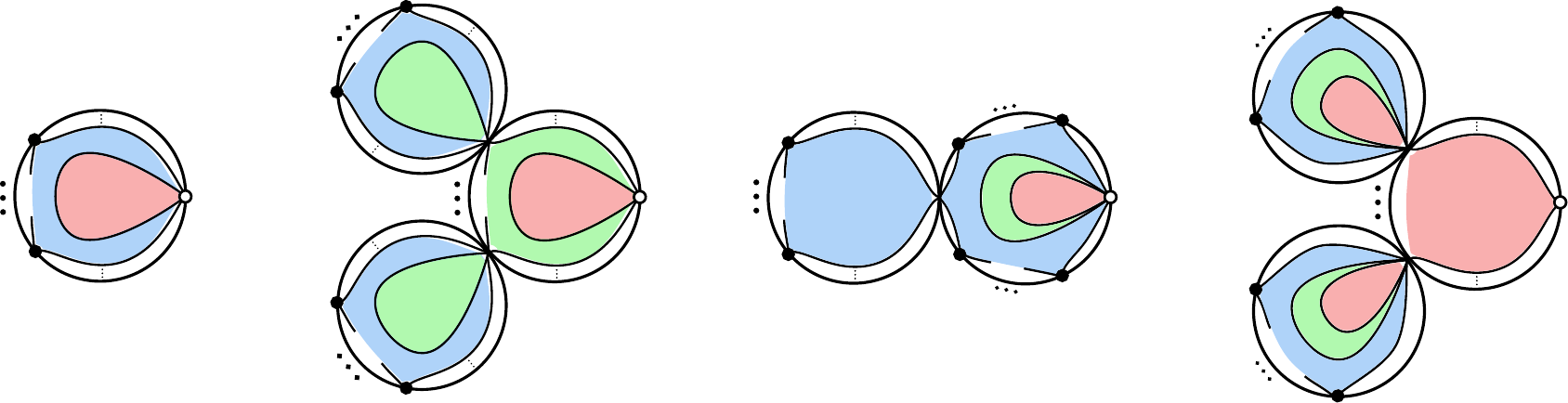
\caption{
\label{fig:homotopy_degens}
}
\end{figure}

\noindent
Counting these quilts nearly produces the four terms in \eqref{eq:homotopy}, from left to right.
The only discrepancy is between the right-most degeneration in Figure \ref{fig:homotopy_degens} and the last sum in \eqref{eq:homotopy}.
Ma'u--Wehrheim--Woodward resolved this by using \emph{delay functions} (c.f.\ \cite[\S7]{mww}).
An alternate approach, carried out in unpublished work of the second author, is to coherently augment the domain moduli spaces in a way that exactly interpolates between the two relevant types of degenerations.

\subsection{A related construction: $A_\infty$-bimodules from correspondences}
\label{ss:bimodules}

Given an $A_\infty$-category $\sA$, an \emph{$\sA$-module} is an $A_\infty$-functor $\sA \to \Ch$.
Here $\Ch$ is the dg-category of chain complexes of $\bK$-vector spaces, considered as an $A_\infty$-category.
Unwinding the definition, an $\sA$-module $\sM$ associates to every object $X \in \sA$ a graded vector space $\sM(X)$, and for every $d\geq 1$ and $d$-tuple of objects $X_0,\ldots,X_{d-1}$, maps
\begin{align}
\mu_{\sM,d}
\colon
\hom(X_0,X_1)
\otimes
\cdots
\otimes
\hom(X_{d-1},X_d)
\otimes
\sM(X_0)
\to
\sM(X_d)[1-d].
\end{align}
These maps are required to satisfy the following hierarchy of coherences that come from the $A_\infty$-equations:
\begin{align}
\label{eq:module_equations}
&\sum_{k,\ell}
\pm\mu_{\sM,d-\ell+1}\bigl(a_1,\ldots,a_k,\mu_{\sA,\ell}(a_{k+1},\ldots,a_{k+\ell}),a_{k+\ell+1},\ldots,a_d;m\bigr)
\\
&\hspace{1.0in}
+
\sum_k
\pm\mu_{\sM,k}\bigl(a_1,\ldots,a_k,\mu_{\sM,d-k}(a_{k+1},\ldots,a_d;m)\bigr)
=
0.
\nonumber
\end{align}
This definition should be compared with the fact that one can define a module over a ring $R$ to be a functor from $R$ to the category $\Ab$ of Abelian groups, where we view $R$ and $\Ab$ as categories enriched in $\Ab$.

\begin{example}
If $\sA$ is an $A_\infty$-category and $Y$ is an object in $\sA$, we can define an $\sA$-module $\sY$ by $\sY(X) \coloneqq \hom(X,Y)$ and $\mu_{\sY,d}\coloneqq\mu_{\sA,d}$.
This is called the \emph{Yoneda module associated to $Y$}.
\null\hfill$\triangle$
\end{example}

Similarly, given $A_\infty$-categories $\sA$ and $\sB$, we can define an $(\sA,\sB)$-bimodule $\sM$ to be an $A_\infty$-bifunctor $(\sA,\sB) \to \Ch$.
This amounts to a graded vector space $\sM(X,Y)$ for every $X \in \sA$ and $Y \in \sB$, and multilinear maps
\begin{align}
&\mu_{\sM,d,e}
\colon
\hom(X_0,X_1)\otimes\cdots\otimes\hom(X_{d-1},X_d)
\otimes
\\
&\hspace{1in}
\otimes
\hom(Y_0,Y_1)\otimes\cdots\otimes\hom(Y_{e-1},Y_e)
\otimes
\sM(X_0,Y_0)
\to
\sM(X_d,Y_e)
\nonumber
\end{align}
for any choice of objects $X_0,\ldots,X_d\in \sA$, $Y_0,\ldots,Y_e\in \sB$.
The $A_\infty$-equations for $\sM$ translate into an equation similar to \eqref{eq:module_equations}.

\begin{example}
If $\sF\colon \sA \to \sB$ is an $A_\infty$-functor, we can define an $(\sA,\sB)$-bimodule $\sM$ in the following way:
\begin{align}
\sM(X,Y)
\coloneqq
 \hom_\sB(\sF(X),Y),
\end{align}
\begin{align*}
&\mu_{\sM,d,e}(a_1,\ldots,a_d;b_1,\ldots,b_e;m)
\\
&\hspace{0.25in}\coloneqq
\nonumber
\sum_{k,s_1,\ldots,s_k}
\pm\mu_\sB\bigl(\sF(a_1,\ldots,a_{s_1}),\ldots,\sF(a_{s_1+\cdots+s_{k-1}+1},\ldots,a_d);b_1,\ldots,b_e;m\bigr).
\end{align*}
This is called the \emph{graph bimodule associated to $\sF$}.
\null\hfill$\triangle$
\end{example}

We can associate to a Lagrangian correspondence $L_{12} \subset M_1^-\times M_2$ a $(\Fuk M_1,\Fuk M_2)$-bimodule.
This construction was first proposed by Ma'u; in the exact context, it was constructed by Gao, c.f.\ e.g.\ \cite{gao2018functors}.
A variant of this construction appeared in \cite{woodward_gauged_Floer_theory_of_toric_moment_fibers}, and the domain moduli spaces were studied systematically by Ma'u in \cite{mau_n-modules}.

\begin{definition}
Given $L_{12} \subset M_1^-\times M_2$, we define its associated $(\Fuk M_1,\Fuk M_2)$-bimodule $\sM_{L_{12}}$ by setting $\sM_{L_{12}}(L_1,L_2)$ to be $CF^*(L_1,L_{12},L_2)$ and defining the structure maps to be the result of counting the following quilts:
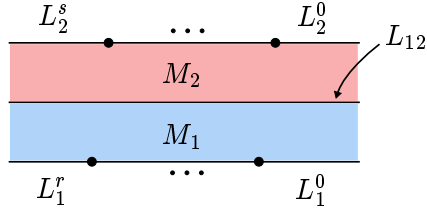
\begin{figure}[H]
\centering
\def\svgwidth{0.35\columnwidth}
%% Creator: Inkscape 1.2 (dc2aeda, 2022-05-15), www.inkscape.org
%% Accompanies image file 'bimodule.pdf' (pdf, eps, ps)
\begingroup%
  \makeatletter%
  \providecommand\color[2][]{%
    \errmessage{(Inkscape) Color is used for the text in Inkscape, but the package 'color.sty' is not loaded}%
    \renewcommand\color[2][]{}%
  }%
  \providecommand\transparent[1]{%
    \errmessage{(Inkscape) Transparency is used (non-zero) for the text in Inkscape, but the package 'transparent.sty' is not loaded}%
    \renewcommand\transparent[1]{}%
  }%
  \providecommand\rotatebox[2]{#2}%
  \newcommand*\fsize{\dimexpr\f@size pt\relax}%
  \newcommand*\lineheight[1]{\fontsize{\fsize}{#1\fsize}\selectfont}%
  \ifx\svgwidth\undefined%
    \setlength{\unitlength}{146.33089616bp}%
    \ifx\svgscale\undefined%
      \relax%
    \else%
      \setlength{\unitlength}{\unitlength * \real{\svgscale}}%
    \fi%
  \else%
    \setlength{\unitlength}{\svgwidth}%
  \fi%
  \global\let\svgwidth\undefined%
  \global\let\svgscale\undefined%
  \makeatother%
  \begin{picture}(1,0.42581847)%
    \lineheight{1}%
    \setlength\tabcolsep{0pt}%
    \put(0,0){\includegraphics[width=\unitlength,page=1]{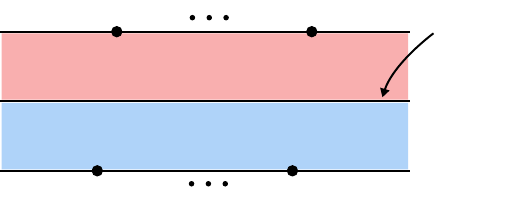}}%
    \put(0.86508542,0.35920445){\makebox(0,0)[lt]{\lineheight{1.25}\smash{\begin{tabular}[t]{l}$L_{12}$\end{tabular}}}}%
    \put(0.35211207,0.12955182){\makebox(0,0)[lt]{\lineheight{1.25}\smash{\begin{tabular}[t]{l}$M_1$\end{tabular}}}}%
    \put(0.35211207,0.27120989){\makebox(0,0)[lt]{\lineheight{1.25}\smash{\begin{tabular}[t]{l}$M_2$\end{tabular}}}}%
    \put(0.65944998,0.00481838){\makebox(0,0)[lt]{\lineheight{1.25}\smash{\begin{tabular}[t]{l}$L_1^0$\end{tabular}}}}%
    \put(0.66267689,0.40209417){\makebox(0,0)[lt]{\lineheight{1.25}\smash{\begin{tabular}[t]{l}$L_2^0$\end{tabular}}}}%
    \put(0.0678078,0.40478311){\makebox(0,0)[lt]{\lineheight{1.25}\smash{\begin{tabular}[t]{l}$L_2^s$\end{tabular}}}}%
    \put(0.06202585,0.00696994){\makebox(0,0)[lt]{\lineheight{1.25}\smash{\begin{tabular}[t]{l}$L_1^r$\end{tabular}}}}%
  \end{picture}%
\endgroup%

\caption{
The positions of the marked points on the boundary of this quilted strip are not fixed.
}
\end{figure}
\null\hfill$\triangle$
\end{definition}

The $A_\infty$-equations for $\sM_{L_{12}}$ follow immediately from an enumeration of the codimension-1 degenerations of the quilts we count to define the structure maps of $\sM_{L_{12}}$.

The advantage of considering bimodule $\sM_{L_{12}}$ rather than the functors $\Phi_{L_{12}}$ and $\Phi_{L_{12}}^\#$ is that the quilts we use to define $\sM_{L_{12}}$ do not involve tangential intersections of seams, so we are not forced to choose between working with such quilts or with the extended version of the Fukaya category.
Moreover, $\sM_{L_{12}}$ does not involve composing Lagrangians.
In situations where $\Phi_{L_{12}}$ can be defined on non-extended Fukaya categories, we expect $\sM_{L_{12}}$ to be isomorphic to the graph bimodule associated to $\Phi_{L_{12}}$.

\subsection{$(\Phi_{L_{12}}, \Phi_{L_{12}^T})$ is an adjoint pair}
\label{ss:adjunction}

Consider the functors
\begin{align}
\label{eq:putative_adjunction}
\xymatrix{\Fuk^\#(M_1) \ar@/^/[r]^{\Phi^\#_{L_{12}}} & \ar@/^/[l]^{\Phi^\#_{L_{12^T}}} \Fuk^\#(M_2)}
\end{align}
associated to a Lagrangian correspondence and its transpose.
These functors  form an adjoint pair.
(To our knowledge, this statement does not appear in Wehrheim--Woodward's work.)

According to \cite[Definition 2.14]{abouzaid2019khovanov}, a pair
\begin{align}
\xymatrix{
\cC \ar@/^/[r]^L & \ar@/^/[l]^R \cD
}
\end{align}
of $A_\infty$-functors is adjoint if there is a natural isomorphism of $\cC$-$\cD$-bimodules
\begin{align}
\hom_\cD(L(-), -)
\sr{\simeq}{\lra}
\hom_\cC(-, R(-)).
\end{align}
In the case of the putative adjunction \eqref{eq:putative_adjunction}, such a natural isomorphism includes a collection of morphisms
\begin{gather}
CF^*(L_1^{d-1},L_1^d)
\otimes\cdots\otimes
CF^*(L_1^0,L_1^1)
\otimes
CF^*(L_2^0,L_2^1)
\otimes\cdots\otimes
CF^*(L_2^{e-1},L_2^e)
\otimes
CF^*((L_1^0, L_{12}), L_2^0)
\\
\downarrow
\nonumber
\\
CF^*(L_1^d, (L_2^e, L_{12}^T)).
\nonumber
\end{gather}

\noindent
Such an isomorphism is given by counting quilts of the following form.

\begin{figure}[H]
\centering
\def\svgwidth{0.35\columnwidth}
%% Creator: Inkscape 1.2 (dc2aeda, 2022-05-15), www.inkscape.org
%% Accompanies image file 'adjunction.pdf' (pdf, eps, ps)
\begingroup%
  \makeatletter%
  \providecommand\color[2][]{%
    \errmessage{(Inkscape) Color is used for the text in Inkscape, but the package 'color.sty' is not loaded}%
    \renewcommand\color[2][]{}%
  }%
  \providecommand\transparent[1]{%
    \errmessage{(Inkscape) Transparency is used (non-zero) for the text in Inkscape, but the package 'transparent.sty' is not loaded}%
    \renewcommand\transparent[1]{}%
  }%
  \providecommand\rotatebox[2]{#2}%
  \newcommand*\fsize{\dimexpr\f@size pt\relax}%
  \newcommand*\lineheight[1]{\fontsize{\fsize}{#1\fsize}\selectfont}%
  \ifx\svgwidth\undefined%
    \setlength{\unitlength}{144.76088931bp}%
    \ifx\svgscale\undefined%
      \relax%
    \else%
      \setlength{\unitlength}{\unitlength * \real{\svgscale}}%
    \fi%
  \else%
    \setlength{\unitlength}{\svgwidth}%
  \fi%
  \global\let\svgwidth\undefined%
  \global\let\svgscale\undefined%
  \makeatother%
  \begin{picture}(1,0.43043669)%
    \lineheight{1}%
    \setlength\tabcolsep{0pt}%
    \put(0,0){\includegraphics[width=\unitlength,page=1]{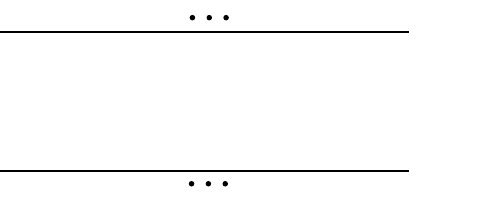}}%
    \put(0.66660206,0.00487064){\makebox(0,0)[lt]{\lineheight{1.25}\smash{\begin{tabular}[t]{l}$L_1^0$\end{tabular}}}}%
    \put(0.66986396,0.40645509){\makebox(0,0)[lt]{\lineheight{1.25}\smash{\begin{tabular}[t]{l}$L_2^0$\end{tabular}}}}%
    \put(0.06854321,0.40917319){\makebox(0,0)[lt]{\lineheight{1.25}\smash{\begin{tabular}[t]{l}$L_2^s$\end{tabular}}}}%
    \put(0.06269855,0.00704554){\makebox(0,0)[lt]{\lineheight{1.25}\smash{\begin{tabular}[t]{l}$L_1^r$\end{tabular}}}}%
    \put(0,0){\includegraphics[width=\unitlength,page=2]{adjunction.pdf}}%
    \put(0.8636222,0.23188402){\makebox(0,0)[lt]{\lineheight{1.25}\smash{\begin{tabular}[t]{l}$L_{12}$\end{tabular}}}}%
    \put(0.59851265,0.22139018){\makebox(0,0)[lt]{\lineheight{1.25}\smash{\begin{tabular}[t]{l}$M_2$\end{tabular}}}}%
    \put(0.12212436,0.20050558){\makebox(0,0)[lt]{\lineheight{1.25}\smash{\begin{tabular}[t]{l}$M_1$\end{tabular}}}}%
    \put(0,0){\includegraphics[width=\unitlength,page=3]{adjunction.pdf}}%
  \end{picture}%
\endgroup%

\caption{
Again, the positions of the marked points on the boundary of this quilted strip are not fixed.
}
\end{figure}

\subsection{Lagrangian correspondences from symplectic reduction, and an example}
\label{ss:reduction}

In this subsection, we will illustrate how the compatibility of categorification with geometric composition of Lagrangian correspondences can be used to effectively compute Floer cohomology.
Since we will use group actions to produce our illustrative examples, we need to introduce the right notion of quotients in symplectic topology: the starting point is to consider a compact Lie group $G$ acting smoothly, and by symplectomorphisms, on a symplectic manifold $(M,\omega)$.
By differentiating this action, we obtain a map from the Lie algebra $\fg$ to the space of vector fields on $M$; note that this map is automatically $G$-equivariant when the domain is equipped with the adjoint action.
As we introduced in \S\ref{ss:HF}, the symplectic form assigns to each function on $M$ its Hamiltonian vector field, and we say the datum of a \emph{Hamiltonian action of $G$ on $M$} consists of an equivariant lift:
\begin{equation}
  \begin{tikzcd}
    & C^\infty(M,\bR) \ar[d] \\
    \fg \ar[r] \ar[ur,dashed] &  C^\infty(M,TM).
  \end{tikzcd}
\end{equation}

In order to record this data more efficiently, it is convenient to express it in terms of the dual Lie algebra $\fg^*$:
\begin{definition}
The \emph{moment map} of a Hamiltonian $G$-action is the map
\begin{equation}
\mu\colon M \to \fg^*
\end{equation}
which is characterized by the equation
\begin{equation}
\rd\mu^X = \iota_{X^\#}\omega,
\end{equation}
for each vector field  $X \in \fg$, with associated vector field  $X^\#$ on $M$.
\null\hfill$\triangle$
\end{definition}

\noindent
While this definition may seem intimidating at first, it is simply the natural generalization to a general $G$ of the action of $\bR$ on a symplectic manifold by the flow of a Hamiltonian vector field.

\begin{example}
Consider the action of $T^n$ on $\bCP^n$ by rotating the latter $n$ homogeneous coordinates:
\begin{align}
\bigl(e^{i\lambda_1},\ldots,e^{i\lambda_n}\bigr)
\cdot
[z_0:\cdots:z_n]
\coloneqq
\bigl[z_0:e^{i\lambda_1}z_1:\cdots e^{i\lambda_n}z_n\bigr].	
\end{align}
Equip $\bCP^n$ with the normalized Fubini--Study form $\omega \coloneqq \alpha\omega_{FS}$, where $\alpha$ is chosen so that $\omega$ has monotonicity constant 1, i.e.\ so that the class of the symplectic form agrees with the first Chern class.
This action is Hamiltonian, with moment map given by:
\begin{align}
\mu
=
(\mu_1,\ldots,\mu_n)
\colon
\bCP^n
\to
\bR^n,
\qquad
\mu([z_0:\cdots:z_n])
\coloneqq
-\frac12
\frac{\bigl(|z_0|^2,\ldots,|z_n|^2\bigr)}{|z_0|^2+\cdots+|z_n|^2}.
\end{align}

\null\hfill$\triangle$	
\end{example}

\noindent
Atiyah and Guillemin--Sternberg's Convexity Theorem \cite{atiyah_convexity,guillemin_sternberg_convexity} asserts that the image of the moment map of a Hamiltonian torus action is the convex hull of the images of the fixed points, and in particular is always a convex polytope.
In the case of the above standard action on projective space, this image is the simplex spanned by the origin and the rescaled standard basis vectors $-\tfrac12e_i$ (the unfortunate scaling is a consequence of our choice of monotonicity constant and of the conventions for Hamiltonian vector fields).

With this necessary groundwork in place, we can define the symplectic reduction $M/\!\!/G$ and the associated correspondence $M/\!\!/G \sr{\Lambda_G}{\lra} M$ associated to a Hamiltonian action of a compact Lie group $G$ on $(M,\omega_M)$, with moment map $\mu$.

\begin{definition}
\label{def:symp_red}
If $a$ is fixed by the coadjoint action, and $G$ acts freely on $ \mu^{-1}(a) $, the \emph{symplectic reduction} or \emph{symplectic quotient} at level $a$ is the quotient
\begin{equation}
    M/\!\!/G \coloneqq \mu^{-1}(a)/G,
\end{equation}
equipped with the symplectic form $\omega_{M/\!\!/G}$ that is characterized  by $\pi^*\omega_{M/\!\!/G} = \iota^*\omega_M$.
The \emph{moment correspondence} is the Lagrangian
\begin{align} \label{eq:moment_correspondence}
\Lambda_G \coloneqq \{([p],p) \in M/\!\!/G \times M \:|\: p \in \mu^{-1}(a)\} \subset M/\!\!/G \times M^-.
\end{align}
\null\hfill$\triangle$
\end{definition}

\noindent
The fact that $\Lambda_G $ is Lagrangian is a consequence of the defining property of $\omega_{M/\!\!/G}$.

\begin{example}[Theorem 6.2.1, \cite{wehrheim2010quilted}]
In this example, which appears as a theorem in \cite{wehrheim2010quilted}, we show how categorification-commutes-with-composition can be used to efficiently compute the self-Floer cohomology of the Clifford torus
\begin{align}
T_\Cl^n
\coloneqq
\bigl\{
[z_0:\cdots:z_n]
\:\big|\:
|z_0|=\cdots=|z_n|
\bigr\}
\subset
\bCP^n,
\end{align}
with its standard spin structure.

We begin by introducing a collection of Lagrangian correspondences between complex projective spaces, which all arise in the fashion of Definition \ref{def:symp_red}.
Given $n \geq 1$, we loosely follow \cite{wehrheim2010quilted}
and define correspondences $\Sigma_A$ for $A \subsetneq \{1,\ldots,n\}$:
\begin{gather}
\Sigma_A
\coloneqq
\bCP^{n-\#A}
\sr{\Sigma_A}{\lra}
\bCP^n,
\\
\Sigma_A
\coloneqq
\Bigl\{
\bigl(
[z_j]_{j \in \{0,\ldots,n\}\setminus A},
[z_0:\cdots:z_n]
\bigr)
\:\Big|\:
|z_k|^2 = \tfrac1{n+1}\sum_{j=0}^n |z_j|^2
\:\forall\:
k \in A
\Bigr\}.
\nonumber
\end{gather}
Note that $\Sigma_A$ is the correspondence that results from taking the symplectic quotient $\bCP^n/\!\!/T^{\#A}$, where $T^{\#A}\colon \bCP^n \to \bCP^n$ is the restriction of the standard action of $T^n$ (on the latter $n$ homogeneous coordinates) to the circle factors corresponding to the indices in $A$.

In fact, one can use strong induction to show $HF^*(T^n_\Cl, T^n_\Cl) \simeq HF^*(T^{n-1}_\Cl, T^{n-1}_\Cl)^{\otimes 2}$, hence $HF^*(T^n_\Cl,T^n_\Cl) \simeq \bZ^{2^n}$.
In the following figure, we illustrate how we can prove the induction step in the case of $HF^*(T^3_\Cl, T^3_\Cl)$ by applying the strip-shrinking isomorphism.
\begin{figure}[H]
\centering
\def\svgwidth{1.0\columnwidth}
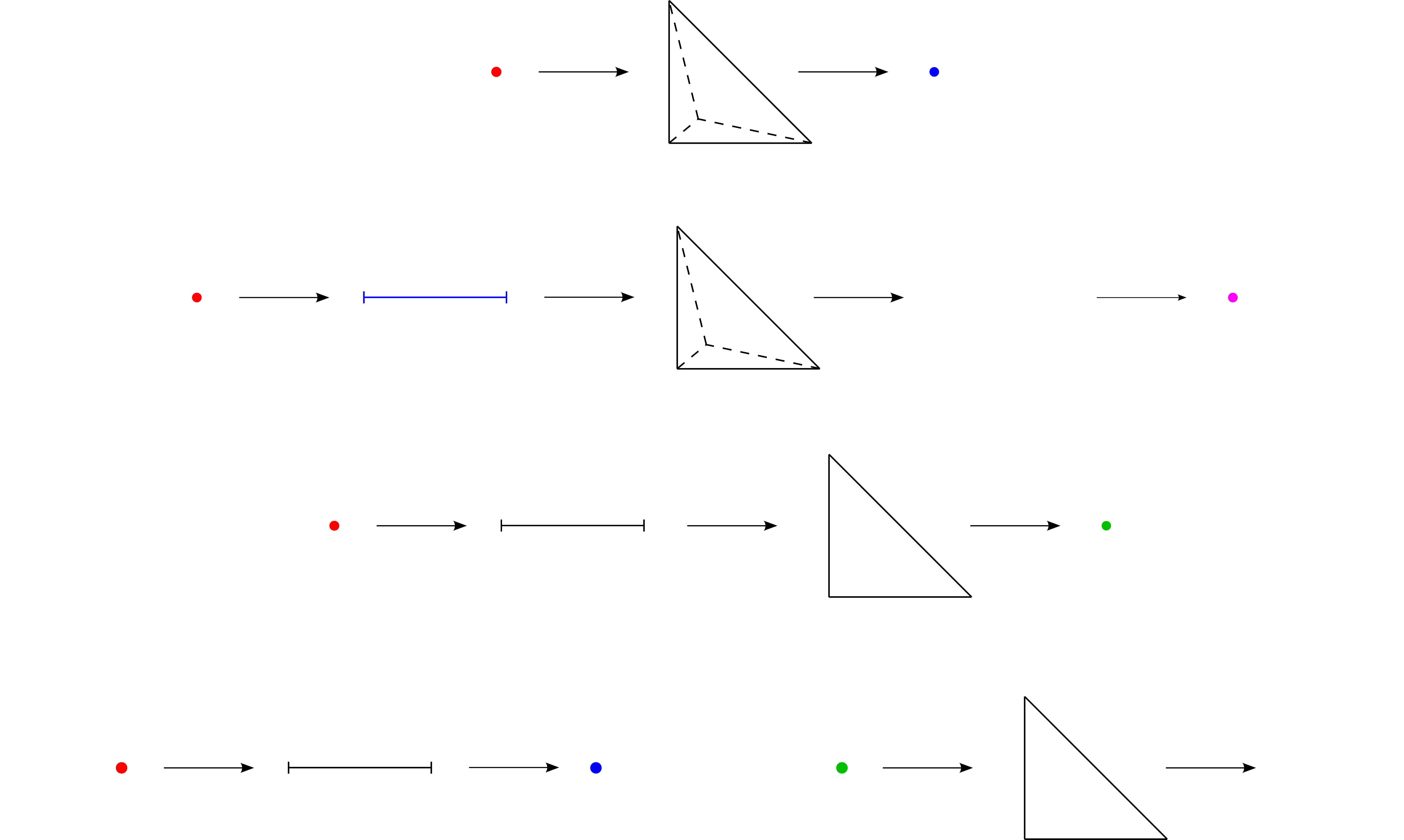
\caption{
These four expressions are isomorphic, thus identifying $HF^*(T^3_\Cl,T^3_\Cl)$ with $HF^*(S^1_\Cl,S^1_\Cl) \otimes HF^*(T^2_\Cl,T^2_\Cl)$.
The first two isomorphisms are essentially the strip-shrinking isomorphism \eqref{eq:strip_shrinking}.
The third isomorphism is straightforward.
\null\hfill$\triangle$
\label{fig:T_Cl_picture_computation}
}
\end{figure}
\end{example}

\subsection{The continuation map approach}
\label{ss:Y-map}

In \cite{lekili2013geometric}, Lekili--Lipyanskiy demonstrated another method for proving categorification-commutes-with-composition results in quilted Floer theory.
They were motivated by Lekili's work on identifying Perutz's Lagrangian matching invariants with Ozsv\'ath--Szab\'o's Heegaard Floer invariants for 3-manifolds equipped with ``broken fibrations'' over $S^1$, which requires working in the strongly negatively monotone case.
(This project of Lekili's resulted in \cite{lekili_heegaard_floer}.)
Lekili--Lipyanski's main result is the following variant of the strip-shrinking isomorphism \eqref{eq:strip_shrinking}.

\begin{theorem}[Paraphrase of Theorem 3, \cite{lekili2013geometric}, in the corrected form described in \cite{lekili_lipyanskiy_corrigendum}]
\label{thm:lekili-lipyanskiy_Y-map}
Fix closed symplectic manifolds $M_0, M_1, M_2$ of dimensions $d_0, d_1, d_2$.
Fix compact Lagrangians
\begin{align}
L_0 \subset M_0,
\qquad
L_{01} \subset M_0^-\times M_1,
\qquad
L_{12} \subset M_1^-\times M_2,
\qquad
L_2 \subset M_2
\end{align}
such that $L_{01}$ and $L_{12}$ have embedded composition.
Suppose that the symplectic manifolds and Lagrangians are negatively monotone, and satisfy a certain index inequality.
Then the canonical bijection $(L_0\times L_{12})\cap(L_{01}\times L_2) \sr{\simeq}{\lra} (L_0 \times L_2) \cap (L_{01}\circ L_{12})$ induces an isomorphism
\begin{align}
\label{eq:lekili-lipyanskiy_isomorphism}
HF^*(L_0, L_{01}, L_{12}, L_2)
\sr{\simeq}{\lra}
HF^*(L_0, L_{01}\circ L_{12}, L_2).
\end{align}
\null\hfill$\square$
\end{theorem}

\noindent
Lekili--Lipyanskiy construct the isomorphism \eqref{eq:lekili-lipyanskiy_isomorphism} using a continuation map.
Specifically, they count quilted strips of the following form, which are known as \emph{Y-maps}:

\begin{figure}[H]
\centering
\def\svgwidth{0.4\columnwidth}
%% Creator: Inkscape 1.2 (dc2aeda, 2022-05-15), www.inkscape.org
%% Accompanies image file 'Y-map.pdf' (pdf, eps, ps)
\begingroup%
  \makeatletter%
  \providecommand\color[2][]{%
    \errmessage{(Inkscape) Color is used for the text in Inkscape, but the package 'color.sty' is not loaded}%
    \renewcommand\color[2][]{}%
  }%
  \providecommand\transparent[1]{%
    \errmessage{(Inkscape) Transparency is used (non-zero) for the text in Inkscape, but the package 'transparent.sty' is not loaded}%
    \renewcommand\transparent[1]{}%
  }%
  \providecommand\rotatebox[2]{#2}%
  \newcommand*\fsize{\dimexpr\f@size pt\relax}%
  \newcommand*\lineheight[1]{\fontsize{\fsize}{#1\fsize}\selectfont}%
  \ifx\svgwidth\undefined%
    \setlength{\unitlength}{137.19583363bp}%
    \ifx\svgscale\undefined%
      \relax%
    \else%
      \setlength{\unitlength}{\unitlength * \real{\svgscale}}%
    \fi%
  \else%
    \setlength{\unitlength}{\svgwidth}%
  \fi%
  \global\let\svgwidth\undefined%
  \global\let\svgscale\undefined%
  \makeatother%
  \begin{picture}(1,0.4531147)%
    \lineheight{1}%
    \setlength\tabcolsep{0pt}%
    \put(0,0){\includegraphics[width=\unitlength,page=1]{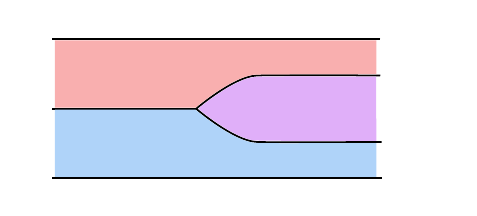}}%
    \put(0.60742628,0.10007348){\makebox(0,0)[lt]{\lineheight{1.25}\smash{\begin{tabular}[t]{l}$M_0$\end{tabular}}}}%
    \put(0.60742628,0.31867239){\makebox(0,0)[lt]{\lineheight{1.25}\smash{\begin{tabular}[t]{l}$M_2$\end{tabular}}}}%
    \put(0.60742628,0.21106497){\makebox(0,0)[lt]{\lineheight{1.25}\smash{\begin{tabular}[t]{l}$M_1$\end{tabular}}}}%
    \put(0,0){\includegraphics[width=\unitlength,page=2]{Y-map.pdf}}%
    \put(0.85610225,0.43067872){\makebox(0,0)[lt]{\lineheight{1.25}\smash{\begin{tabular}[t]{l}$L_{12}$\end{tabular}}}}%
    \put(0,0){\includegraphics[width=\unitlength,page=3]{Y-map.pdf}}%
    \put(0.85610225,0.00558052){\makebox(0,0)[lt]{\lineheight{1.25}\smash{\begin{tabular}[t]{l}$L_{01}$\end{tabular}}}}%
    \put(0,0){\includegraphics[width=\unitlength,page=4]{Y-map.pdf}}%
    \put(-0.00220658,0.42499982){\makebox(0,0)[lt]{\lineheight{1.25}\smash{\begin{tabular}[t]{l}$L_{01}\circ L_{12}$\end{tabular}}}}%
  \end{picture}%
\endgroup%

\caption{The Y-quilt.
}
\label{fig:Y-quilt}
\end{figure}

\subsection{Fukaya's alternate approach to constructing composition (bi)functors}
\label{ss:fukaya}

In \cite{fukaya2017unobstructed}, Kenji Fukaya built on Lekili--Lipyanskiy's work to construct functors
\begin{align}
\Phi_{L_{12}}
\colon
\Fuk M_1
\to
\Fuk M_2
\end{align}
in the general compact setting, and proved that they are compatible with compositions.

The essential difficulty in establishing such a result lies in understanding how to extract the desired algebraic structures from the compactification of the moduli space of Y-quilts: in virtual codimension $1$, the possible breakings take place at the ends, along the seams, or at the singular point where all three seams meet.
The breakings at the end correspond to differentials in Floer complexes, and those along the seams to the curvature of each of the Lagrangian correspondences, which means that it remains to account for breaking at the singular point.
The corresponding bubble turns out to be a cylinder with three parallel seams, labelled by $L_{01}$, $L_{12}$, and the geometric compositions $L_{01} \circ L_{12}$, like so:

\begin{figure}[H]
\centering
\def\svgwidth{0.45\columnwidth}
%% Creator: Inkscape 1.2 (dc2aeda, 2022-05-15), www.inkscape.org
%% Accompanies image file 'three-patch_cylinder.pdf' (pdf, eps, ps)
\begingroup%
  \makeatletter%
  \providecommand\color[2][]{%
    \errmessage{(Inkscape) Color is used for the text in Inkscape, but the package 'color.sty' is not loaded}%
    \renewcommand\color[2][]{}%
  }%
  \providecommand\transparent[1]{%
    \errmessage{(Inkscape) Transparency is used (non-zero) for the text in Inkscape, but the package 'transparent.sty' is not loaded}%
    \renewcommand\transparent[1]{}%
  }%
  \providecommand\rotatebox[2]{#2}%
  \newcommand*\fsize{\dimexpr\f@size pt\relax}%
  \newcommand*\lineheight[1]{\fontsize{\fsize}{#1\fsize}\selectfont}%
  \ifx\svgwidth\undefined%
    \setlength{\unitlength}{178.66684743bp}%
    \ifx\svgscale\undefined%
      \relax%
    \else%
      \setlength{\unitlength}{\unitlength * \real{\svgscale}}%
    \fi%
  \else%
    \setlength{\unitlength}{\svgwidth}%
  \fi%
  \global\let\svgwidth\undefined%
  \global\let\svgscale\undefined%
  \makeatother%
  \begin{picture}(1,0.42119476)%
    \lineheight{1}%
    \setlength\tabcolsep{0pt}%
    \put(0,0){\includegraphics[width=\unitlength,page=1]{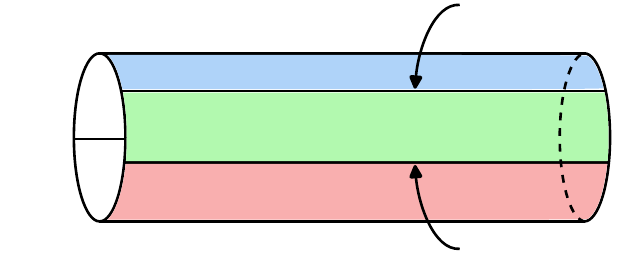}}%
    \put(0.74885626,0.40396647){\makebox(0,0)[lt]{\lineheight{1.25}\smash{\begin{tabular}[t]{l}$L_{01}$\end{tabular}}}}%
    \put(0.74885626,0.00428521){\makebox(0,0)[lt]{\lineheight{1.25}\smash{\begin{tabular}[t]{l}$L_{01}\circ L_{12}$\end{tabular}}}}%
    \put(0.3706763,0.19866928){\makebox(0,0)[lt]{\lineheight{1.25}\smash{\begin{tabular}[t]{l}$M_0$\end{tabular}}}}%
    \put(0.3706763,0.2921444){\makebox(0,0)[lt]{\lineheight{1.25}\smash{\begin{tabular}[t]{l}$M_1$\end{tabular}}}}%
    \put(0.3706763,0.09635864){\makebox(0,0)[lt]{\lineheight{1.25}\smash{\begin{tabular}[t]{l}$M_2$\end{tabular}}}}%
    \put(0,0){\includegraphics[width=\unitlength,page=2]{three-patch_cylinder.pdf}}%
    \put(-0.00169441,0.32866884){\makebox(0,0)[lt]{\lineheight{1.25}\smash{\begin{tabular}[t]{l}$L_{12}$\end{tabular}}}}%
  \end{picture}%
\endgroup%

\caption{
\label{fig:quilted_cylinder_3-marked-points}
}
\end{figure}

\noindent
This cylinder exactly corresponds to the differential in the quilted Floer complex
\begin{equation}
\label{eq:quilted_Floer_composition}
CF^*(L_{01}, L_{12}, L_{01} \circ L_{12}).  
\end{equation}

The naive expectation is that, since the fibre product of $L_{01} \times L_{12}$ with $L_{01} \circ L_{12}$ over $M_0 \times M_1 \times M_1 \times M_2$ is exactly a copy of $ L_{01} \circ L_{12}$, the fundamental class of this manifold should define a cycle in the Floer complex of \eqref{eq:quilted_Floer_composition}, and that inserting this cycle at the Y-point gives rise to the desired map associated to the Y-quilt.

Wehrheim--Woodward worked in a setting where every Lagrangian $L$ has the property that $\mu_0 = 0$, i.e.\ the count of rigid disks on $L$ equals 0, and  the topological assumptions alluded to in Theorem~\ref{thm:lekili-lipyanskiy_Y-map} are imposed specifically in order to ensure that this breaking does not occur for the moduli spaces that are required in the construction of the functor.
In the general compact setting, in which the Lagrangians are moreover only assumed to be immersed, this is no longer the case. 
For the particular problem at hand, the differential on the quilted Floer complex may not square to $0$, so that it does not even make sense to consider cycles.
The proper algebraic structure is that of a \emph{curved $A_\infty$-trimodule}, over the curved Floer algebras of the three Lagrangians: the equation $d^2x=0$ is replaced, in this context, by the equality between $d^2$ and the result of acting on $x$ by the curvatures of the Lagrangians $L_{01}$, $L_{12}$, and $L_{01} \circ L_{12}$.

Fukaya therefore needed to consider Lagrangians $L$ equipped with bounding cochains (c.f.\ \S\ref{sss:strip-shrinking}).
Constructing a functor associated to the Y-quilt thus amounts to solving the following problem:

\begin{theorem}[Paraphrase of Theorem 1.5 and Proposition 8.11, \cite{fukaya2017unobstructed}]
\label{thm:fukaya_bounding_cochain_transfer}
Fix immersed Lagrangians $L_{01} \subset M_0^- \times M_1$ and $L_{12} \subset M_1^- \times M_2$ equipped with bounding cochains $b_{01}, b_{12}$, and assume that the fiber product $L_{01} \circ  L_{12}$ is cut out transversely, with clean self-intersections.
There is a bounding cochain on $L_{01} \circ L_{12}$ which is characterised, up to gauge equivalence, by the property that the Yoneda module of $ L_{01} \circ L_{12}$ is quasi-isomorphic to the module (over the Fukaya category of $M_0^- \times M_2$) associated to the composition of the correspondences $L_{01}$ and $L_{12}$.
\null\hfill$\square$
\end{theorem}

\begin{remark}
Since classical techniques do not suffice to even define the Fukaya category in the general compact setting, Fukaya relies heavily on the techniques developed in \cite{fooo_2}, whose exposition goes beyond what we can hope to achieve in this paper.
\null\hfill$\triangle$
\end{remark}

Returning to the discussion above, Fukaya then shows in \cite[Theorem 9.1]{fukaya2017unobstructed} that the choice of bounding cochain on $ L_{01} \circ L_{12}$ fixed in Theorem \ref{thm:fukaya_bounding_cochain_transfer}, combined with the study of Lekili and Lipyanki's Y-map, determine a homotopy equivalence between the functors associated to the geometric composition, and the composition of functors associated to $L_{01}$ and $L_{12}$:
\begin{equation}
  \begin{tikzcd}
    \Fuk M_0 \ar[r] \ar[dr] & \Fuk M_1 \ar[d] \\
    &  \Fuk M_2.
  \end{tikzcd}
\end{equation}

Fukaya then upgrades this result to prove the following result, which constructs a composition bifunctor between Fukaya categories.

\begin{theorem}[Paraphrase of Theorem 1.8, \cite{fukaya2017unobstructed}]
The quasi-equivalence of Theorem \ref{thm:fukaya_bounding_cochain_transfer} extends to a filtered $A_\infty$-bifunctor
\begin{align}
\bigl(\Fuk(M_0^-\times M_1),
\Fuk(M_1^-\times M_2)
\bigr)
\to
\Fuk(M_0^-\times M_2),
\end{align}
which is associative up to homotopy in the sense that there is a prescribed homotopy in the following square of $A_\infty$-trifunctors:
\begin{equation} \label{eq:homotopy_associative_composition}
\begin{tikzcd}
\bigl(\Fuk(M_0^-\times M_1), \Fuk(M_1^-\times M_2), \Fuk(M_2^-\times M_3) \bigr) \ar[r] \ar[d]
&
\bigl(\Fuk(M_0^-\times M_1), \Fuk(M_1^-\times M_3) \bigr) \ar[d]
\\
\bigl(\Fuk(M_0^-\times M_2), \Fuk(M_2^-\times M_3) \bigr)  \ar[r]
&
\Fuk(M_0^-\times M_3).
\end{tikzcd}
\end{equation}
\null\hfill$\square$
\end{theorem}
Fukaya's proof of the first part of the above result does not use any significantly new geometric input beyond the moduli spaces in Figure \ref{fig:quilted_cylinder_3-marked-points} (with additional marked points along the seams).
On the other hand, the proof of the second part --- i.e.\ the homotopy in Diagram \eqref{eq:homotopy_associative_composition} --- uses a new type of quilted surface, lying on a genus-$0$ curve with four punctures, as shown in Figure \ref{fig:fukaya_jumpsuit}.

\begin{figure}[H]
\centering
\def\svgwidth{0.55\columnwidth}
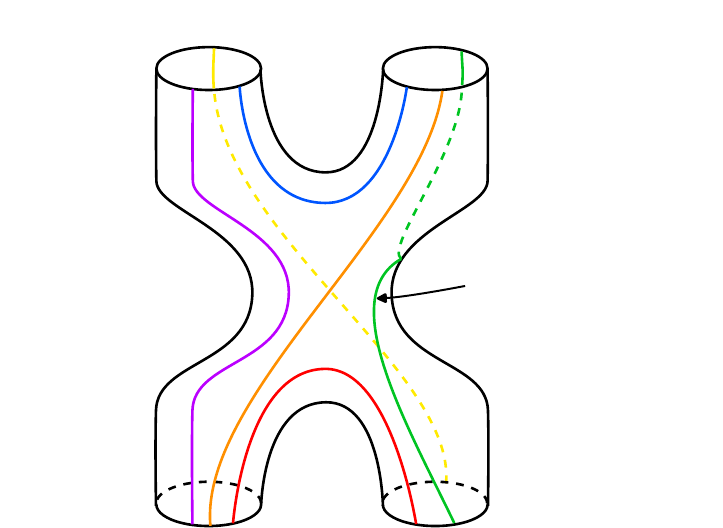
\caption{
\label{fig:fukaya_jumpsuit}
}
\end{figure} 

Heuristically, the appearance of the quilted surface in Figure \ref{fig:fukaya_jumpsuit} can be justified as follows: given a triple $L_{01}$, $L_{12}$, and $L_{23}$ of Lagrangian correspondences, which are in generic position, one defines a Lagrangian correspondence $L_{ij}$ for each pair $0 \leq i< j \leq 3$ by geometric composition.
Now, each end of the surface in Figure \ref{fig:fukaya_jumpsuit} is labelled by a triple of integers $i < j < k$, and carries parallel seams to which the Lagrangian correspondences $L_{ij}$, $L_{jk}$, and $L_{ik}$.
As an outcome of Theorem \ref{thm:fukaya_bounding_cochain_transfer}, a choice of bounding cochains on the initial three Lagrangians $L_{01}$, $L_{12}$, and $L_{23}$ determines a bounding cochain on each Lagrangian $L_{ij}$, and this choice is such that the quilted Floer group
\begin{equation}
CF^*(L_{ij}, L_{jk}, L_{ik}^T)    
\end{equation}
admits a canonical cycle representing the equivalence between $L_{ik}$ and the composition of  $L_{ij}$ with $L_{jk}$ (we call this cycle the unit).
The essential point in showing that Figure \ref{fig:fukaya_jumpsuit} induces the desired equivalence is to compute that the count of constant pseudoholomorphic quilts with the given seam conditions, with the unit as input in three of the ends and as output at the third end, is exactly one, which is analogous to the fact that the unique constant disc with four marked points in fixed conformal position, passing through any point on a Lagrangian submanifold, is regular.
This fact ultimately reduces to the maximum principle for holomorphic functions, which shows that any holomorphic function on a disc, with value in $\bC^n$, and with boundary condition on $\bR^n$, must be constant.

\begin{remark}
In \S\ref{s:symp}, we explain an alternative conjectural approach to reprove these result, following the geometric ideas initiated by Wehrheim--Woodward.
Fukaya's proof of Theorem \ref{thm:fukaya_bounding_cochain_transfer} amazingly succeeds in bypassing all the analytic and geometric difficulties in the study of figure eight bubbling, and ultimately reduces it to an algebraic lemma about curved bimodules between curved $A_\infty$-algebras.
We expect that the complete functoriality package that we discuss in \S\ref{s:symp} can also be implemented in Fukaya's approach, but that doing so would require complicated arguments in the theory of curved algebras, which are further removed from the geometry of Floer theory than the quilted approach we discuss.
\null\hfill$\triangle$
\end{remark}

%%%% Local Variables:
%%%% mode: latex
%%%% TeX-master: "functoriality_in_categorical_symplectic_geometry"
%%%% End:

\section{The symplectic $(A_\infty,2)$-category $\Symp$}
\label{s:symp}

Wehrheim--Woodward's package of $A_\infty$-functors and homotopies
\begin{align} \label{eq:functor_and_homotopy}
\Phi^\#_{L_{12}}
\colon
\Fuk^\#(M_1) \to \Fuk^\#(M_2),
\qquad
\Psi^\#_{L_{12},L_{23}}
\colon
\Phi^\#_{L_{23}}\circ\Phi^\#_{L_{12}}
\sr{\simeq}{\lra}
\Phi^\#_{L_{12}\circ L_{23}}
\end{align}
represented a major step toward a notion of functoriality for the Fukaya category, but it faces fundamental restrictions because Wehrheim and Woodward only considered settings in which figure eight bubbles are a priori excluded.
This leads to an obvious question: is there a single algebraic framework that incorporates the composition operations in the Fukaya category, functors $\Phi_{L_{12}^\#}$, the homotopies $\psi_{L_{12},L_{23}}$, and a hypothetical algebraic operation defined by counting figure eight bubbles?

In this section, we describe a project of the second author and his collaborators which aims to achieve this goal and additionally constructs a coherent package for the functoriality of the Fukaya category, incorporating higher homotopies between the composition of the maps in \eqref{eq:functor_and_homotopy}.
This package is called the \emph{symplectic $(A_\infty,2)$-category}, denoted $\Symp$.
It is a chain-level version of a 2-category constructed by Wehrheim--Woodward in \cite{wehrheim2010functoriality}.

We begin with the observation that a number of pseudoholomorphic maps and quilts can be subsumed as instances of a more general family of quilts, called \emph{witch balls}.
Indeed, consider the pseudoholomorphic quilts depicted in the following figure, which we encountered while discussing the Fukaya category (\S\ref{ss:fuk}), the $A_\infty$-functor $\Phi_{L_{12}}\colon \Fuk M_1 \to \Fuk M_2$ (\S\ref{ss:approaching_Phi}), the $A_\infty$-homotopy $\Psi_{L_{12},L_{23}}\colon \Phi_{L_{23}}\circ\Phi_{L_{12}} \to \Phi_{L_{12}\circ L_{23}}$ (\S\ref{ss:def_of_Phi_L12}), and strip-shrinking (\S\ref{sss:strip-shrinking}), respectively.
\begin{figure}[H]
\centering
\def\svgwidth{0.9\columnwidth}
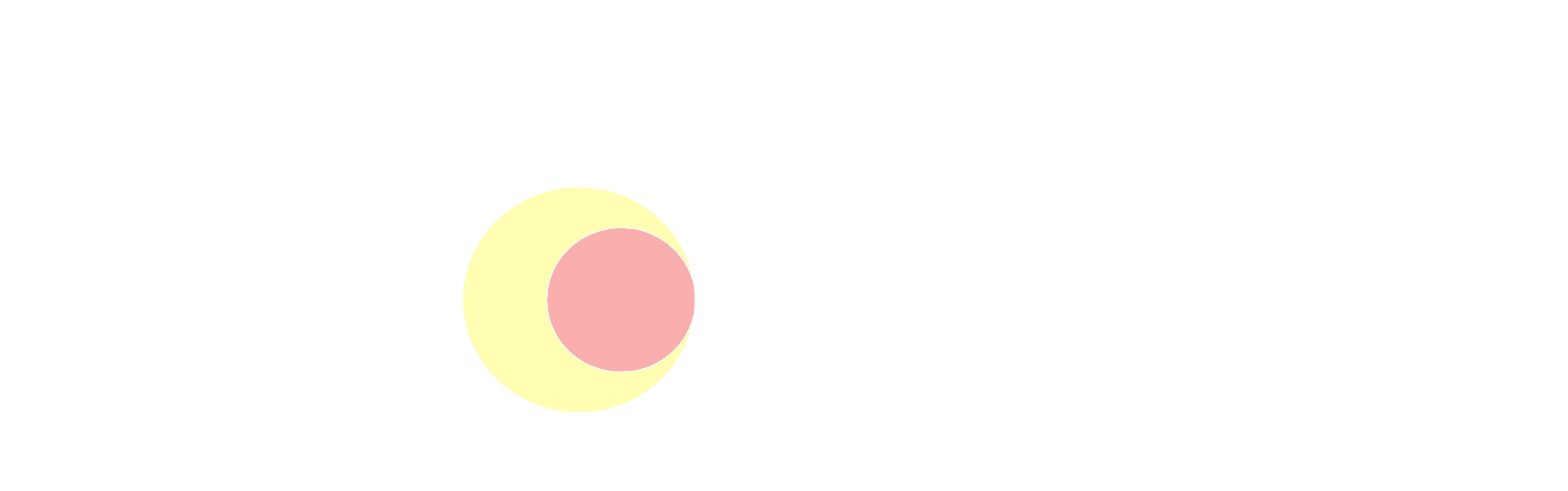
\caption{
\label{fig:witch_specializations}
}
\end{figure}

\noindent
Each of these quilts is an instance of a \emph{witch ball}, which is a quilt whose domain is depicted in the following figure (borrowed from \cite[p.\ 2]{bottman_2-associahedra}).
\begin{figure}[H]
\centering
\def\svgwidth{0.35\columnwidth}
%% Creator: Inkscape 1.2 (dc2aeda, 2022-05-15), www.inkscape.org
%% Accompanies image file 'witch_ball.pdf' (pdf, eps, ps)
\begingroup%
  \makeatletter%
  \providecommand\color[2][]{%
    \errmessage{(Inkscape) Color is used for the text in Inkscape, but the package 'color.sty' is not loaded}%
    \renewcommand\color[2][]{}%
  }%
  \providecommand\transparent[1]{%
    \errmessage{(Inkscape) Transparency is used (non-zero) for the text in Inkscape, but the package 'transparent.sty' is not loaded}%
    \renewcommand\transparent[1]{}%
  }%
  \providecommand\rotatebox[2]{#2}%
  \newcommand*\fsize{\dimexpr\f@size pt\relax}%
  \newcommand*\lineheight[1]{\fontsize{\fsize}{#1\fsize}\selectfont}%
  \ifx\svgwidth\undefined%
    \setlength{\unitlength}{319.89737688bp}%
    \ifx\svgscale\undefined%
      \relax%
    \else%
      \setlength{\unitlength}{\unitlength * \real{\svgscale}}%
    \fi%
  \else%
    \setlength{\unitlength}{\svgwidth}%
  \fi%
  \global\let\svgwidth\undefined%
  \global\let\svgscale\undefined%
  \makeatother%
  \begin{picture}(1,0.41149468)%
    \lineheight{1}%
    \setlength\tabcolsep{0pt}%
    \put(0,0){\includegraphics[width=\unitlength,page=1]{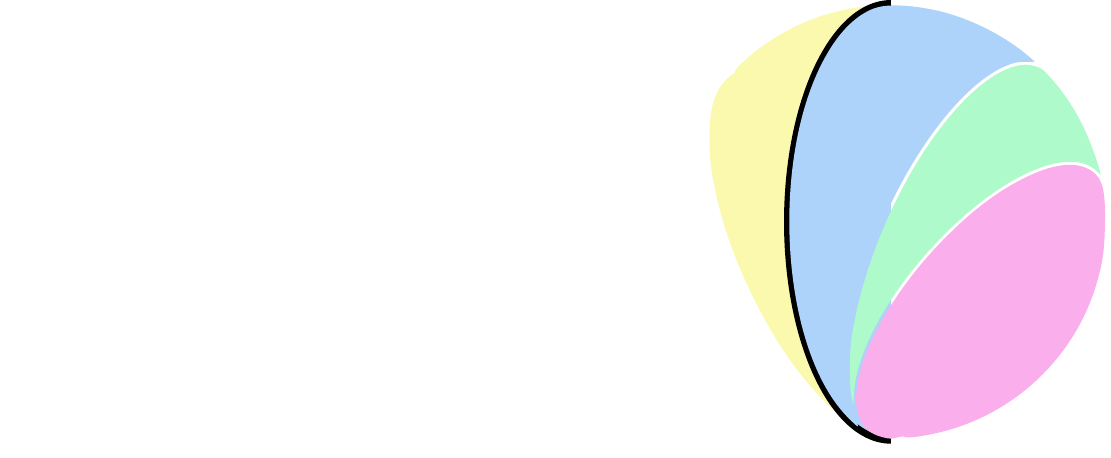}}%
    \put(0.02449367,0.22024361){\color[rgb]{0,0,0}\makebox(0,0)[lt]{\lineheight{0}\smash{\begin{tabular}[t]{l} \end{tabular}}}}%
    \put(0,0){\includegraphics[width=\unitlength,page=2]{witch_ball.pdf}}%
  \end{picture}%
\endgroup%

\caption{Two equivalent views of the domain of a witch ball.
On the left, we depict $\bR^2$, divided into patches by vertical lines with marked points.
On the right, we compactify $\bR^2$ to $S^2$; the lines become circles that intersect at the south pole.}
\label{fig:two_pictures_for_witch_ball}
\end{figure}
\noindent
The domain moduli spaces of witch balls are called \emph{2-associahedra}, and are denoted $\ol{2M}_\bn$ or $W_\bn$ depending on whether we are referring to the stratified topological space or to the poset of strata; the indexing set $\bn$ is a sequence of natural numbers that records the number ofpoints on each seam.

The symplectic $(A_\infty,2)$-category $\Symp$ is the structure that emerges from counting witch balls.
We will now give a blueprint for this structure.
After that, we will explain what portions of this structure have been defined and what parts remain to be constructed.
In the subsequent subsections, we will delve into the details of the components of $\Symp$.

\medskip

\noindent
{\bf Blueprint for $\Symp$, and a roadmap of which parts have and have not been completed.}
$\Symp$ is an $(A_\infty,2)$-category consisting of the following data:
\begin{itemize}
\item
The category $\Symp_1$, whose objects are symplectic manifolds $(M,\omega)$ and where $\hom(M_1,M_2)$ is the set of Lagrangian correspondences $M_1 \sr{L_{12}}{\lra} M_2$.

\smallskip

\item
For each pair of 1-morphisms, i.e.\ Lagrangian correspondences $M_1 \sr{L_{12},L_{12}'}{\lra} M_2$, the Floer cochain complex $CF^*(L_{12},L_{12}')$ of \emph{2-morphisms from $L_{12}$ to $L_{12}'$}.

\smallskip

\item
For each $r \geq 1$ and $\bm \in \bZ_{\geq0}^r\setminus\{\bzero\}$, for each sequence $M_0,\ldots,M_r$, and for each collection of sequences of Lagrangian correspondences
\begin{align*}
L_{01}^0,\ldots,L_{01}^{m_1} 
\subset &
M_0^-\times M_1, \\ 
 & \ldots, \\
L_{(r-1)r}^0,\ldots,L_{(r-1)r}^{m_r}
\subset &
M_{r-1}^-\times M_r,
\end{align*}
a composition map
\begin{multline}
2c_\bm\colon
C_*^\sing(\ol{2M}_\bn)
\otimes
\bigotimes_{{1\leq i\leq r,}
\atop
{1 \leq j\leq m_i}}
CF^*\bigl(L_{(i-1)i}^{j-1},L_{(i-1)i}^j\bigr)
\\
\lra
CF^*\bigl(
L_{01}^0\circ\cdots\circ L_{(r-1)r}^0,
L_{01}^{m_1}\circ\cdots\circ L_{(r-1)r}^{m_r}\bigr).
\end{multline}
\end{itemize}
\null\hfill$\triangle$

We now describe the current status of progress toward the definition of $\Symp$.

\begin{itemize}
\item
As we will explain in \S\S\ref{ss:witch_curves}--\ref{ss:2-associahedra}, the second author defined in \cite{bottman_2-associahedra} the 2-associahedra $\ol{2M}_\bn$ in terms of two equivalent models $W_\bn^\tree \simeq W_\bn^\br \eqqcolon W_\bn$.
He established the basic combinatorial properties of the 2-associahedra, in particular that $W_\bn$ is an abstract polytope with a recursive structure.
Next, in \cite{bottman_realizations}, the second author constructed the moduli spaces $\ol{2M}_\bn$ of witch curves and established their basic topological properties.
These spaces will form the domain spaces for the maps whose counts define $\Symp$.
In \cite{bottman_oblomkov}, Bottman--Oblomkov upgraded the topological structure on $\ol{2M}_\bn$ to a smooth structure --- specifically, they equipped $\ol{2M}_\bn$ with the structure of a smooth manifold with g-corners in the sense of \cite{joyce}.

\smallskip

\item
As we will explain in \S\ref{ss:relative_2-operad}, Bottman and Carmeli defined in \cite{bottman_carmeli} the notion of a relative 2-operad.
They showed that $\bigl(\ol{2M}_\bn\bigr)$ forms a 2-operad relative to $\bigl(\ol M_r\bigr)$, and used this to define an $(A_\infty,2)$-category to be a 2-category-like object in which there is an operation on 2-morphisms associated to every singular chain on $\ol{2M}_\bn$.
($\ol M_r$ is the topological instantiation of the $(r-2)$-dimensional associahedron, which we referred to as $K_r$ earlier in this paper.)

Depending on the particular regularization theory, one might hope for a definition of $\Symp$ in which operations on 2-morphisms are associated to \emph{cellular} chains on $\ol{2M}_\bn$.
This requires a nontrivial modification to the definition of an $(A_\infty,2)$-category.
The second author is currently developing such a modification.

\smallskip

\item
As we will describe in \S\ref{ss:analysis}, Bottman and Wehrheim established two basic analysis results necessary for it to be conceivable to define a curve-counting theory via witch balls.
Specifically, Bottman proved a removal-of-singularity theorem in \cite{bottman_figure_eight_singularity} via a collection of width-independent elliptic inequalities, and Bottman and Wehrheim built on this in \cite{bottman_wehrheim} to establish a Gromov compactness result for moduli spaces of witch balls.

\smallskip

\item
The major step toward $\Symp$ that has not yet been completed is the construction of a regularization theory for moduli spaces of witch balls.
While a number of regularization theories for moduli spaces of pseudoholomorphic curves exist, none of these theories currently allow for families of quilts involving colliding seams.
The second author and Katrin Wehrheim are currently working on an approach to this task via the theory of polyfolds.

In addition, there is the issue of Lagrangian correspondences that do not have transversely-defined composition, and those that do, but for which the composition is immersed, rather than embedded.
\end{itemize}

\begin{remark}
Here, we describe $\Symp_1$ as an ordinary category, which implies that composition of 1-morphisms $L_{12} \circ L_{23}$ is strictly associative.
To carry out the construction of $\Symp_1$ in generality, this will need to be relaxed to allow homotopy-associative composition.
The reason for this is the figure-eight bubbling discussed in \S\ref{sss:strip-shrinking}.
\null\hfill$\triangle$
\end{remark}

\subsection{The domain moduli spaces of witch curves}
\label{ss:witch_curves}

The Operadic Principle described in \S\ref{ss:fuk} tells us that to understand the structure that results from counting witch balls, we must first define and understand the relevant compactified domain moduli spaces.
These moduli spaces are indexed by a non-negative integer $r \geq 1$ which records the number of vertical lines appearing in the left part of Figure \ref{fig:two_pictures_for_witch_ball}, and a sequence $\bn \in \bZ_{\geq0}^r\setminus\{\bzero\}$ of integers which records the number of marked points on each vertical line.

There is an open moduli space $2M_\bn$ associated to these data that parametrizes \emph{witch curves}, i.e.\ configurations of vertical lines in $\bR^2$ equipped with marked points; we identify two configurations if they differ by an overall translation and positive dilation.
This moduli space is not compact, because points on a single line can collide, or lines can collide.
We compactify $2M_\bn$ to a space $\ol{2M}_\bn$ of \emph{nodal witch curves} like so: when a collection of lines collide, then wherever the marked points on these lines are as this collision happens, we bubble off another configuration of lines and points.
This compactified moduli space is called the \emph{$\bn$-th 2-associahedron}.
To define $\ol{2M}_\bn$, we need to specify the allowed degenerations, and this is where the 2-associahedra come in: for $r \geq 1$ and $\bn \in \bZ_{\geq0}^r\setminus\{\bzero\}$ we define the \emph{2-associahedron} $W_\bn$ to be the poset of degenerations in $\ol{2M}_\bn$.

There are two combinatorial models for $W_\bn$, which take the form of isomorphic posets $W_\bn^\tree$ and $W_\bn^\br$.
These models are completely analogous to the models $K_r^\tree$ and $K_r^\br$ for the associahedra:
\begin{itemize}
\item
$W_\bn^\tree$ consists of \emph{tree-pairs} $T_b \to T_s$, where $T_b$ resp.\ $T_s$ are planted trees called the \emph{bubble tree} resp.\ \emph{seam tree}.
When identifying the strata of $\ol{2M}_\bn$ with the elements of $W_\bn^\tree$, it is straightforward to go from a tree of spheres to the bubble tree $T_b$: $T_b$ has a cluster of solid edges for every sphere, where the number of solid edges corresponds to the number of seams on that sphere.
$T_b$ has a dashed edge for every attachment point and marked point in the tree of spheres.
The seam tree $T_s$ keeps track of how the seams have collided --- and, importantly, enforces certain coherences, as we will explain below.

\smallskip

\item
$W_\bn^\br$ consists of \emph{2-bracketings}.
Roughly, a 2-bracketing encodes the data of a witch tree by including a 2-bracket for every sphere.
For each sphere, the corresponding 2-bracket contains the information of all the lines and marked points that are either on that sphere, or on a sphere farther from the root.
\end{itemize}
We illustrate this in the following figure (borrowed from \cite[p.\ 3]{bottman_2-associahedra}): on the left is the compactified moduli space $\ol{2M}_{200}$, and in the middle and on the right are two presentations of $W_{200}$.
These figures are not intended to be understandable just yet, but they will serve as points of reference as we describe $\ol{2M}_\bn$, $W_\bn^\tree$, and $W_\bn^\br$ in more detail.

\begin{figure}[H]
\centering
\def\svgwidth{1.0\columnwidth}
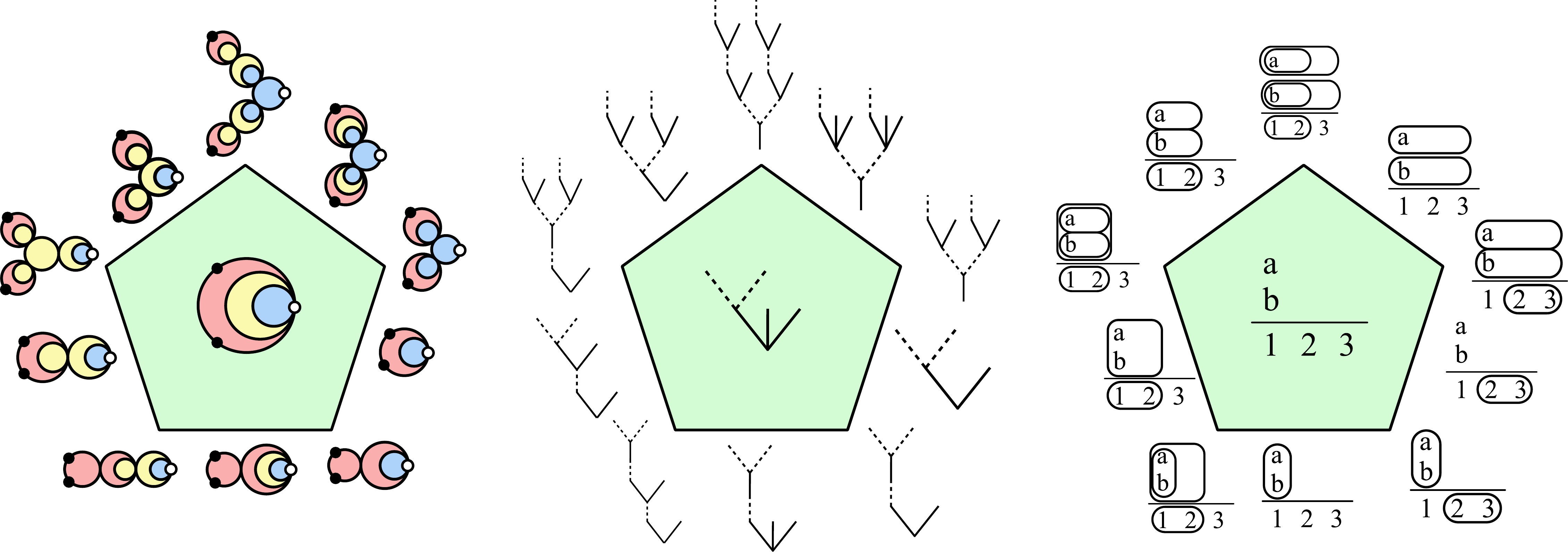
\caption{
\label{fig:rosetta}
}
\end{figure}

\noindent
In the current subsection, we will focus on the domain moduli spaces $\ol{2M}_\bn$.
In \S\ref{ss:2-associahedra}, we will return to the posets $W_\bn^\tree, W_\bn^\br$.

\begin{remark}
\label{rmk:quilted_disks_vs_spheres}
We defined $\ol{2M}_\bn$ to be a moduli space of quilted spheres, but in Figure \ref{fig:rosetta}, we have labeled its strata by representative quilted disks.
The reason is that an element of $2M_\bn$ can be identified with a quilted disk, by excising the left-most patch of the quilted sphere; moreover, when $\bn$ is of the form $\bn = (n_1,0,\ldots,0)$, this extends to the boundary to yield an identification of an element of $\ol{2M}_\bn$ with a nodal quilted disk.
\null\hfill$\triangle$
\end{remark}

\medskip

We begin by being describing in $\ol{2M}_\bn$ in more detail, while still falling short of providing a complete description, which will require the definitions of $W_\bn^\tree$ and $W_\bn^\br$ (but to understand the definitions of these posets, it is helpful to first have some intuition about $\ol{2M}_\bn$!).
First, we record the definition of its interior, $2M_\bn$:
\begin{align}
2M_\bn
\coloneqq
\left\{
{\ell_1, \ldots, \ell_r \text{ vertical lines in } \bR^2 \text{ ordered from left to right,}}
\atop
{p_{i1},\ldots,p_{in_i} \text{ points on } \ell_i, \text{ ordered from top to bottom}}
\right\}\bigg/_{\bR^2 \rtimes \bR_{>0}},
\end{align}
where these lines and points are required to be distinct and where we identify two configurations that differ by an element of the group $\bR^2 \rtimes \bR_{>0}$ of automorphisms of the plane generated by translations and positive dilations.
$2M_\bn$ is not compact, because lines or points can collide.
To define a theory by counting witch balls, we therefore need to compactify this domain moduli space.
Bottman Gromov-compactified $2M_\bn$ to form $\ol{2M}_\bn$, according to the following paradigm:

\smallskip

\begin{quotation}
\it
When a marked points $p_{ij}$ collides with either another marked point $p_{i'j'}$ or with a line $\ell_{i'}$, we resolve this collision by ``bubbling off'' a new copy of $\bR^2$, which we obtain by zooming in at the collision point $c$ at a rate commensurate to the slowest collision occurring at $c$.
If in this zoomed-in view there is still a collision, we resolve this collision in the same way, and so on inductively.
\end{quotation}

\smallskip

A detailed example of this compactification process is given in \cite[\S1.1]{bottman_realizations}.
We summarize this example now.
For $\eps \in (0,\tfrac 12)$, take the following configuration in $\ol{2M}_{10010}$ (pictured with $\eps = 2/5$):

\vspace{-0.5em}
\begin{figure}[H]
\centering
\def\svgwidth{0.325\columnwidth}
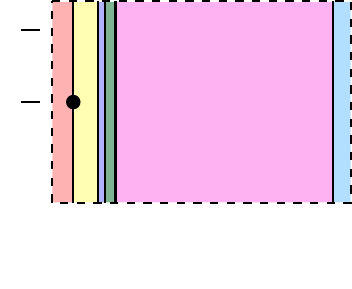
\caption{}
\end{figure}
\vspace{-1em}

\noindent
As $\eps \to 0$, all lines except the right-most one collide, and the two marked points also collide.
The limit is defined by inductively rescaling on the marked points involved in collisions, according to the paradigm in italics above.
We depict this limit below:

\begin{figure}[H]
\centering
\def\svgwidth{1.0\columnwidth}
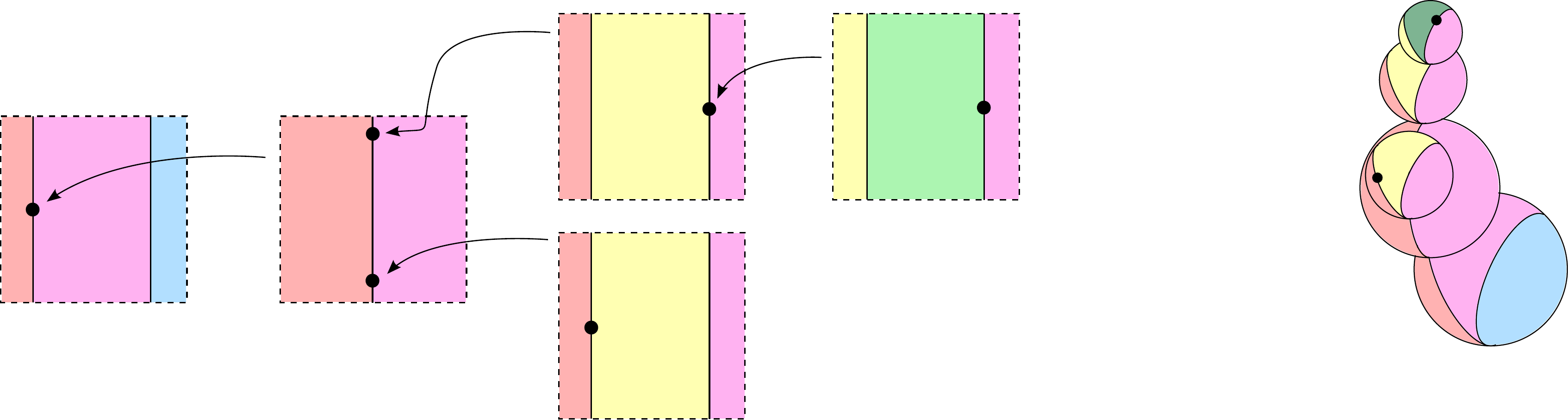
\caption{
Two equivalent depictions of the limit.
On the left, we have a tree of configurations in $\bR^2$.
On the right, we have a tree of configurations in $S^2$.
}
\end{figure}

A subtle but important aspect of $\ol{2M}_\bn$ is that unlike $\ol M_r$, the strata of $\ol{2M}_\bn$ do not decompose as products of lower-dimensional instances of $\ol{2M}_\bn$!
We will approach this through the example of $\ol{2M}_{200}$, as in Figure \ref{fig:rosetta}.
Specifically, we consider the upper-right stratum of that pentagon,
which corresponds to the degeneration where the three lines collide commensurately.
We see that we can naturally identify the codimension-1 stratum resulting from degenerations of this form with the locus inside $\ol{2M}_2 \times \ol{2M}_{100} \times \ol{2M}_{100}$ where the positions of the lines on the two bubbled-off screens agree, up to translation and dilation.
Equivalently, if we denote by $\pi\colon \ol{2M}_{100} \to \ol M_3$ the forgetful map which remembers the $x$-positions of the lines, we have identified this stratum with the fiber product $\ol{2M}_2 \times \ol{2M}_{100} \times_{\ol M_3} \ol{2M}_{100}$.
We will return to this aspect of the 2-associahedra in \S\ref{ss:relative_2-operad}.

The full construction of $\Symp$ will depend on the choice of an abstract perturbation scheme.
At least one of these schemes --- based on the polyfold theory developed by Hofer, Wysocki, Zehnder, and Fish --- requires smooth structures on the domain moduli spaces.
In the case of $\Symp$, the relevant domain moduli spaces are the spaces $\ol{2M}_\bn$ of witch curves.
It turns out that endowing $\ol{2M}_\bn$ with a smooth structure is a nontrivial proposition, because $\ol{2M}_\bn$ cannot be a smooth manifold with boundary and corners in a way compatible with its natural stratification.
The first example where we can see this is in the 3-dimensional space $\ol{2M}_{40}$ corresponding to two seams, one of which carries four marked points.
Consider the portion of $\partial\bigl(\ol{2M}_{40}\bigr)$ depicted below:

\medskip

\begin{figure}[H]
\centering
\def\svgwidth{0.35\columnwidth}
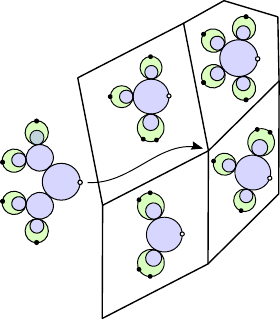
\caption{
\label{fig:bad_corner}
}
\end{figure}

\noindent
This configuration cannot appear in a 3-dimensional manifold with boundary and corners, because $\ol{2M}_{40}$ is not locally diffeomorphic to $[0,1)^3$ at the corner depicted here.

Nevertheless, one can equip $\ol{2M}_{40}$ with a smooth structure.
To approach this, we recall the main result from \cite{bottman_oblomkov}.
This result is concerned with $\ol{2M}_\bn^\bC$, which is a ``complexification'' of $\ol{2M}_\bn$: $\ol{2M}_\bn^\bC$ is a compactified moduli space of configurations of vertical complex lines in $\bC^2$, up to complex dilations and translations.

\medskip

\noindent
{\bf Theorem 1.1, \cite{bottman_oblomkov}.}
{\it
$\ol{2M}_\bn^\bC$ is a proper complex variety with toric singularities.
There is a forgetful morphism $\pi\colon \ol{2M}_\bn^\bC \to \ol M_{0,r+1}$, which on the open locus sends a configuration of lines and points to the positions of the lines, thought of as a configuration of points in $\bC$.
\null\hfill$\square$
}

\medskip

\noindent
Bottman--Oblomkov's result has a direct implication for $\ol{2M}_\bn$.
This implication did not appear in their paper, but it is a straightforward translation from the complex to the real picture.

\medskip

\noindent
{\bf Corollary of Theorem 1.1, \cite{bottman_oblomkov}.}
{\it
There is a canonical way to endow $\ol{2M}_\bn$ with the structure of a (compact) manifold with g-corners, in the sense of \cite{joyce}.
\null\hfill$\square$
}

\medskip

\noindent
A manifold with g-corners (short for ``generalized corners'') can be thought of as a manifold with a smooth structure (in particular, with a well-behaved notion of tangent bundle) that is modeled on polytopes that are not necessarily simple.
Alternately, one can think of a manifold with g-corners as the differential-topological, positive-real analogue of a toric variety.

\begin{remark}
\label{rem:W_n_manifold_with_boundary}
An immediate consequence of this corollary is that $\ol{2M}_\bn$ is a topological manifold with boundary.
In fact, if $X$ is a compact topological manifold with boundary whose interior is homeomorphic to $\bR^k$, then $X$ is a closed ball.
It follows that $\ol{2M}_\bn$ is homeomorphic to $\ol B^{|\bn|+r-3}$.
\null\hfill$\triangle$
\end{remark}

\subsection{The combinatorial models $W_\bn^\tree$ and $W_\bn^\br$}
\label{ss:2-associahedra}

In \S\ref{ss:witch_curves}, we described the domain moduli spaces $\ol{2M}_\bn$ for $\Symp$, which parametrize witch curves.
The precise definition of $\ol{2M}_\bn$ proceeds by first defining the stratum corresponding to each element of $W_\bn$, and then defining a topology on the union of these strata by formulating a notion of a ``Gromov-convergent sequence''.
In this subsection, we will describe the equivalent posets $W_\bn^\tree \simeq W_\bn^\br \eqqcolon W_\bn$.
The precise definitions of $W_\bn^\tree$ and $W_\bn^\br$ can be found in Definitions 3.1 and 3.11 of \cite{bottman_2-associahedra}.
These definitions are rather technical and by now are well-established, so we will limit ourselves to sketching them.

We begin with $W_\bn^\tree$.
Recall from \S\ref{ss:associahedra_and_OP} that $K_r^\tree$ consists of stable rooted planar trees.
Moreover, recall that there is a correspondence between $K_r^\tree$ and the combinatorial type of a nodal tree of disks: one replaces each disk by a vertex, adds an interior edge for every nodal point, and adds an exterior edge for every boundary marked point.
$W_\bn^\tree$ is an analogous construction, but where we are summarizing the combinatorial information of a witch curve instead of a nodal tree of disks.
We must now incorporate data that record the sphere components, the seams, the attachment point between spheres, and the marked points that appear on seams.
Given a witch curve, we translate it into a tree-like object by doing the following:
\begin{itemize}
\item
Replace a sphere with $k$ seams by a corolla of $k$ solid edges (i.e.\ a vertex with $k$ attached edges).

\smallskip

\item
If the south pole of a sphere is attached to the $i$-th seam of another sphere, add a dashed edge connecting the bottom point of the first sphere's corolla to the top point of the $i$-th solid edge in the second sphere's corolla.

\smallskip

\item
If a sphere has $j$ input marked points on its $i$-th seam, add an exterior dashed edge attached to the top point of the $i$-th solid edge in the sphere's corolla.
\end{itemize}

\medskip

A naive definition of $W_\bn^\tree$ might allow for all trees whose edges alternate between solid and dashed, subject to a suitable stability condition.
However, this would include many trees that do not correspond to degenerations in $\ol{2M}_\bn$.
The key to defining trees corresponding to legitimate degenerations is to introduce an auxiliary tree, the \emph{seam tree}; to disambiguate, we refer the tree we had been considering as the \emph{bubble tree}.
The seam tree is a stable rooted planar tree, which tracks hown the seams have collided.
It enforces the necessary coherences in the bubble tree.
In particular, the seam tree is necessary in order to produce a poset of degenerations which has the recursive structure described in \S\ref{ss:witch_curves}, where strata naturally decompose as products of fiber products.

\medskip

\noindent
{\bf Sketch definition of $W_\bn^\tree$, the model for $W_\bn$ consisting of stable tree-pairs.}
$W_\bn^\tree$ is the set of stable tree-pairs of type $\bn$, where the latter is a datum $2T = T_b \sr{f}{\to} T_s$.
\begin{itemize}
\item The \emph{bubble tree} $T_b$ is a planar rooted tree whose edges alternate between solid and dashed.
We impose a stability condition which corresponds to a fact that a screen with one seam has finitely many automorphisms if and only if the seam has at least two marked points, and a screen with two or more seams has finitely many automorphisms if and only if there is at least one marked seam.

\smallskip

\item
The \emph{seam tree} $T_s$ is a planar rooted tree.

\smallskip

\item
The \emph{coherence map} is a map $f\colon T_b \to T_s$ of trees, which contracts all dashed edges and all solid corollas with only a single solid edge.
Every solid corolla with $k \geq 2$ edges is mapped bijectively by $f$ to a corolla in $T_s$ with $k$ edges.
\null\hfill$\triangle$
\end{itemize}

\medskip

\noindent
The middle pentagon in Figure \ref{fig:rosetta} illustrates all stable tree-pairs in the case $\bn = (2,0,0)$.

\medskip

Next, we turn to the model $W_\bn^\br$, which a posteriori is equivalent to $W_\bn^\tree$.
Recall from \S\ref{ss:associahedra_and_OP} that $K_r^\br$ consists of legal bracketings of $1,\ldots, r$.
Moreover, recall that there is a correspondence between $K_r^\br$ and the combinatorial type of a nodal tree of disks: for every disk, one tabulates all the input marked points that are either on this disk or on a disk further from the output marked point, and includes a bracket containing the corresponding numbers.
We define $W_\bn^\br$ in an analogous fashion.
Given a witch curve, we translate it into a \emph{2-bracketing} by doing the following:
\begin{itemize}
\item
Label the seams from left to right by $1, \ldots, r$.
Label the marked points on the $i$-th seam from bottom to top by $1, \ldots, n_i$.

\smallskip

\item
Consider one of the sphere components $C$ in the witch curve we are considering.
Define $B$ to be the subset of $\{1,\ldots,r\}$ corresponding to the seams that appear either on $C$ or on a sphere component further from the output marked point than $C$.
Similarly, for every $i \in B$, define $2B_i$ to be the subset of $\{1,\ldots,n_i\}$ corresponding to the input marked points that appear either on $C$ or on a sphere component further from the output marked point than $C$.
$\bigl(B,(2B_i)_i\bigr)$ is the \emph{2-bracket corresponding to $C$}.

\smallskip

\item
Define the resulting 2-bracketing to be the collection of 2-brackets corresponding to sphere components in the given witch curve, together with the information of the order in which bubbles appear along each seam.
\end{itemize}

\medskip

\noindent
In fact, the resulting 2-bracketing also contains the information of the fashion in which the seams collided, which takes the form of a bracketing in $K_r^\br$.
This information is redundant unless there are collisions of unmarked seams.

\medskip

\noindent
{\bf Sketch definition of $W_\bn^\br$, the model for $W_\bn$ consisting of 2-bracketings.}
A \emph{2-bracket of $\bn$} is a pair $\btB = (B, (2B_i))$\label{p:btB} consisting of a 1-bracket $B \subset \{1,\ldots,r\}$ and a consecutive subset $2B_i \subset \{1,\ldots,n_i\}$ for every $i \in B$ such that at least one $2B_i$ is nonempty.
$W_\bn^\br$ is the set of 2-bracketings of $\bn$.
The latter is a pair $(\sB, \stB)$, where $\sB$ is a bracketing of $r$ and $\stB$ is a collection of 2-brackets of $\bn$ (together with partial orders reflects the order in which bubbles appear on seams) that satisfies these properties:
\begin{itemize}
\item
The 2-brackets in $\stB$ are nested, in the sense that if two have nonempty intersection, one must contain the other.
Moreover, if $\bigl(B,(2B_i)_i\bigr)$ is a 2-bracket in $\stB$, then $B$ must be an element of $\sB$.

\smallskip

\item
We impose technical conditions that are too complicated to state precisely here, but which reflect (i) the fact that marked points can only appear on unfused seams and (ii) the natural coherences amongst partial orders.
\null\hfill$\triangle$
\end{itemize}

\medskip

\noindent
The right pentagon in Figure \ref{fig:rosetta} illustrates all 2-bracketings in the case $\bn = (2,0,0)$.

Finally, we summarize the properties of the 2-associahedra.

\medskip

\noindent
{\bf Theorem 4.1, \cite{bottman_2-associahedra}.}
{\it
For any $r \geq 1$ and $\bn \in \bZ^r_{\geq0}\setminus\{\bzero\}$, the 2-associahedron $W_\bn$ is a poset, the collection of which satisfies the following properties:
\begin{itemize}
\item[] \textsc{(abstract polytope)}
$\wh{W_\bn}$, which denotes $W_\bn$ with a formal, minimal element of dimension $-1$, is an abstract polytope of dimension $|\bn| + r - 3$.

\smallskip

\item[] \textsc{(forgetful)}
$W_\bn$ is equipped with a \emph{forgetful map} $\pi\colon W_\bn \to K_r$, which is a surjective map of posets.

\smallskip

\item[] \textsc{(recursive)}
Each closed face of $W_\bn$ decomposes as a canonical way as a product of fiber products of lower-dimensional 2-associahedra, where the fiber products are with respect to the forgetful maps $\pi$.
\null\hfill$\square$
\end{itemize}
}

\medskip

\begin{remark}
The \textsc{(abstract polytope)} property says that $W_\bn$ shares several combinatorial properties with face posets of convex polytopes of dimension $|\bn|+r-3$.
It is not directly relevant to the construction of $\Symp$, but it does direct our intuition.

The forgetful map $\pi\colon W_\bn \to K_r$ has simple descriptions in both models: $\pi^\tree\colon W_\bn^\tree \to K_r^\tree$ remembers only the seam tree, while $\pi^\br\colon W_\bn^\br \to K_r^\br$ remembers the underlying 1-bracketing.
\null\hfill$\triangle$
\end{remark}

\subsection{Relative 2-operads and $(A_\infty,2)$-categories}
\label{ss:relative_2-operad}

As we described in \S\ref{ss:fuk}, the associahedra form an operad in the category of topological spaces.
Since the structure maps of the Fukaya category are defined by counting nodal disks, the Fukaya category is an $A_\infty$-category.

We might hope that a similar story holds for $\Symp$.
However, the collection of spaces $\bigl(\ol{2M}_\bn\bigr)_{r\geq 1,\bn \in \bZ_{\geq0}^r\setminus\{\bzero\}}$ does not form an operad.
At the most basic level, these spaces are not indexed by the positive integers.
A more nontrivial reason that these domain moduli spaces do not form an operad is that, as we explained in \S\ref{ss:witch_curves}, the strata of $\ol{2M}_\bn$ decompose as products of \emph{fiber products of} lower-dimensional 2-associahedra.
In fact, 2-associahedra form a 2-categorical version of an operad, called a \emph{relative 2-operad}.
This structure was defined by Bottman--Carmeli in \cite{bottman_carmeli} with the specific example of $\bigl(\ol{2M}_\bn\bigr)$ in mind.

\medskip

\noindent
{\bf Definition-Proposition 2.3, \cite{bottman_carmeli}.}
The 2-associahedra $(\ol{2M}_\bn)$, together with the forgetful maps $\pi\colon \ol{2M}_\bn \to \ol M_r$ and certain of the structure maps $\Gamma_{2T}$, form a 2-operad relative to the realized associahedra $(\ol M_r)$.
\null\hfill$\triangle$

\medskip

\noindent
The relative 2-operadic structure of the 2-associahedra consists of the forgetful maps from the 2-associahedra to the associahedra noted in Theorem 4.1 from \cite{bottman_2-associahedra} (mentioned above), along with the structure maps established in the same theorem from products of fiber products of 2-associahedra to other 2-associahedra.
We illustrate one of these structure maps in the case of $\ol{2M}_{300}$ below (figure borrowed from \cite[p.\ 6]{bottman_2-associahedra}):

\begin{figure}[H]
\centering
\def\svgwidth{0.9\columnwidth}
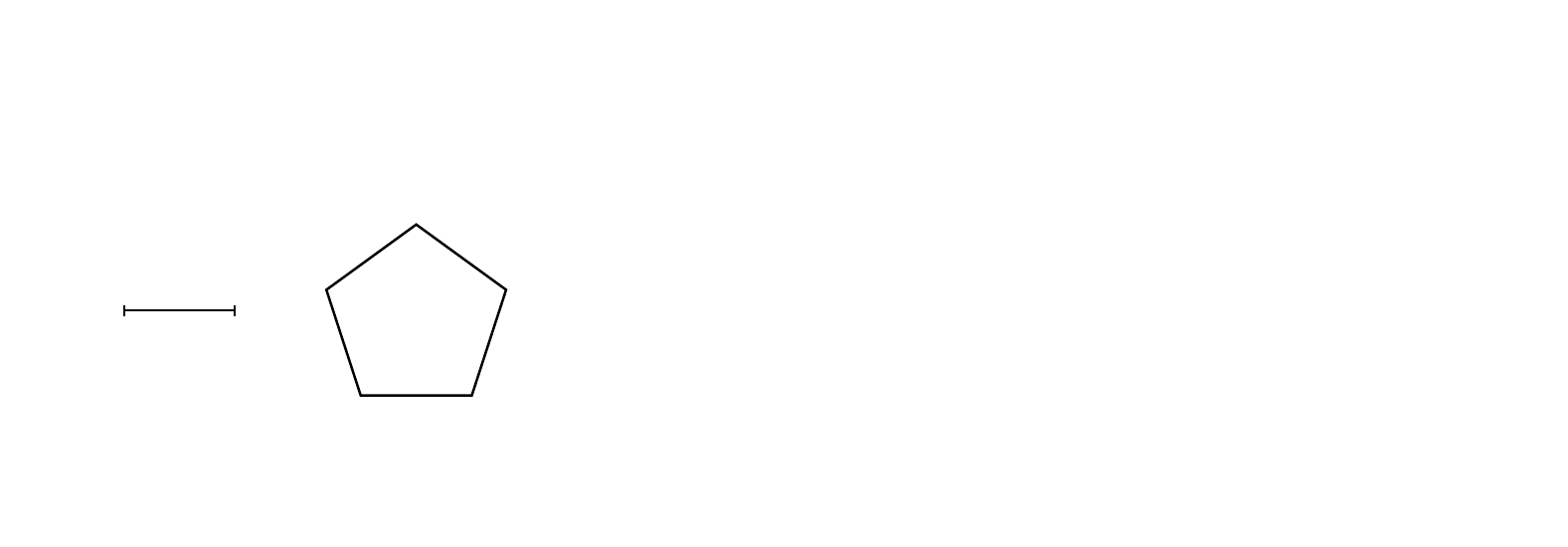
\caption{}
\end{figure}

\noindent
The relative 2-operadic structure of the 2-associahedra enabled the second author and Carmeli to define in \cite{bottman_carmeli} the notion of an \emph{$(A_\infty,2)$-category}.
When reading the following definition, one should keep in mind the primordial relative 2-operad, $\bigl(\bigl(\ol M_r\bigr), \bigl(\ol{2M}_\bn\bigr)\bigr)$.

\begin{definition}
\label{def:2op}
A \emph{(nonsymmetric) relative 2-operad} in a category $\cC$ with finite limits is a pair
\begin{align}
\Bigl((P_r)_{r\geq 1}, (Q_\bm)_{\bm \in \bZ^r_{\geq0}\setminus\{\bzero\},r \geq 1}\Bigr),
\end{align}
where $(P_r)_{r\geq 1}$ is a nonsymmetric operad in $\cC$, and where $(Q_\bm) \subset \cC$ is a collection of objects together with a family of structure morphisms
\begin{gather}
\label{eq:Gamma_def}
\Gamma_{\bm,(\bn^a_i)}
\colon
Q_\bm \times \prod_{1\leq i\leq r} \prod^{P_{s_i}}_{1\leq a\leq m_i} Q_{\bn_i^a} \to Q_{\sum_a \bn_1^a,\ldots,\sum_a \bn_r^a},
\\
r, s_1, \ldots s_r \geq 1,
\quad
\bm \in \bZ_{\geq0}^r\setminus\{\bzero\},
\quad
\bn_i^a \in \bZ_{\geq0}^{s_i}\setminus\{\bzero\}.
\nonumber
\end{gather}
(Here the subscript in $Q_{\sum_a \bn_1^a,\ldots,\sum_a \bn_r^a}$ denotes the concatenation of $\sum_a \bn_1^a$, $\sum_a \bn_2^a$, etc., which is a vector of length $\sum_i s_i$, and the superscript $ P_{s_i}$ indicates that we are taking a fibre product.)

We require these objects and morphisms to satisfy the following axioms.
\begin{itemize}
\item[] {\sc(projections)} $\bigl((P_r),(Q_\bm)\bigr)$ is equipped with projection morphisms
\begin{align}
\pi_\bm\colon Q_\bm \to P_r,
\qquad
r \geq 1,
\:
\bm \in \bZ^r_{\geq0}\setminus\{\bzero\}
\end{align}
that intertwine the structure morphisms $\Gamma_{\bm,(\bn_i^a)}$ with those in the underlying operad $(P_r)$.

\smallskip

\item[] {\sc(associative), (unit)} 
The structure maps satisfy a straightforward associativity condition, and there is a unit map $1 \to Q_1$ satisfying the obvious properties.
\null\hfill$\triangle$
\end{itemize}
\end{definition}

\begin{definition}
\label{def:lin_cat_over_rel_2-operad}
Let $\bK$ be a ring.
An \emph{$\bK$-linear category over a  relative 2-operad $\bigl((P_r),(Q_\bn)\bigr)$ in \textsf{Top}} consists of: 
\begin{itemize}
\item A category $(\Ob,\Mor,s,t)$.

\smallskip

\item For each pair of morphisms $L,K\colon M\to N$, a complex of free $\bK$-modules $2\Mor(L,K)$.

\smallskip

\item
Composition maps
\begin{align}
C_*(Q_\bm)\otimes \bigotimes_{{1\leq i \leq r}\atop{1\leq j\leq m_i}}  2\Mor(L_i^{j-1},L_i^j)\to 2\Mor(L_1^0\circ\cdots\circ L_r^0,L_1^{m_1}\circ\cdots\circ L_r^{m_r}),
\end{align}
where $C_*(Q_\bm)$ denotes the complex of singular chains in $Q_\bm$ with coefficients in $\bK$.
\end{itemize}
\noindent
We require the composition maps to be associative.\footnote{In fact, this condition is not entirely straightforward.
We invite the interested reader to consult \cite{bottman_carmeli}.}
\null\hfill$\triangle$
\end{definition}

\begin{definition}
\label{def:A_infty_2_cat}
An \emph{$\bK$-linear $(A_\infty,2)$-category} is an $\bK$-linear category over the relative 2-operad $\bigl((\ol M_r),(\ol{2M}_\bn)\bigr)$.
\null\hfill$\triangle$
\end{definition}

\begin{remark}
Observe that by passing to homology, i.e.\ applying $H$ to $2\Mor(L,K)$, we can extract an $\bK$-linear 2-category from an $\bK$-linear $(A_\infty,2)$-category.
(By ``$\bK$-linear 2-category'', we mean a category enriched in $\bK$-linear categories.)
In particular, note that the commutativity of horizontal and vertical composition of 2-morphisms follows from the fact that the two ways to compose correspond to the following two points in $W_{22}$, which can be connected by a path.
\begin{figure}[H]
\centering
\def\svgwidth{0.6\columnwidth}
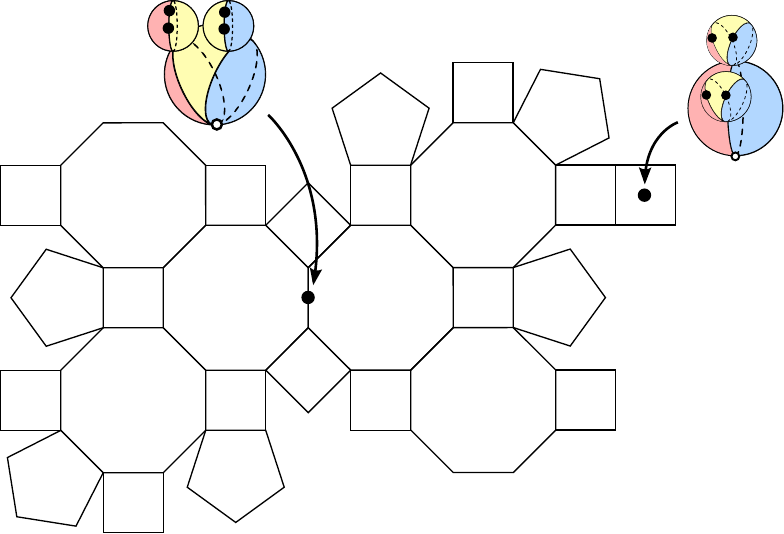
\caption{}
\end{figure}
\null\hfill$\triangle$
\end{remark}

\subsection{Analytical aspects of the construction of moduli spaces of witch balls}
\label{ss:analysis}

As we explained at the end of the introduction to \S\ref{s:symp}, the operations on 2-morphisms in $\Symp$ will be defined by counting rigid witch balls.
To do this, one needs to develop a regularization theory for moduli spaces of witch balls.
Analytic issues arise from the witch ball's ``singularity'', the point where all the domain's seams intersect tangentially, and from the related phenomenon that in a moduli space of witch balls, the width of one of the strips in the domain can shrink to zero.
In \cite{bottman_figure_eight_singularity} and \cite{bottman_wehrheim}, Bottman--Wehrheim made progress toward overcoming these challenges in the case of figure eight bubbles; the same analysis applies to general witch balls.
The analytic core of these results, which we will describe below, is a strengthening of the strip-shrinking estimates in \cite{wehrheim_woodward_geometric_composition}.

The first result is a ``removal of singularity'' for figure eight bubbles.
Such a bubble can be viewed as a tuple of finite-energy pseudoholomorphic maps
\begin{equation*} \label{eq:8}
w_0:\bR\times(-\infty,0]\to M_0, \qquad
w_1:\bR\times[0,1]\to M_1, \qquad
w_2:\bR\times[0, \infty)\to M_2
\end{equation*}
satisfying the seam conditions $(w_0(s,0),w_1(s,0))\in L_{01}$ and $(w_1(s,1),w_2(s,0))\in L_{12}$ for $s \in \bR$.
The second author established the following property of figure eights, as conjectured in \cite{wehrheim2010quilted}.

\begin{theorem}[Removal of singularity, \cite{bottman_figure_eight_singularity}]
\label{thm:rem_sing}
If the composition $L_{01} \circ L_{12}$ is cleanly immersed, then $w_0$ resp.\ $w_2$ extend to continuous maps on $D^2 \cong (\bR \times (-\infty, 0]) \cup \{\infty\}$ resp.\ $D^2 \cong (\bR \times [0, \infty)) \cup \{\infty\}$, and $w_1(s, -)$ converges to constant paths as $s\to\pm\infty$.
\null\hfill$\square$
\end{theorem}

The second result concerns strip-shrinking, a phenomenon new to quilted Floer theory: in a moduli space of quilted maps, the width of a strip or annulus in the domain of a pseudoholomorphic quilt may shrink to zero, as in the figure to the right.
To understand the topology of moduli spaces of maps from such domains, we need a ``Gromov Compactness Theorem'': given a sequence of quilts in which strip-shrinking occurs and in which the energy is bounded, a subsequence of the maps must converge $\cC^\infty_\loc$ away from finitely many points where the gradient blows up, and at each blowup point a tree of quilted spheres forms.
We depict this type of degeneration below:

\begin{figure}[H]
\centering
\includegraphics[width=0.425\columnwidth]{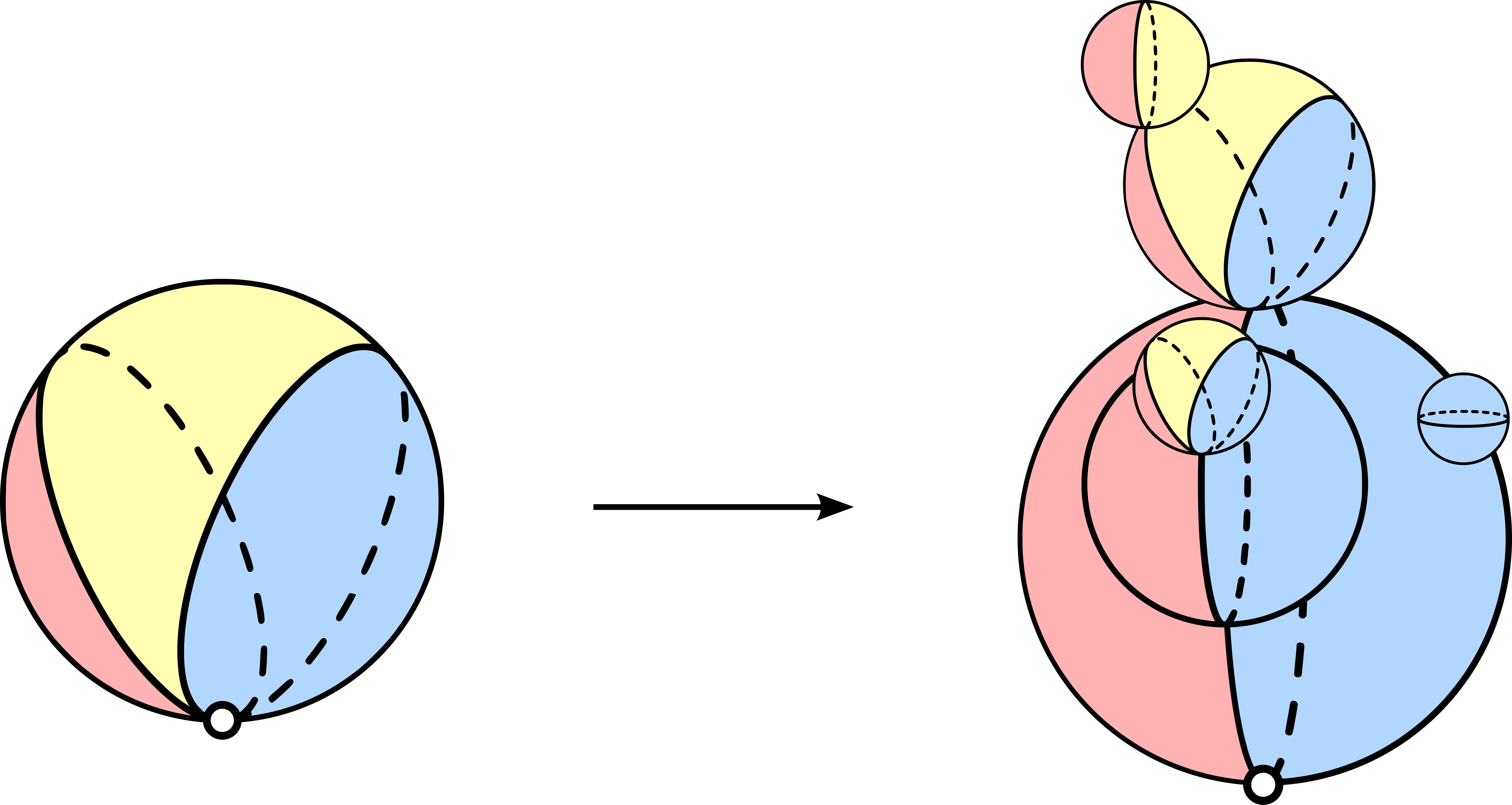}
\caption{}
\end{figure}

\noindent
Bottman--Wehrheim established full $\cC^\infty_\loc$-convergence in the following theorem.

\begin{theorem}[Gromov compactness, \cite{bottman_wehrheim}]
Say that $\ul{Q}^\nu$ is a sequence of pseudoholomorphic quilted maps, whose domains have a strip $Q_1^\nu$ of width $\delta^\nu \to 0$.
Denote the target of $Q_1^\nu$ by $M_1$, and the targets of the neighboring patches $M_0, M_2$; call the Lagrangians defining the adjacent seam conditions $L_{01}, L_{12}$.
Under the assumptions of Theorem \ref{thm:rem_sing}, there is a subsequence that converges up to bubbling to a punctured quilt, and the energy that concentrates at each puncture is captured in a bubble tree consisting of disks, spheres, and figure eight quilts.
\null\hfill$\square$
\end{theorem}

The major outstanding step toward the definition of $\Symp$ is the completion of a regularization theory for moduli spaces of witch balls.
Classical regularization techniques are not sufficient, except in specialized situations.

\begin{remark}
There are concrete obstructions to regularizing moduli spaces of pseudoholomorphic curves by classical means.
For instance, when the symplectic manifold has spheres with negative Chern number, the standard argument for regularizing moduli spaces of Floer cylinders by choosing a ``generic'' almost complex structure breaks down.
Indeed, the presence of multiply-covered spheres of negative Chern number lead to components of the moduli space of dimension higher than one would expect from standard index formulas
(see \cite[\S5.1]{salamon_lectures} for a detailed discussion of this phenomenon).
The same phenomenon occurs in Lagrangian Floer theory in the presence of  multiply-covered discs of negative index.
It seems that the only analogous phenomenon for witch balls occurs for components with only one seam, as these are the only ones that can be multiply covered. Of course, since one studies the space of \emph{stable maps} associated to witch balls, in which sphere and disc bubbles arise, the transversality obstructions associated to such multiply-covered configurations will affect the transversality theory for the definition of $\Symp$.
\null\hfill$\triangle$
\end{remark}

\subsection{Interpretations of existing structures in terms of $\Symp$}
\label{ss:interpretations}

Quite a number of existing constructions can be interpreted as substructures of the symplectic $(A_\infty,2)$-category.
In this subsection, we will explain some of them.

\subsubsection{The Fukaya $A_\infty$-category, the $A_\infty$-functors $\Phi_{L_{12}}$, and the geometric composition $A_\infty$-bifunctor}
\label{sss:specialization_to_bifunctor}

As we explained in the discussion around Figure \ref{fig:witch_specializations}, the pseudoholomorphic disks whose counts define the structure maps in $\Fuk M$ are just two-patch witch balls mapping to $\pt$ and $M$.
This is reflected and extended by the fact that in $\Symp$, the vertical composition maps in $\hom(M_0,M_1) = \Fuk(M_0^-\times M_1)$ are exactly the $A_\infty$-compositions in $\Fuk(M_0^-\times M_1)$.

Next, recall that in \S\ref{ss:approaching_Phi}, our initial attempted definition of a functor $\Phi_{L_{12}}\colon \Fuk M_1 \to \Fuk M_2$ proceeded by counting two-patch quilted disks (see the second illustration in Figure \ref{fig:witch_specializations}).
These quilted disks are exactly three-patch witch balls mapping to $\pt$, $M_1$, and $M_2$, with no marked points on the right seam.
This is reflected by the fact that in $\Symp$, there is a \emph{horizontal composition $A_\infty$-bifunctor}
\begin{align}
\bigl(
\Fuk(M_0^-\times M_1),
\Fuk(M_1^-\times M_2)
\bigr)
\to
\Fuk(M_0^-\times M_2),
\end{align}
whose structure constants are given by counts of three-patch witch balls.

\subsubsection{The closed-open string map}

\label{sss:CO}

The diagonal Lagrangian
\begin{equation}
  \Delta_M \subset M^- \times M
\end{equation}
is a canonically-defined Lagrangian which corresponds to the identity object of the endomorphism category $\Fuk(M^-\times M)$ of $M$ in $\Symp$.
When $M$ is compact, the Floer cohomology of the diagonal is known to be isomorphic to the ordinary cohomology \cite{fooo_antisymplectic_involution}, while in the noncompact case, it is expected to correspond to the symplectic cohomology $SH^*(M)$ originally introduced by Hofer and Floer (there is a discrepancy in the literature, as we use \emph{cohomology} to refer to the group that they referred to as \emph{homology}).
While the construction of this group is now understood in large generality (c.f.\ \cite{groman_floer_thy_open_mfds}), this expectation is asserted as a well-known result in \cite[\S8]{ganatra_thesis} for the special class of \emph{Liouville domain}.

On the other hand, there is the \emph{closed-open string map}
\begin{equation}
\CO
\colon
SH^*(M)
\to
HH^*(\Fuk M),
\end{equation}
which is a homomorphism from symplectic cohomology to the Hochschild cohomology of the Fukaya $A_\infty$-category of $M$.
This is an important tool for studying deformations of $\Fuk M$ which goes back to Fukaya, Oh, Ohta, and Ono and to Seidel.

The closed-open map is defined by counting disks with one incoming cylindrical end and any number of incoming boundary marked points, and one outgoing boundary marked point, as on the left of the following figure:
\begin{figure}[H]
\centering
\def\svgwidth{4in}
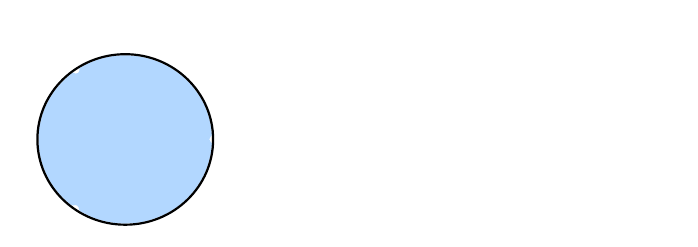
\caption{
\label{fig:CO}
}
\end{figure}
On the other hand, the figure on the right depicts an equivalent quilt, which moreover is one of the contributions to the specialization of the composition $A_\infty$-bifunctor in \S\ref{sss:CO} with $M_0 \coloneqq \pt, M_1 \coloneqq M \eqqcolon M_2$.
The conclusion we can draw from this is that the natural lift of $\CO$ to an $A_\infty$-homomorphism is part of the structure of $\Symp$.
That is, the following square commutes:

\begin{align}
\xymatrix{
SC^*(M) \ar[r] \ar@{<->}[d]_\simeq
&
CC^*(\Fuk M) \ar@{<->}[d]^\simeq
\\
CF^*(\Delta_M,\Delta_M) \ar[r]
&
\hom(\id_{\Fuk M}, \id_{\Fuk M})
}
\end{align}

\noindent
This point of view was (at least implicitly) suggested in \cite{ganatra_thesis} and \cite{rs:openclosed}.

\subsubsection{The open-closed string map, and symplectic cohomology as a module over a decategorification of $\Symp$}

Similarly to the closed-open string map described in \S\ref{sss:CO}, there is a homomorphism $\OC\colon \CC_*(\Fuk M) \to SC^*(M)$ from the Hochschild chain complex of the Fukaya category to the symplectic cochain complex, which is a crucial ingredient in the first author's generation criterion for $\Fuk M$ (which has so far been proven in the wrapped \cite{abouzaid_generation} and monotone \cite{rs:openclosed} settings).
This homomorphism is known as the \emph{open-closed string map}, and it is defined by counting disks with boundary punctures and one outgoing cylindrical end, as depicted below:

\begin{figure}[H]
\centering
\def\svgwidth{1.375in}
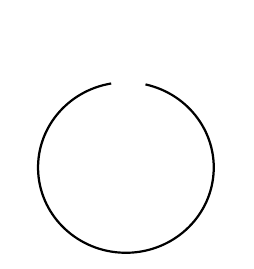
\caption{
\label{fig:OC}
}
\end{figure}

The open-closed string map also has an interpretation in terms of $\Symp$ --- but this interpretation is notably different than that of the closed-open string map.
Observe that the punctured disk in Figure \ref{fig:OC} is a special case of the following sort of quilted cylinder:

\begin{figure}[H]
\centering
\def\svgwidth{0.5\columnwidth}
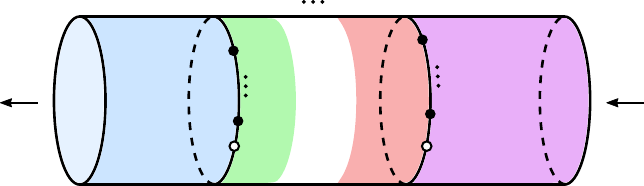
\caption{
\label{fig:SC_module}
}
\end{figure}

\noindent
Indeed, a cylinder of this form with one seam and patches mapping to $\pt$ and $M$, we get exactly the disk in Figure \ref{fig:OC}.

Counting quilted cylinders as in Figure \ref{fig:SC_module} has a clear algebraic interpretation.
We think of feeding in Hochschild chains in
\begin{align}
\CC_*(\Fuk(M_1^-\times M_2)),
\ldots,
\CC_*(\Fuk(M_{r-1}^-\times M_r))
\end{align}
and obtaining a map
\begin{align}
\SC^*(M_1)
\to
\SC^*(M_r).
\end{align}
Counting quilted cylinders as in Figure \ref{fig:SC_module} exhibits $\bigoplus_M \SC^*(M)$ as an $A_\infty$-module over $\Symp_{\CC_*}$, where the latter objects is the decategorification of $\Symp$ to an $A_\infty$-category whose morphism complexes are given by
\begin{align}
\hom(M_1,M_2)
\coloneqq
\CC_*(\Fuk(M_1^-\times M_2)).
\end{align}

%%% Local Variables:
%%% mode: latex
%%% TeX-master: "functoriality_in_categorical_symplectic_geometry"
%%% End:

\section{Applications}
\label{s:applications}

In this final section, we will survey a number of developments which either use the theory of pseudoholomorphic quilts directly, or take inspiration from it without actually implementing its constructions.

\subsection{Categorification of Dehn twists}
\label{sec:categ-dehn-twists}

The study of Dehn twists in symplectic topology started with Arnold's generalization in \cite{Arnold1995} of the fact that the monodromy of a family of complex curves which acquire a nodal singularity is given by the Dehn twist around the vanishing circle to higher dimensions.
Arnold constructed a symplectomorphism of the disc cotangent bundle of the $n$-sphere, which essentially wraps the fibres around the base, showed that it models the monodromy of a family of smooth projective varieties near ordinary double point singularities.

In \cite{seidel_les} Seidel then showed that Arnold's model symplectomorphisms are not isotopic to the identity, relative the boundary, by computing their action on Floer homology groups: more precisely, imposing some technical conditions on the ambient symplectic manifold, he proved that, if $V$ is a Lagrangian sphere, and $L$ a Lagrangian supporting an object of the Fukaya category, there is an exact triangle
\begin{equation}
\label{eq:seidel_triangle}
\begin{tikzcd}
HF^*(V,L) \otimes V \ar[rr] & & L \ar[dl] \\
& \tau_V L \ar[ul]
\end{tikzcd}
\end{equation}
in the Fukaya category.
This result was a milestone in the study of rigidity phenomena in symplectic topology, and has motivated key developments in the study of homological mirror symmetry, categorification of braid invariants, and, as we shall discuss below, symplectic constructions of invariants of knots and $3$-manifolds.

Wehrheim and Woodward extended the reach of Seidel's result in two ways:
\begin{enumerate}
\item
They generalised the result by replacing $V$ by a spherically fibered Lagrangian correspondence, i.e.\ a Lagrangian $C \subset N^- \times M$ whose projection to $M$ is an embedding and whose projection to $N$ is a sphere bundle, and

\smallskip

\item
They formulated an exact sequence depending only the $C$, and not on any auxiliary Lagrangian in $M$.
\end{enumerate}

\noindent
They geometrically formulate their result in terms of a variant of the extended Fukaya category of $M^- \times M$ (c.f.\ \S\ref{ss:def_of_Phi_L12}); objects of this category are cyclic generalized correspondences starting and ending at the symplectic manifold $M$.
This simplest such correspondence is the empty correspondence, which we denote $\Delta_M$ because it corresponds to the diagonal of $M$.

We need three additional geometric constructions to formulate Wehrheim and Woodward's lift of the Seidel exact triangle:
\begin{enumerate}
\item The graph of every symplectomorphism $\psi$ define a Lagrangian correspondence $\Delta_\psi$ in $ M^- \times M$.

\smallskip

\item Given a Lagrangian $C \subset N^- \times M$, the product  $C \times C$ is a Lagrangian in $M^- \times N \times  N^- \times M$, hence defines an object of $\Fuk(M \times M^-)$, which is denoted $C^t \# C $.

\smallskip

\item There is a canonical morphism $C^t \# C \to \Delta_M$, of degree $\dim N$, given by counting quilts of the following form.

\begin{figure}[H]
\centering
\def\svgwidth{0.25\columnwidth}
%% Creator: Inkscape 1.2 (dc2aeda, 2022-05-15), www.inkscape.org
%% Accompanies image file 'counit.pdf' (pdf, eps, ps)
\begingroup%
  \makeatletter%
  \providecommand\color[2][]{%
    \errmessage{(Inkscape) Color is used for the text in Inkscape, but the package 'color.sty' is not loaded}%
    \renewcommand\color[2][]{}%
  }%
  \providecommand\transparent[1]{%
    \errmessage{(Inkscape) Transparency is used (non-zero) for the text in Inkscape, but the package 'transparent.sty' is not loaded}%
    \renewcommand\transparent[1]{}%
  }%
  \providecommand\rotatebox[2]{#2}%
  \newcommand*\fsize{\dimexpr\f@size pt\relax}%
  \newcommand*\lineheight[1]{\fontsize{\fsize}{#1\fsize}\selectfont}%
  \ifx\svgwidth\undefined%
    \setlength{\unitlength}{102.81906668bp}%
    \ifx\svgscale\undefined%
      \relax%
    \else%
      \setlength{\unitlength}{\unitlength * \real{\svgscale}}%
    \fi%
  \else%
    \setlength{\unitlength}{\svgwidth}%
  \fi%
  \global\let\svgwidth\undefined%
  \global\let\svgscale\undefined%
  \makeatother%
  \begin{picture}(1,0.47831404)%
    \lineheight{1}%
    \setlength\tabcolsep{0pt}%
    \put(0,0){\includegraphics[width=\unitlength,page=1]{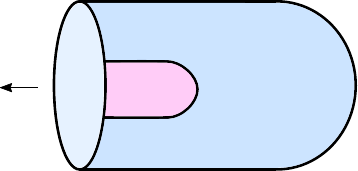}}%
    \put(0.36053477,0.33951826){\makebox(0,0)[lt]{\lineheight{1.25}\smash{\begin{tabular}[t]{l}$C$\end{tabular}}}}%
  \end{picture}%
\endgroup%

\caption{
}
\end{figure}
\end{enumerate}

\noindent
The morphism $C^t \#C \to \Delta_M$ is simply the co-unit of the adjunction between $C$ and $C^t$.

Denoting by $\tau_C$ the \emph{fibred Dehn twist} associated to $C$, Wehrheim and Woodward's result is:

\begin{theorem}[Theorem 7.4, \cite{wehrheim_woodward_exact_triangle}]
The counit $C^{t} \# C  \to \Delta_M $  fits in an exact triangle in $\Fuk(M \times M^-)$
\begin{equation}
\begin{tikzcd}
C^t \# C[\dim N] \ar[rr] & & \Delta_M \ar[dl]
\\
& \Delta_{\tau_C}. \ar[ul]
\end{tikzcd}
\end{equation}
\null\hfill$\square$
\end{theorem}

\noindent
We have formulated this result in a slightly stronger way than Wehrheim--Woodward, by specifying one of the morphisms in the exact triangle, which determines the isomorphism class of $\Delta_{\tau_C}$.

Since it is stated at the level of the product $M \times M^-$, rather than on $M$, the quilt formalism yield a categorification of Seidel's result, which lifts the exact triangle of Floer groups to an exact triangle of bimodules.
More precisely, by passing from Fukaya categories to their categories of bimodules as in \S\ref{ss:bimodules}, we have:

\begin{corollary}[Theorem 7.4, \cite{wehrheim_woodward_exact_triangle}] \label{cor:LES-bimodules}
  There is an exact triangle in the category of bimodules over $\Fuk M$
\begin{equation}
\begin{tikzcd}
CF^*(L_0, C) \otimes CF^*(L_1, C)  \ar[rr] & & CF^*(L_0, L_1) \ar[dl]
\\
& CF^*(L_0, \tau_C L_1). \ar[ul]
\end{tikzcd}
\end{equation}
\null\hfill$\square$
\end{corollary}

\noindent
In fact, Seidel's triangle \eqref{eq:seidel_triangle} is a consequence of this corollary.

\subsection{Khovanov homology}
In \cite{SeidelSmith2006} Seidel and Smith introduced an invariant of oriented links in $S^3$ as follows: they associate to each natural number $n$ a symplectic  manifold $\cS_n$, equipped with a symplectic action (up to Hamiltonian isotopy) of the braid group on $n$ letters, i.e.\ a homomorphism
\begin{equation} 
\Phi
\colon
\mathrm{Braid}_{n}
\to
\pi_0\mathrm{Symp}(\cS_n).
\end{equation}
They construct as well a Lagrangian $L^{\Cap}_{n}$ in a symplectic manifold $\cS_n$, thus obtain a homology group associated to each oriented knot $K$, which is presented as the closure of a braid $\beta$ with $n$ strands (see Figure \ref{fig:braid_closure}) from Lagrangian Floer cohomology:
\begin{equation}
\mathrm{Kh}_{symp}(\beta)
\coloneqq
HF^*(L^{\Cap}_{n}, \Phi_{\beta} L^{\Cap}_{n}).
\end{equation}
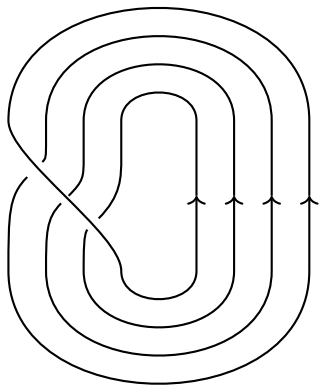
\begin{figure}[h]
\centering
\begin{tikzpicture}[scale=0.5]
  \draw[thick] (-3.5,0.5) .. controls (-4,0) and (-4,-.5) .. (-4,-2) .. controls (-4,-6) and (4,-6) .. (4,-2);
  \draw[thick,->] (4,-2) -- (4,0);
  \draw[thick] (4,0) -- (4,2) .. controls (4,6) and (-4,6) .. (-4,2) .. controls (-4,1) and (-1,-1) .. (-1,-2) .. controls (-1,-3) and (1,-3) .. (1,-2);
  \draw[thick,->] (1,-2) -- (1,0);
   \draw[thick] (1,0) -- (1,2) .. controls (1,3) and (-1,3) .. (-1,2) -- (-1,1) .. controls (-1,.5) and (-1,0) .. (-1.6,-.6);
  \draw[thick] (-1.9,-.9) .. controls (-2,-1) and (-2,-1.5) .. (-2,-2) .. controls (-2,-4) and (2,-4) .. (2,-2);
 \draw[thick,->] (2,-2) -- (2,0);
  \draw[thick] (2,0)  -- (2,2) .. controls (2,4) and (-2,4) .. (-2,2) -- (-2,1) .. controls (-2,0.5) and (-2,.4) .. (-2.4,0);
  \draw[thick] (-2.6,-.2) .. controls (-3,-.6) and (-3,-1) .. (-3,-2) .. controls (-3,-5) and (3,-5) .. (3,-2);
 \draw[thick,->] (3,-2) -- (3,0);
   \draw[thick] (3,0)  -- (3,2) .. controls (3,5) and (-3,5) .. (-3,2) -- (-3,1.5) .. controls (-3,1) and (-3,1) ..(-3.1, .9);
\end{tikzpicture}
\caption{\label{fig:braid_closure}
The closure of the product of positive braid generators $\sigma_3\sigma_2\sigma_1$.}
\end{figure}

The essential problem is now to prove that the above group does not depend on the choice of presentation of $K$ as a braid closure, i.e.\ that it defines a knot invariant, which is called \emph{symplectic Khovanov homology} (see Theorem \ref{thm:Skh=Kh} below).
This can be analyzed using classical Markov moves \cite{birman_book} which can be used to relate any two braid closure presentations of a knot.
The first move is a conjugation, which can be implemented by a Hamiltonian isotopy in the ambient symplectic manifold $\cS_n$.

The second move is a stabilization, and entails comparing Floer homology groups in $\cS_n$ and $\cS_{n+2}$.
In Seidel--Smith's construction of their invariant, this move can be described in terms of a pair of Lagrangian correspondences $\cup_1$ and $\cup_2$ in $\cS^-_n \times \cS_{n+1}$, which are $S^2$ fibered over $\cS_n$, that meet cleanly along a copy of $\cS_n$.
Writing $\tau_2$ for the fibered Dehn twist about $\cup_2$, Seidel--Smith prove:
\begin{proposition}[Lemma 44 of \cite{SeidelSmith2006}]
Given a pair of Lagrangians $L$ and $L'$ in $\cS_n$, there is a natural isomorphism
\begin{equation}
HF^*(L,L') \cong HF^*(\cup_1 L, \tau_2 \cup_1 L').
\end{equation}
\qed
\end{proposition}
This establishes the key step in the proof that $\mathrm{Kh}_{symp}$ defines a knot invariant because, under the homomorphism $\mathrm{Braid}_{n} \to
\pi_0\mathrm{Symp}(\cS_n)$ studied by Seidel and Smith, the image of each elementary braid is a fibered Dehn twist about an $S^2$-fibered Lagrangian $\cup_i$ in $\cS^-_n \times \cS_{n+1}$.
This was exploited by Rezazadegan \cite{reza_khovanov} to extend Seidel and Smith's theory to a tangle invariant.

Because the Lagrangian correspondences $\cup_i$ are noncompact, it has been difficult to directly use the theory of pseudoholomorphic quilts in this specific context.
Nonetheless, a version of Corollary  has \ref{cor:LES-bimodules} proved essential in implementing the comparison between Seidel--Smith's theory and Khovanov homology which was initiated in \cite{abouzaid_symp_arc_algebra_formal}:

\begin{theorem}[\cite{abouzaid2019khovanov}] \label{thm:Skh=Kh}
Over a field of characteristic $0$, the knot invariant $ \mathrm{Kh}_{symp}$ agrees with Khovanov categorification of the Jones polynomial (after collapsing the bigrading to a single grading).
\null\hfill$\square$
\end{theorem}

\subsection{Floer field theory}
\label{ss:floer_field_theory}

Building on work of Donaldson, Floer, Fukaya, Donaldson, Salamon--Wehrheim, and others, Wehrheim--Woodward use quilted Floer theory to construct in \cite{wehrheim_woodward_floer_field_theory} a $(2+1)$-dimensional connected category-valued field theory.
More specifically, for coprime positive $r, d$, Wehrheim--Woodward construct a functor
\begin{align}
&
\Phi
\colon
\text{(compact connected oriented 2-manifolds, 3-bordisms)}
\\
&
\hspace{2in}
\lra
\text{ ($A_\infty$-categories, homotopy classes of $A_\infty$-functors)}.
\nonumber
\end{align}
$\Phi$ is defined on objects and morphisms in the following way:
\begin{itemize}
\item
Fix a (compact, connected, oriented) surface $X$.
Up to isomorphism, there is a unique degree-$d$ $U(r)$-bundle $P$ on $X$.
$\Phi$ sends a (compact connected oriented) surface $X$ to $\Fuk^\#(M(X))$, where $M(X)$ is the representation variety of central-curvature connections with fixed determinant.
The Hodge pairing equips $M(X)$ with a monotone symplectic form.

\smallskip

\item
Suppose that $Y$ is a 3-cobordism from $X_-$ to $X_+$ that is \emph{elementary}, i.e.\ it is either a cylinder or it represents a handle attachment.
Define $L(Y)$ to be the moduli space of central-curvature fixed-determinant connections.
Then restriction defines a Lagrangian correspondence
\begin{align}
L(Y)
\hra
M(X_-)^- \times M(X_+),
\end{align}
and the machinery of \S\ref{ss:def_of_Phi_L12} then produces an $A_\infty$-functor
\begin{align}
\Phi(Y)
\coloneqq
\Phi^\#_{L(Y)}
\colon
M(X_-)
\to
M(X_+).
\end{align}

When $Y$ is not elementary, we can choose a decomposition into elementary cobordisms $Y_1, \ldots, Y_m$, and define $\Phi(Y)$ to be the composition $\Phi(Y_m) \circ \cdots \circ \Phi(Y_1)$.
We of course need to know that this definition is independent (up to $A_\infty$-homotopy) of the choice of decomposition, and this follows from the strip-shrinking analysis described in \S\ref{ss:qHF} and the fact that any two decompositions differ by a sequence of \emph{Cerf moves}.
\end{itemize}
(This is part of a long history of work on the Atiyah--Floer conjecture, as explained in \cite{wehrheim_floer_field_philosophy}.)

Following an idea proposed by Kronheimer--Mrowka in \cite{kronheimer_mrowka:knot_homology_from_instantons}, Wehrheim--Woodward defined an invariant $\ol{HF}_{r,d}(Y)$ of closed oriented 3-manifolds $Y$ via the following procedure:
\begin{itemize}
\item
Define $\ol Y$ to be the result of connect-summing $Y$ with the toric cylinder $[-1,1] \times T^2$.

\smallskip

\item
The representation variety $M(T^2)$ is a point, so the functor $\Phi$ associates to $\ol Y$ an endofunctor of $\Fuk^\#(\pt)$.
We now define the 3-manifold invariant like so:
\begin{align}
\ol{HF}(Y)
\coloneqq
H\bigl(\hom\bigl(\Phi\bigl(\ol Y\bigr)\bigr)(\pt), \pt\bigr).
\end{align}
Equivalently, $\ol{HF}(Y)$ is the quilted Floer cohomology of the sequence $\pt \sr{L(Y_1)}{\lra} \cdots \sr{L(Y_m)}{\lra} \pt$, where $Y_1,\ldots,Y_m$ are the elementary cobordisms in a chosen decomposition of $\ol Y$.
\end{itemize}
An alternate construction applies to 3-manifolds $Y$ equipped with a nonseparating embedded compact connected oriented surface $X \subset Y$, c.f.\ \cite[Definition 4.3.1]{wehrheim_woodward_floer_field_theory}.

See \cite{wehrheim_floer_field_philosophy,lekili_heegaard_floer,auroux_gokova} for work toward an interpretation of Heegaard Floer homology in terms of quilted Floer cohomology.
In a different direction, Ivan Smith incorporated insights from mirror symmetry in order to study the instanton Floer homology of a 3-manifold fibered by genus-2 curves in \cite{smith:pencils}.

\subsection{Homological mirror symmetry}
\label{sec:homol-mirr-symm}

We return to mirror symmetry as a motivating source of structures in Floer theory, and note a particular mystery whose solution is provided by the functorial structures discussed so far: on the algebro-geometric side, the fact that coherent sheaves are locally given by the datum of modules over commutative rings allows one to assign to a pair of coherent sheaves a tensor product, which is a coherent sheaf on the same space, that is locally given by the tensor product of the corresponding modules.
This can be formulated as the existence of a natural functor
\begin{equation} \label{eq:monoidal_structure_DbCoh}
D^b(X) \otimes D^b(X)
\to
D^b(X)
\end{equation}
together with structures witnessing the associativity and commutativity of the tensor product construction.

This functor has no straightforward analogue in symplectic geometry, because there is no universal way of assign to a pair of Lagrangian submanifolds of a given symplectic manifold a putative tensor product that is a submanifold of the same space.
From the perspective of the theory discussed in \S\ref{s:symp}, such a functor would arise most naturally from a Lagrangian correspondence
\begin{equation} 
\Gamma^{\otimes}
\subset
M \times M \times M^-
\end{equation}
that is invariant under the involution which permutes the first two factors, and that satisfies the property that the two possible ways of composing $M$ with itself (by pairing the $M^-$ summand of the first factor with either of the two $M$ summands of the second factor) agree:
\begin{equation}
\Gamma^{\otimes} \circ_1 \Gamma^{\otimes}
\subset
M \times M \times M \times M^{-}
\supset
\Gamma^{\otimes} \circ_2 \Gamma^{\otimes}.
\end{equation}

In his thesis \cite{subotic}, Subotic considered this problem for $M$ a surface of genus $1$, and found that the construction of $\Gamma^{\otimes}$ depends on two choices: a projection map $\pi$ from $M$ to the circle, and a section thereof.
Writing $B$ for the base circle, the symplectic form canonically identifies the universal cover of the fibre $X_b$ over each point $b \in B$, based at its intersection with the chosen section, with the cotangent fibre $T^*_b B$, based at the origin: the pairing of a vector in $T_b B$ with a point in $\wt M_b$ is given by measuring the (symplectic) area of the (infinitesimal) region obtained by transporting the associated path in $M_b$, in the direction of the chosen vector, along some local trivialisation of the projection map $M \to B$ in a neighbourhood of $b$. This procedure identifies the deck transformations on $\wt M_b$ with translation by the multiples of a nonzero covector in $T^*_b B$, and hence equips the fibre $M_b$ with the structure of an abelian group.  

Subotic's key idea at this stage is to define a correspondence using fibrewise addition:
\begin{equation}
\Gamma^{\otimes}
\coloneqq
\bigl\{
(p,q,r) \in M\times M \times M^-
\:\big|\:
\pi(p) = \pi(q) = \pi(r), p + q = r
\bigr\},
\end{equation}
where the last equation makes sense because the first imposes the condition that the three points $(p,q,r)$ lie over the same point in $B$.
An application of the Arnold--Liouville theorem shows that this is a Lagrangian correspondence, and it is straightforward to prove that it satisfies the desired commutativity and associativity properties.
The main result of \cite{subotic} is:

\begin{theorem} \label{thm:Subotic}
The Lagrangian correspondence $\Gamma^{\otimes}$ induces a symmetric monoidal structure of the Donaldson-Fukaya category of a $2$-torus, which corresponds, under mirror symmetry, to the tensor product of bounded complexes of coherent sheaves, on the mirror elliptic curve.
\null\hfill$\square$
\end{theorem}

In subsequent work, Pascaleff \cite{pascaleff:poisson} explored the formal structure giving rise to correspondences of the form $\Gamma^{\otimes}$ from local actions of a Lie group, and identified Weinstein's notion of a \emph{symplectic groupoid} as a general notion encompassing generalisations of Subotic's construction to higher dimensional tori, cotangent bundles, and twisted versions thereof.
Under suitable technical assumptions, one can then prove the analogue of Theorem \ref{thm:Subotic}.

In order to lift Theorem \ref{thm:Subotic} from cohomology to the chain level, one needs an appropriate formulation of what it means to have a \emph{symmetric monoidal $A_\infty$-category}.
We expect that such a lift contains substantially more information than the cohomological version: concretely, consider a Lagrangian $L$ which is a unit for the monoidal structure, in the sense that
\begin{equation}
\Gamma^{\otimes} \circ_{1,2} (L \times L)
=
L.
\end{equation}
The outcome of Theorem \ref{thm:Subotic} is that the multiplicative structure of the Floer cohomology ring of such a Lagrangian is commutative.
In characteristic $0$, commutativity is essentially a property of $A_\infty$-algebras, and there is little additional information to be extracted.
On the other hand, the theory of (coherently) homotopy commutative algebras in finite characteristics is rich enough that one can extract from it the Steenrod operations on the cohomology of spaces.
One is therefore led to formulate the following:

\begin{conjecture}
The Steenrod operations on the cohomology of a closed manifold can be extracted from the Floer theory of its cotangent bundle, and the geometry of fibrewise addition.
\null\hfill$\square$
\end{conjecture}

\subsection{$\Fuk$ of $G$-manifolds}
We have already discussed, in Definition \ref{def:symp_red}, the fact that a Hamiltonian $G$-action on a symplectic manifold $M$ induces a correspondence between $M$ and its Hamiltonian reduction.
In \cite{evans2019generating}, Evans and Lekili considered instead the correspondence associated to the action $G \times M \to M$:
\begin{equation}
\Gamma_G
\coloneqq
\bigl\{
(g,a,x,y)
\:\big|\:
a = \mu(g\cdot x),
y = g\cdot x
\bigr\}
\subset
\left(T^* G \right)^- \times M^- \times M.
\end{equation}
By applying the Wehrheim--Woodward formalism, this correspondence gives rise to a functor
\begin{equation}
\Fuk T^*G
\to
\Fuk(M^- \times M).
\end{equation}
Evans and Lekili observed that this functor maps
\begin{enumerate}
\item
the cotangent fibre at the identity element of $G$ to the diagonal $\Delta_M$, and

\smallskip

\item
the image of the zero section of $T^*G$ agrees with the composition $ \Lambda_G \circ \Lambda_G^{t}$, of the moment correspondence and its adjoint.
\end{enumerate}
The domain of this functor is well-understood to be equivalent to a subcategory of the category of modules over the chains $C_{-*} \Omega G$ \cite{abouzaid:based_loops} (the paper \cite{evans2019generating} systematically misstates this to be the category of bimodules, but the arguments are unaffected).
This implies in particular that the $0$-section $G$ is equivalent, in the Fukaya category of $T^*G$, to a complex built from copies of the cotangent fibre.
Applying the functor therefore implies that $ \Lambda_G \circ \Lambda_G^{t}$ is built from the diagonal $\Delta_M$.
In general, the diagonal object may formally decompose into summands $\bigoplus_{\alpha} \Delta_{M,\alpha} $ corresponding to the maximal decomposition of the quantum-cohomology ring $QH^*(M)$ into a direct sum of rings, and an elementary algebra argument shows that $ \Lambda_G \circ \Lambda_G^{t}$ can in fact be built from those summands $\Delta_{M,\alpha} $ which are not Floer-theoretically orthogonal to it.

The main result of \cite{evans2019generating} is that, under mild assumptions on $G$, this procedure can be inverted:
\begin{theorem}[Theorem 1.2.2 of \cite{evans2019generating}]
If $\bK$ is a field of characteristic relatively prime to the torsion subgroup of the cohomology group $H^*(G)$, then every summand $\Delta_{M,\alpha}$ which is not Floer-theoretically orthogonal to $ \Lambda_G \circ \Lambda_G^{t} $ lies in the category which it split-generates.
\end{theorem}

\noindent
While we do not reproduce their argument here, we note that it is a completely algebraic consequence, given the above discussion, of the fact that the quantum cohomology of a closed manifold is finite-dimensional and commutative.

Assuming that the Hamiltonian reduction of $M$ by $G$ is smooth, this result, combined with those discussed in Section \ref{sec:conj-sympl-barr} below, yield a description of the summands of $\Fuk M$ on which $ \Lambda_G \circ \Lambda_G^{t} $ acts non-trivially in terms of the Fukaya category of the reduction.
Evans and Lekili restricted themselves to the case in which this reduction is a point, and concluded:

\begin{corollary}
If $\mu^{-1}(0)$ is a free orbit of $G$, then it split-generates an orthogonal factor of $\Fuk M$.
\end{corollary}

The most straightforward class of examples satisfying this condition are toric varieties, where the group $G$ is a torus $\mathbb{T}^n$. We invite the reader to peruse \cite{evans2019generating} for how one can derive explicit computations about Fukaya categories of toric varieties from the above result, with minimal effort.

\subsection{A formal group structure on $MC(M)$}

In \cite{seidel_formal}, Paul Seidel used ideas from the theory of pseudoholomorphic quilts to define a new invariant of compact monotone symplectic manifolds.
This invariant is denoted $MC(M)$.
To introduce it, we begin by reviewing the notion of Maurer--Cartan elements, in the setting of \cite{seidel_formal}.

Suppose that $\cA$ is an $A_\infty$-ring, and that $N$ is an adic ring.
(``Adic'' means that $N$ is a nonunital commutative ring and that the map $N \to \varprojlim_m N/N^m$ is an isomorphism.
Some standard examples are $q\bZ[[q]]$, $q\bZ[q]/q^{m+1}$, $q\bF_p[[q]]$, and $p\bZ_p$.)
Now define $\cA\hat\otimes N$ to be the inverse limit $\varprojlim_m \cA \otimes (N/N^m)$, which is the right way to formulate ``$\cA$ with coefficients in $N$''.

\begin{definition}
A \emph{Maurer--Cartan element in $\cA \hat\otimes N$} is an element $\gamma \in \cA^1\hat\otimes N$ satisfying the Maurer--Cartan equation
\begin{align}
\label{eq:MC}
\sum_{d\geq1}
\mu_d(\gamma,\ldots,\gamma)
=
0.	
\end{align}
We say that	two solutions $\gamma, \wt\gamma$ of \eqref{eq:MC} are (gauge-)equivalent if there exists $h \in \cA^0 \hat\otimes N$ satisfying
\begin{align}
\label{eq:gauge_equivalence}
\sum_{p, q \geq 0}
\mu_{p+q+1}\Bigl(\underbrace{\gamma,\ldots,\gamma}_p,h,\underbrace{\wt\gamma,\ldots,\wt\gamma}_q\Bigr)
=
\gamma - \wt\gamma.
\end{align}
(\eqref{eq:gauge_equivalence} can be interpreted as saying that $\gamma - \wt\gamma$ is exact with respect to the differential on $\cA$ deformed by $\gamma$ and $\wt\gamma$.)
We now define $MC(\cA;N)$ to be the set of equivalence classes of Maurer--Cartan elements in $\cA \hat\otimes N$.
\null\hfill$\triangle$
\end{definition}

This construction is functorial in $N$, so we can think of $MC(\cA;-)$ as a functor from adic rings to sets.
Seidel's main contribution in \cite{seidel_formal} is to show that when $\cC \coloneqq QC^*(M)$ is a suitable model for the integral quantum cochain complex of $M$, thought of as an $A_\infty$-ring, then $MC(\cC;-)$ can be upgraded to a functor from adic rings to \emph{groups}.
Such a functor is, in Seidel's parlance, a \emph{formal group}, so we obtain an invariant of compact monotone symplectic manifolds denoted $MC(X)$ and valued in formal groups.

The group structure on $MC(M)$ is defined by counting certain pseudoholomorphic spheres in $M$.
While these spheres are conventional, unquilted pseudoholomorphic curves, the motivation for the construction comes from quilted Floer theory --- particularly, from the composition bifunctor
\begin{align}
C
\colon
\bigl(
\Fuk(M_1^-\times M_2),
\Fuk(M_2^-\times M_3)
\bigr)
\to
\Fuk(M_1^-\times M_3),
\end{align}
where the objects in this version of the Fukaya category are Lagrangians equipped with bounding cochains, as in \S\ref{sss:strip-shrinking} and \S\ref{ss:fukaya}.
A version of this functor has been defined by Fukaya, as explained in \S\ref{ss:fukaya}, and in the monotone setting, a version of this functor was constructed by Ma'u--Wehrheim--Woodward in \cite{mww}.
Ma'u--Wehrheim--Woodward's approach is expected to extend to the general compact setting, in which it will act on objects by sending $\bigl((L_{12},b_{12}),(L_{23},b_{23})\bigr)$ to $(L_{12}\circ L_{23}, b_{13})$, where $b_{13}$ is defined to be the result of counting two-seam witch balls with patches mapping to $M_1, M_2, M_3$, and with arbitrarily many insertions of $b_{12}$ and $b_{23}$ on its seams.
When we specialize to $M_1=M_2=M_3$ and $L_{12}=L_{23}=\Delta_M$, we obtain an operation that sends $\bigl((\Delta_M,b),(\Delta_M,b')\bigr)$ to $(\Delta,b'')$, where $b''$ is the result of counting pseudoholomorphic spheres in $M$ with arbitrarily many insertions of $b$ resp.\ $b'$ arranged on two circles.
We can think of this as an operation on $QC^*(M)$.
One of Seidel's insights is that this operation descends to $MC(M)$, and that it is associative as a product on $MC(M)$.
Before descending to $MC(M)$, this operation is not associative.

This invariant is quite new, and its properties and applications have not been explored fully.
One intriguing relationship with quantum Steenrod squares arises when we work with coefficients in $q\bF_p[q]/q^{p+1}$: then $p$-th power map on Maurer--Cartan elements intertwines with the $t^{\frac{p-1}2}$-component of the quantum Steenrod square.

\subsection{A Barr--Beck theorem in Floer theory}
\label{sec:conj-sympl-barr}

\subsubsection{Adjunction, monads, and the Barr--Beck theorem}

Our purpose in this section is to recall a formalism, going back to category theorists in the 1960's, for answering the following question, which we shall discuss more specifically the case of Fukaya categories in \S\ref{sss:barr-beck_and_fuk} below.
\begin{question}
Given a functor $X \to Y$, when is $Y$ ``computable'' from $X$, together with some additional data?
\end{question}

The simplest example to have in mind is the case where $X$ decomposes as the (orthogonal) union of $Y$ and another category.
In this case, there is a functor in the other direction, so that the composite endofunctor of $X$ is idempotent, and $Y$ can readily be recovered from this data.

A more sophisticated example is the case in which the domain is the category of modules over a ground ring $\bk$, and the target is the category of modules over a $\bk$-algebra $A$, and the functor assigns to each $\bk$-module its tensor product with $A$, considered as an $A$-module.
We again have a functor in the other direction, which assigns to each $A$-module the underlying $\bk$-module, and consider the composite endofunctor $T$ which assigns to each $\bk$-module its tensor product with $A$, now considered only as a $\bk$-module.

In this case, the functor $T$ is not idempotent, but multiplication in $A$ defines a map
\begin{equation}
\mu :  A \otimes A \otimes M \to A \otimes M
\end{equation}
which gives a natural transformation from $T^2$ to $T$.
Moreover, the identity element of $A$ gives
\begin{equation}
  \eta : M \to A \otimes M
\end{equation}
which defines a natural transformation from the identity to $T$.
The axioms for an algebra that are taught to undergraduates are then reflected in properties of the functors $\mu$ and $\eta$, and the content of the theory which we shall describe is that the category of $A$-modules can be recovered from the endo-functor $T$ and these operations.

Note that, in the above discussion, we used the existence of a functor in the other direction.
The general framework in which this functor can be placed is that of an adjunction $L \dashv R$:
\begin{align}
\label{eq:adjoint_pair}
\xymatrix{
X \ar@/^/[r]^L & Y. \ar@/^/[l]^R
}
\end{align}
We shall presently formulate what it means for such an adjuction to be \emph{monadic}, before introducing the Barr--Beck theorem, which is a convenient characterization of this property.

Denote the unit and counit of the adjunction by $\eta\colon \id_X \to RL$ and $\eps\colon LR \to \id_Y$.
This data gives rise to a monoid object $(T, \mu, \eta)$ in the category of endofunctors of $X$ (a \emph{monad} on $X$), i.e.\ the following data (subject to two straightforward coherence conditions):
\begin{itemize}
\item
An endofunctor $T\colon X \to X$.

\smallskip

\item
Two natural transformations $\eta\colon \id_X \to T$ (``identity'') and $\mu\colon T^2 \to T$ (``composition'').
\end{itemize}
Indeed, we define the monad associated to $L \dashv R$ by
\begin{align}
(T, \mu, \eta)
\coloneqq
(RL, R\eps L, \eta).
\end{align}

It turns out that every monad comes from an adjunction (as Eilenberg--Moore and Kleisli proved in 1965, via completely different constructions).
One way to see this is by considering the \emph{Eilenberg--Moore category}
$\mod\, T$
associated to a monad $(T,\mu,\eta)$ on $X$.
As the notation suggests, $\mod\, T$ is the category of \emph{$T$-modules}, i.e.\ pairs $(x, h)$ of $x \in X$ and $h\colon Tx \to x$ (the ``structure map'') satisfying diagrams representing the associative and unit laws.
Then there is an adjunction
\begin{align}
\xymatrix{
X \ar@/^1pc/[r]^{L^T} & \mod\, T, \ar@/^1pc/[l]^{R^T}
}
\end{align}
which the reader may either work out as an exercise or find in \cite[Theorem 1]{maclane_working}.
This adjunction $L^T \dashv R^T$ has the property that its associated monad is exactly $(T,\mu,\nu)$.

\begin{remark}
The Eilenberg--Moore category is typically denoted $X^T$, and its elements referred to as \emph{$T$-algebras}.
We use the terminology of $T$-modules in part because it enables us to refer to \emph{perfect $T$-modules} later in this subsection.
\null\hfill$\triangle$
\end{remark}

Given an adjunction as in \eqref{eq:adjoint_pair}, we can define the associated \emph{comparison functor} $K\colon Y \to \mod\, T$, which acts on objects by sending $y \in Y$ to $(Ry, R\eps_y)$.
We say that $L \dashv R$ is a \emph{monadic adjunction} if $K\colon Y \to \mod\, T$ is an equivalence.
Monadic adjunctions are thus those with the property that if we form the associated monad, and then form the adjunction associated to that monad via the Eilenberg--Moore category, we return to the adjunction we started with.

The Barr--Beck theorem\footnote{A more historically-accurate name may be ``Beck's precise tripleability theorem'', as in \cite[Theorem 10]{maclane_working}.
However, modern authors typically use the name ``Barr--Beck theorem'' (see e.g.\ \cite[Theorem 4.7.3.5]{lurie2012higher}).} provides a criterion for monadicity that is often easier to check than by working from the definition.
We state a weak version of this theorem, because it is easier to state and sufficient for our purposes.
(This is the ``Weak Tripleability Theorem'' from Beck's thesis.)

\begin{theorem}[paraphrase of Exercise VI.7.3, \cite{maclane_working}]
\label{thm:barr--beck}
An adjunction $L \dashv R$ is monadic if the following conditions hold:
\begin{itemize}
\item[(BB1)]
$R$ is conservative, i.e.\ it reflects isomorphisms.
(That is, if $Rf$ is an isomorphism, then $f$ must be an isomorphism.)

\smallskip

\item[(BB2)]
$Y$ has and $L$ preserves all coequalizers.
\null\hfill$\square$
\end{itemize}
\end{theorem}

\noindent
We now make some comments about these two conditions.
\begin{enumerate}
\item
An example of a conservative functor is the forgetful functor from $Grp$ to $Set$, because a group homomorphism is an isomorphism if and only if it is a bijection.
A nonexample is the forgetful functor from $Top$ to $Set$, because there are continuous bijections that are not homeomorphisms.
Note also that conservativity does not imply the property that if $RX \simeq RY$, then $X \simeq Y$: for instance, if we consider the forgetful functor from $Grp$ to $Set$, this latter property is not satisfied, because two finite groups of the same cardinality are not necessarily isomorphic.

\smallskip

\item
A \emph{coequalizer} is a colimit of a parallel pair of morphisms.
That is, a coequalizer of the pair $f, g\colon A \rightrightarrows B$ is an morphism $h\colon B \to C$ such that $h\circ f = h\circ g$, and such that for any other such morphism $h'\colon B \to C'$, there exists a morphism $C \to C'$ making the following diagram commute:
\begin{align}
\xymatrix{
A \ar@<0.5ex>[r]^f \ar@<-0.5ex>[r]_g & B \ar[r]^h \ar[dr]_{h'} & C \ar@{-->}[d]
\\
&&
C'.
}
\end{align}
In the category $Grp$, the coequalizer of $f, g\colon A \to B$ is the quotient of $B$ by the normal closure of the set of elements of the form $f(x)g(x)^{-1}$.
One should think about a coequalizer as the generalization to an arbitrary category of the notion of the quotient by an equivalence relation.
\end{enumerate}

\subsubsection{Barr--Beck and the Fukaya category}
\label{sss:barr-beck_and_fuk}

As we explained in \S\ref{ss:adjunction}, the functors $\Phi^\#_{L_{12}}$ and $\Phi^\#_{L_{12}^T}$ are expected to form an adjunction.
The same is true of the extended versions of these functors:

\begin{align}
\label{eq:WW_adjunction}
\xymatrix{
\Fuk^\#(M_1) \ar@/^1pc/[r]^{\Phi_{L_{12}}^\#} & \Fuk^\#(M_2), \ar@/^1pc/[l]^{\Phi_{L_{12}^T}^\#}
}
\end{align}

\noindent
It is natural to ask when this adjunction is monadic, and in particular, when the hypotheses of Theorem \ref{thm:barr--beck} hold.

\begin{remark}
We stated Theorem \ref{thm:barr--beck} for an adjoint pair of functors, rather than an adjoint pair of $A_\infty$-functors (c.f.\ Equation \eqref{eq:putative_adjunction}).
Since our main goal in this section is to describe some geometric ideas regarding the relevance of the Barr--Beck theorem to symplectic topology, we leave the precise formulation to future work by the present authors.
\null\hfill$\triangle$
\end{remark}

The fundamental difficulty in applying Theorem \ref{thm:barr--beck} to \eqref{eq:WW_adjunction} is that checking condition (BB2) a priori involves a computation for each Lagrangian in $M_2$.
This is not a reasonable computation to directly perform, because we do not have any concrete description of the collection of all Lagrangians in any symplectic manifold of dimension greater than $2$.

Quilted Floer theory provides us with a solution: consider the ``quilted closed-closed map''
\begin{align}
\label{eq:QCC}
SH^*(M_1)
\otimes
HH_*(CF^*(L_{12},L_{12}),CF^*(L_{12},L_{12}))
\to
SH^*(M_2)
\end{align}
which is defined by counting quilts of the following form:
\begin{figure}[H]
\centering
\def\svgwidth{0.15\columnwidth}
%% Creator: Inkscape 1.2 (dc2aeda, 2022-05-15), www.inkscape.org
%% Accompanies image file 'QCC.pdf' (pdf, eps, ps)
\begingroup%
  \makeatletter%
  \providecommand\color[2][]{%
    \errmessage{(Inkscape) Color is used for the text in Inkscape, but the package 'color.sty' is not loaded}%
    \renewcommand\color[2][]{}%
  }%
  \providecommand\transparent[1]{%
    \errmessage{(Inkscape) Transparency is used (non-zero) for the text in Inkscape, but the package 'transparent.sty' is not loaded}%
    \renewcommand\transparent[1]{}%
  }%
  \providecommand\rotatebox[2]{#2}%
  \newcommand*\fsize{\dimexpr\f@size pt\relax}%
  \newcommand*\lineheight[1]{\fontsize{\fsize}{#1\fsize}\selectfont}%
  \ifx\svgwidth\undefined%
    \setlength{\unitlength}{1019.71545bp}%
    \ifx\svgscale\undefined%
      \relax%
    \else%
      \setlength{\unitlength}{\unitlength * \real{\svgscale}}%
    \fi%
  \else%
    \setlength{\unitlength}{\svgwidth}%
  \fi%
  \global\let\svgwidth\undefined%
  \global\let\svgscale\undefined%
  \makeatother%
  \begin{picture}(1,0.95179379)%
    \lineheight{1}%
    \setlength\tabcolsep{0pt}%
    \put(0,0){\includegraphics[width=\unitlength,page=1]{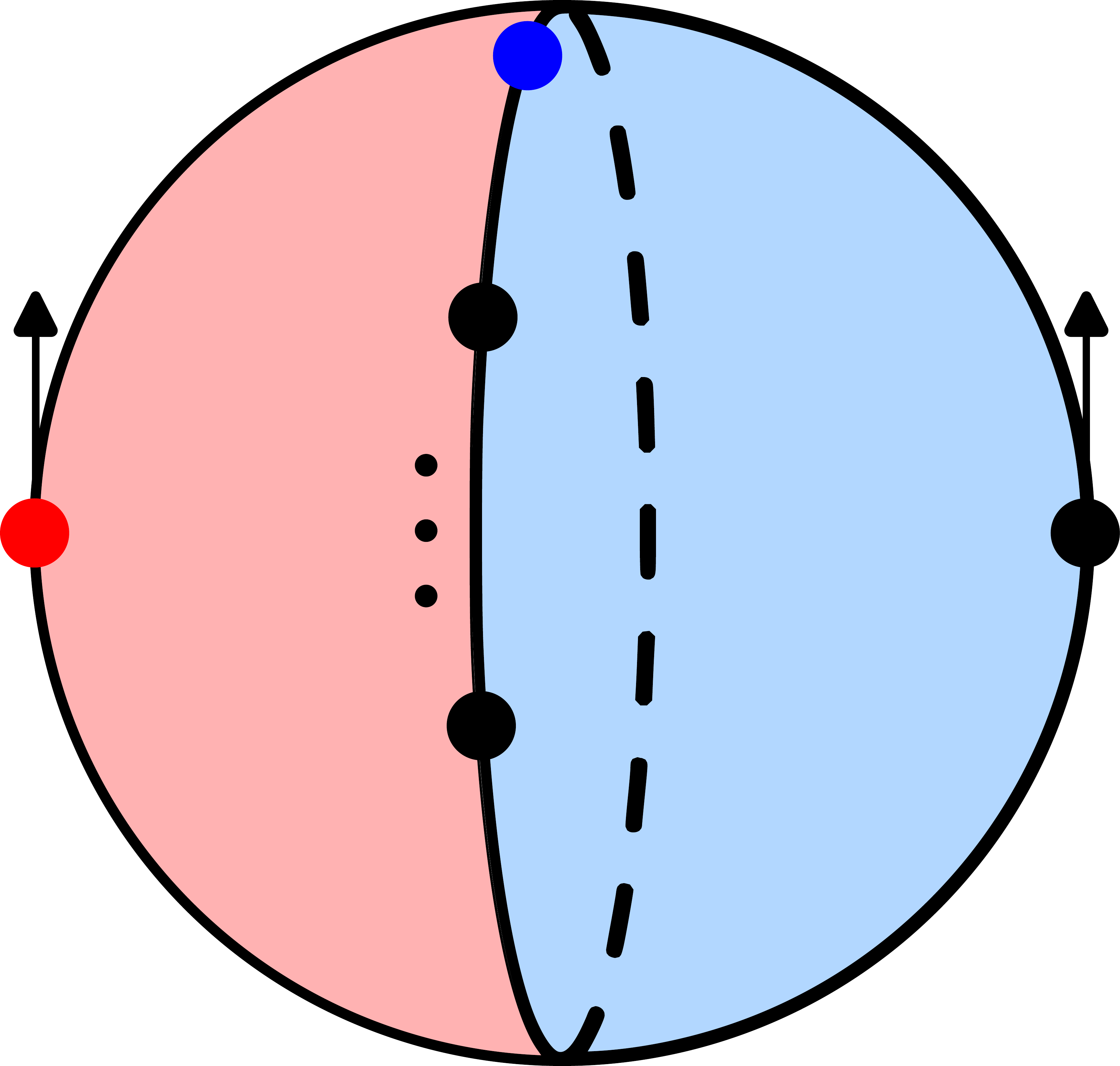}}%
    \put(0.66325768,0.41632785){\makebox(0,0)[lt]{\lineheight{1.25}\smash{\begin{tabular}[t]{l}$M_1$\end{tabular}}}}%
    \put(0.13277953,0.41632785){\makebox(0,0)[lt]{\lineheight{1.25}\smash{\begin{tabular}[t]{l}$M_2$\end{tabular}}}}%
  \end{picture}%
\endgroup%

\caption{}
\end{figure}

\noindent
In this figure, the red marked point is the output, the arrows on the output and the interior input marked points indicate a choice of framing, and the blue marked point is the distinguished element of the Hochschild chain.
The framings and the blue marked point are all aligned.

\begin{expectation}
\label{exp:SBB}
If the unit $\id_{M_2} \in SH^*(M_2)$ lies in the image of the quilted closed-closed map \eqref{eq:QCC}, then 
there is then a quasiisomorphism
\begin{align}
\Fuk M_2
\simeq
\mod\,\bigl(\Phi^\#_{L_{12}^T}\circ\Phi^\#_{L_{12}}\bigr)
\end{align}
\null\hfill$\triangle$
\end{expectation}

\medskip

\noindent
{\bf Heuristic of the proof.}
Hypothesis (BB2) is straightforward: $\Fuk^\#(M_2)$ is a pretriangulated $A_\infty$-category, so it has finite colimits, hence it has coequalizers.
The adjunction $\Phi_{L_{12}^T}^\# \dashv \Phi_{L_{12}}^\#$ implies that $\Phi_{L_{12}^T}^\#$ preserves colimits, hence $\Phi_{L_{12}^T}^\#$ preserves coequalizers.

The rest of our discussion is devoted to hypothesis (BB1).

\medskip

\noindent
{\bf Step 1:}
{\it We explain how the existence of a certain commutative diagram for every $L_2 \in \Fuk M_2$ implies hypothesis (BB1).}

\medskip

\noindent
The desired commutative diagram is analogous to the ``Cardy relation'' in \cite{a:geocrit}, and it has the following form:
\begin{align}
\label{eq:BB_diagram}
\xymatrix{
SH^*(M_1)
\otimes
HH_*(CF^*(L_{12}),CF^*(L_{12}))
\ar[r]
\ar[d]_{QCC}
&
HH_*(CF^*(L_{12}),CF^*(L_2\#L_{12}^T))
\ar[d]
\\
SH^*(M_2)
\ar[r]_{CO}
&
HF^*(L_2).
}
\end{align}
Here we are abbreviating by $CF^*(L)$ the $A_\infty$-algebra associated to $L$, and we are regarding $CF^*(L_2\#L_{12}^T)$ as an $A_\infty$-bimodule over $CF^*(L_{12})$.
By assumption, $QCC$ hits the unit $\id_{M_2} \in SH^*(M_2)$.
Since $CO$ is unital, the composition $CO\circ QCC$ hits the unit $\id_{L_2} \in HF^*(L_2)$.
Commutativity implies that the map
\begin{align}
HH_*(CF^*(L_{12}),CF^*(L_2\#L_{12}^T))
\to
HF^*(L_2)
\end{align}
also hits the unit $\id_{L_2}$.

We now fix a morphism $f \in CF^*(\ul K_2, \ul K_2')$ in $\Fuk^\# M_2$ such that $\Phi_{L_{12}^T}^\#(f)$ is an isomorphism; to establish hypothesis (BB1), we must show that $f$ is an isomorphism.
To simplify the exposition, assume that $\ul K_2 = K_2$ and $\ul K_2' = K_2'$ are Lagrangians, rather than genuinely generalized Lagrangians.
Set $L_2 \coloneqq \cone(f)$ in \eqref{eq:BB_diagram}.
Since $\Phi_{L_{12}^\#}(f)$ is an isomorphism, $\cone\bigl(\Phi_{L_{12}^T}^\#(f)\bigr)$ is the zero object, but
$\Phi_{L_{12}^T}^\#$ commutes with colimits, so $\Phi_{L_{12}^T}^\#\bigl(\cone(f)\bigr) = \cone(f) \# L_{12}^T$ is the zero object.
It follows that the upper-right entry in \eqref{eq:BB_diagram} is zero.
Since the right-hand map in \eqref{eq:BB_diagram} hits the unit, $\id_{L_2}$ must be zero, hence $L_2 = \cone(f)$ is the zero object.
It follows that $f$ is an isomorphism, so $\Phi_{L_{12}^T}^\#$ is indeed conservative.

\medskip

\noindent
{\bf Step 2:}
{\it We construct the commutative diagram used in the previous step.}

\medskip

\noindent
We will define the maps appearing in \eqref{eq:BB_diagram}, as well as the chain homotopy which descends to an isomorphism between the compositions, by counting certain quilts.
As a first step toward defining these quilts, consider the following 1-dimensional family of quilted surfaces:

\begin{figure}[H]
\centering
\def\svgwidth{0.95\columnwidth}
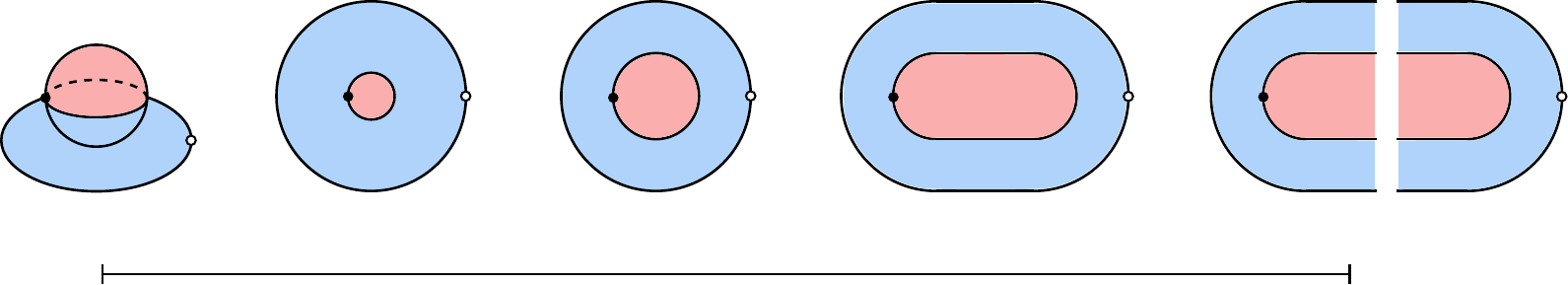
\caption{
\label{fig:bb_family}
}
\end{figure}

\noindent
The surface corresponding to the midpoint of the interval is the radius-2 disk with an interior circle of radius 1, with marked points on the interior and boundary circles that are opposite from one another.
The left half of the interval corresponds to letting the radius of the interior circle vary in the interval $[0,1]$, such that when the radius becomes 0, we bubble off a quilted sphere.
The right half of the interval corresponds to inserting a rectangular portion into the disk; the right endpoint corresponds to (quilted) Floer breaking.
We do not quotient out by any automorphisms.

Next, we construct a higher-dimensional family of quilted surfaces by adding arbitrarily-many marked points to the seam, and adding a framed interior marked point in the inner patch.
We do not allow the positions of the seam marked points to vary, but we do fix the position of the interior marked point, in the way illustrated in the diagram below.

Finally, we consider the pseudoholomorphic quilts whose domains are these quilted surfaces.
The inner patch maps to $M_1$, and the outer patch maps to $M_2$.
We impose a seam condition in $L_{12}$, and a boundary condition in $L_2$.
The framing of the interior puncture always points to the left.
Counting the rigid quilts of this form defines a chain homotopy between the rigid quilts that live over the endpoints of the interval in Figure \ref{fig:bb_family}, thus proving the commutativity of \eqref{eq:BB_diagram}.
\null\hfill``$\square$''

\medskip

To clarify the situation further, we describe the relationship of this result with the split-generation criterion of  \cite{a:geocrit}:
When $M_1 = \pt$, $M_2 = M$, and $L$ is a Lagrangian in $M$, the quilted count in Equation \eqref{eq:QCC} specializes to the open-closed map defined in \cite{a:geocrit}.
Moreover, \eqref{eq:BB_diagram} is the key diagram (1.5) in \cite{a:geocrit}).
It follows that in this case, the hypothesis of Expectation \ref{exp:SBB}  specializes to the hypothesis in the  main result of \cite{a:geocrit} that the open-closed map hits the identity.
However, it turns out that the conclusion that we obtain by applying Barr--Beck is not exactly identical.

To see this, let us specialize to the case that $M$ is closed.
Conflating $\Fuk M$ and $\Fuk^\# M$ and identifying $\Fuk\pt \simeq \perf_\bk$, Barr--Beck says in this case that the following adjunction is monoidal:
\begin{align}
\xymatrix{
\perf_\bk \ar@/^/[r]^{\Phi_{L}} & \Fuk M. \ar@/^/[l]^{\Phi_{L^T}}
}
\end{align}
That is, $\Fuk M$ is equivalent to the Eilenberg--Moore category $\perf_\bk\,\bigl(\Phi_{L^T}\circ\Phi_L\bigr)$ of $CF^*(L)$-modules whose underlying $\bk$-modules have finite rank cohomology (this is also known as the category of \emph{pseudoperfect} modules).

The conclusion of the generation criterion in \cite{a:geocrit} (as applied to closed symplectic manifolds \cite{rs:openclosed}) is instead that $\Fuk M$ is equivalent to the category of \emph{perfect} modules over $CF^*(L)$, i.e.\ those which are built from the free module $CF^*(L)$ using finitely many cones and summands.

For a general $A_\infty$-algebra $\cA$, the categories of pseudoperfect and perfect modules are not equivalent.
Assuming that the cohomology of $\cA$  is finite dimensional (which is automatic for closed Lagrangians), every perfect module is automatically pseudoperfect, and this inclusion is an equivalence if we impose the additional property that $\cA$ is \emph{smooth}, in the sense that the diagonal bimodule admits a split-resolution.

The fact that the hypothesis of Expectation \ref{exp:SBB} implies that the Floer algebra $CF^*(L)$ is smooth (c.f.\ \cite[Theorem 1.2]{ganatra_thesis}) thus relates the generation criterion to Expectation \ref{exp:SBB}.

It is desirable to have an extension of Expectation \ref{exp:SBB} which directly generalises the generation criterion.
Its formulation requires one to introduce the category $\perf\, T$ of \emph{perfect modules over a monad}, i.e.\ the subcategory of $T$-algebras which are obtained from an object of $X$ by taking its image under $T$ (and passing to the closure under cone and summands).
The analogue of Expectation \ref{exp:SBB} is then formulated in terms of the ``quilted open closed map,'' which is the composite
\begin{equation} \label{eq:QOC}
\begin{tikzcd}[row sep = 10]
       HH_*(\Fuk M_1)
\otimes
HH_*(CF^*(L_{12},L_{12}),CF^*(L_{12},L_{12}))   \ar[d] \\ SH^*(M_1) \otimes
HH_*(CF^*(L_{12},L_{12}),CF^*(L_{12},L_{12}))  \ar[d] \\ SH^*(M_2).
\end{tikzcd}
\end{equation}

\begin{expectation}
\label{exp:SBBP}
If the unit $\id_{M_2} \in SH^*(M_2)$ lies in the image of the quilted open-closed map \eqref{eq:QCC}, then the Eilenberg--Moore comparison map lifts to the category of perfect modules, and this lift is an equivalence:
\begin{equation}
\begin{tikzcd}
& \perf\,\bigl(\Phi_{L^T}\circ\Phi_{L}\bigr) \ar[d]
\\
\Fuk M_2 \ar[r] \ar[ur, dashed] & \mod\,\bigl(\Phi_{L^T}\circ\Phi_L\bigr).
\end{tikzcd}
\end{equation}
\null\hfill$\triangle$
\end{expectation}

%%% Local Variables:
%%% mode: latex
%%% TeX-master: "functoriality_in_categorical_symplectic_geometry"
%%% End:

\bibliographystyle{alpha}
\small
\bibliography{biblio}

\end{document}